\documentclass[11pt]{article}
\newcommand{\filename}{LipDomCoSi2017-10-07m.tex}
\newcommand{\comment}[1]{}
\usepackage{makeidx}         
\usepackage{graphicx}        
\usepackage{psfrag}          
\usepackage{multicol}        
\usepackage{amsmath}
\usepackage{amssymb}
\usepackage{amsfonts}
\usepackage{amsthm}
\usepackage{latexsym}
\usepackage[T1]{fontenc}
\usepackage{subfigure}                      
\usepackage{float}
\usepackage{mathrsfs}
\usepackage{color}

\allowdisplaybreaks
\newtheorem{lemma}{LEMMA}[section]
\newtheorem{theorem}[lemma]{THEOREM}
\newtheorem{definition}[lemma]{DEFINITION}
\newtheorem{corollary}[lemma]{COROLLARY}

\newtheorem{remark}[lemma]{REMARK}

\def\supp{\mathop{\mbox{\rm supp}}\nolimits}

 \newcommand{\nc}{\newcommand}
 \nc{\ha}{\frac{1}{2}}
 \nc{\tha}{\frac{3}{2}}
 \nc{\s}{\widetilde}
 \nc{\dst}{\displaystyle}
 \nc{\gm}{\gamma}
 \nc{\ga}{\Gamma}
 \nc{\ka}{\kappa}
 \nc{\eps}{\varepsilon}
 \nc{\vep}{\varepsilon}
 \nc{\hi}{\varphi}
 \nc{\vfi}{\varphi}
 \nc{\oa}{\Omega}
 \nc{\Om}{\Omega}
 \nc{\om}{\omega}
 \nc{\ov}{\overline}
 \nc{\lon}{\longrightarrow}
 \nc{\scr}{\scriptstyle}
 \nc{\ex}{\exists}
 \nc{\fo}{\forall}
 \nc{\pa}{\partial}
 \nc{\pO}{{\partial\Omega}}
 \nc{\und}{\underline}
 \nc{\ze}{\zeta}
 \nc{\si}{\sigma}
 \nc{\tri}{\triangle}
 \nc{\al}{\alpha}
 \nc{\bt}{\beta}
 \nc{\lf}{\left}
 \nc{\ri}{\right}
 \nc{\lm}{\lambda}
 \nc{\lam}{\lambda}
 \nc{\de}{\delta}
 \nc{\te}{\theta}
 \nc{\tl}{\tilde}
 \nc{\wt}{\widetilde}
 \nc{\p}{\prime}
 \nc{\m}{\mu}
 \nc{\R}{{\mathbb R}}
 \nc{\B}{{\Bbb B}}
 \nc{\N}{{\Bbb N}}
 \nc{\C}{{\Bbb C}}

\newcommand{\E}{{\cal E}}
\newcommand{\F}{{\cal F}}

\renewcommand{\L}{{\cal L}}
\renewcommand{\P}{{\cal P}}

\newcommand{\U}{{\cal U}}

\newcommand{\V}{{\cal V}}

\newcommand{\W}{{\cal W}}
\newcommand{\Wp}{{\W\,^\prime}}
\newcommand{\be}{\begin{equation}}
\newcommand{\ee}{\end{equation}}
\newcommand{\bes}{\begin{equation*}}
\newcommand{\ees}{\end{equation*}}
\newcommand{\bea}{\begin{eqnarray}}
\newcommand{\eea}{\end{eqnarray}}
\newcommand{\beas}{\begin{eqnarray*}}
\newcommand{\eeas}{\end{eqnarray*}}

\def\text#1{\;\;\mbox{#1}\;}

\vspace{1cm}
 \makeatletter
\renewcommand{\@oddhead}{\vbox{\hbox to\textwidth{\scriptsize %
 \filename\hfill S.E.Mikhailov\hfill 
 \hfill\arabic{page}}
 \vspace{0.25ex}
 
 \hrule
 }}
 \renewcommand{\@evenhead}{\vbox{\hbox to\textwidth{\scriptsize %
 \filename\hfill 
 }
  \vspace{0.25ex}
  
\hrule
}}
 \makeatother
 \allowdisplaybreaks
 \numberwithin{equation}{section}

\begin{document}
\title
{Analysis of Segregated
Boundary-Domain Integral Equations for
 BVPs with Non-smooth Coefficient on Lipschitz Domains\\
}


\author
{S.E. Mikhailov
\footnote{%
e-mail: {\sf sergey.mikhailov@brunel.ac.uk}}\\ 
Department of Mathematics, Brunel University London, UK
     \\
}
 \date{\em \today
 \hfill \filename }

 \maketitle

\maketitle

\begin{abstract}\noindent

Segregated direct
boundary-domain integral equations (BDIEs) based on a  parametrix and
associated with the Dirichlet and Neumann boundary value problems for the linear stationary diffusion
partial differential equation with a variable non-smooth (or limited-smoothness) coefficient on Lipschitz domains are formulated. 
The PDE right hand sides belong to the Sobolev (Bessel-potential) space $H^{s-2}(\Omega)$ or $\widetilde H^{s-2}(\Omega)$, $\ha<s<\tha$, when neither strong classical nor weak canonical co-normal derivatives are well defined.
Equivalence of the BDIEs to the original BVP, BDIE solvability, solution uniqueness/non-uniqueness,
as well as Fredholm property and invertibility of the BDIE operators  are
analysed in appropriate Sobolev spaces.
It is shown that the BDIE operators for the Neumann BVP are not invertible, however
some finite-dimensional perturbations are constructed leading to invertibility of the perturbed (stabilised) operators.\\


\end{abstract}

\section{Introduction}

Many applications in science and engineering can be modeled by
boundary-value problems (BVPs) for partial differential equations with variable coefficients.
Reduction of the BVPs with arbitrarily variable coefficients to explicit
boundary integral equations is usually not possible, since the
fundamental solution needed for such reduction is generally not
available in an analytical form (except for some special
dependence of the coefficients on coordinates). Using a parametrix
(Levi function) introduced in \cite{Levi1909}, \cite{Hilbert1912}
as a substitute of a fundamental solution, it is possible however
to reduce such a BVP to a system of boundary-domain integral equations, BDIEs, (see
e.g.  \cite[Sect. 18]{Miranda1970}, \cite{Pomp1998b, Pomp1998a}, where the
Dirichlet, Neumann and Robin problems for some PDEs were reduced
to {\em indirect} BDIEs). However, many questions about their
equivalence to the original BVP, solvability, solution uniqueness
and invertibility of corresponding integral operators remained
open for rather long time.

In \cite{CMN-1, CMN-2, MikMMAS2006, CMN-NMPDE-crack, CMN-Ext-AA2013}, the 3D mixed (Dirichlet-Neumann)
boundary value problem (BVP) for the 
stationary diffusion PDE {\em with infinitely smooth variable coefficient on a domain with an infinitely smooth boundary and a square-integrable right-hand side} was
reduced to either segregated
or united direct Boundary-Domain Integral or Integro-Differential
Equations,  some of the which coincide with
those formulated in \cite{MikEABEM2002}. 
Such BVPs appear e.g. in electrostatics,
stationary heat transfer and other diffusion problems for
inhomogeneous media.

For a function from the Sobolev space $H^s(\Omega)$,  $\ha<s<\tha$, a classical co-normal derivative
in the sense of traces may not exist. 
However, when this function satisfies a second order partial differential equation with a right-hand side from $H^{s-2}(\Omega)$,  the generalised co-normal derivative can be defined in the weak sense, associated with the first Green identity and with an extension of the PDE right hand side to $\widetilde H^{s-2}(\Omega)$ (see \cite[Lemma 4.3]{McLean2000}, \cite[Definition 3.1]{MikJMAA2011}).
Since the extension is non-unique, the co-normal derivative operator appears to be also non-unique and non-linear in $u$ unless a linear relation between
$u$ and the PDE right hand side extension is enforced. This
creates some difficulties in formulating the
boundary-domain integral equations. 

These difficulties are addressed in this paper
presenting formulation and analysis of direct
segregated BDIE systems equivalent to the Dirichlet and Neumann boundary value
problems, on Lipschitz domains, for the divergent-type PDE with a non-smooth H\"older-Lipschitz variable scalar coefficient  and a general
right hand side from ${H}^{s-2}(\Omega)$, extended when necessary to  $\widetilde{H}^{s-2}(\Omega)$. 
This needed a non-trivial generalisation of the third Green identity and its co-normal derivative for such functions,
which essentially extends the approach implemented in \cite{CMN-1, CMN-2, MikMMAS2006, CMN-NMPDE-crack, CMN-Ext-AA2013} for the right hand side from $L_2(\Omega)$, with smooth coefficient and domain boundary. 
Equivalence of the BDIEs to the original BVP, BDIE solvability, solution uniqueness/non-uniqueness,
as well as Fredholm properties and invertibility of the BDIE operators  are
analysed in the Sobolev (Bessel potential) spaces.
It is shown that the BDIE operators for the Neumann BVP are not invertible, and
appropriate finite-dimensional perturbations are constructed leading to invertibility of the perturbed (stabilised) operators.
Some preliminary results in this direction, for the infinitely smooth coefficient and domains, were presented in \cite{MikArxiv2015}.

Note that our analysis is mainly aimed not at the boundary-value problems, which properties are well-known nowadays, but rather at the BDIE systems per se. The analysis is interesting not only in its own rights but is also to be used further on  for analysis of convergence and stability of BDIE-based numerical methods for PDEs,
 see e.g. \cite{GMR2013, MikEABEM2002, MikNakJEM, MikMoh2012, SSA2000, SSZ2005, Taigbenu1999, ZZA1998, ZZA1999}.

\section{Spaces, co-normal derivatives and boundary value problems}\label{S2}
Let $\Omega=\Omega_+$ be a bounded open $n$--dimensional region of
$\R^n$, $n\ge 3$, and $\Omega_-=\R^n\setminus\overline\Omega_+$ is the corresponding exterior domain. 
For simplicity, we assume that their common boundary $\pa \Omega$
is a simply connected, closed, Lipschitz surface. Let $\Omega_0$  denote  $\Omega_+$, $\Omega_-$ or $\R^n$. 

In what follows $\mathcal D(\Omega_0):=C^\infty_{comp}(\Omega_0)$, 
$\mathcal D(\overline{\Omega_0}):=\{r_{_{\Omega_0}}g:\;g\in\mathcal D(\R^n)\}$.
Here and further on, $r_{_{\Omega_0}}$ denotes the restriction operator on
$\Omega_0$; we will also use the equivalent notation
$g|_{_{\Omega_0}}:=r_{_{\Omega_0}}g$. 
Further, $ H^s(\Omega_0)= H^s_2(\Omega_0)$, $ H^s(\pO)=H^s_2(\pO)$ are the
Bessel potential spaces, where $s$ is an arbitrary real
number (see, e.g., \cite{LiMa1}, \cite{McLean2000}). We recall that $H^s$
coincide with the Sobolev--Slobodetski spaces $W^s_2$ for any
non-negative $s$.
By $\widetilde{H}^s (\Omega_0)$  we denote the closure of  $\mathcal D(\Omega_0)$ in $H^s(\R^n)$. It is a subspace of ${H}^s (\R^n)$, and for Lipschitz domains,
$
\widetilde{H}^s (\Omega_0)=\{g:\;g\in H^s  (\R^n),\; \supp \,g
\subset\ov{\Omega_0}\}.
$
By ${H}^s (\Omega_0)$ and 
${\widetilde H}^s_\bullet(\Omega_0)$
we denote the spaces of restrictions on
$\Omega_0$ of distributions  from ${H}^s(\R^n)$ and $\widetilde{H}^s(\Omega_0)$, respectively,
\begin{align}
H^s (\Omega_0):=\{r_{_{\Omega_0}}g:\;g\in{H}^s (\R^n)\}, \quad 
\widetilde H^s_\bullet(\Omega_0):=r_{\Omega_0}{\widetilde H}^s(\Omega_0):=\{r_{_{\Omega_0}}g:\;g\in\widetilde{H}^s (\Omega_0)\} \subset H^s (\Omega_0),
\end{align}
endowed by the corresponding infimum norms and the Hilbert structure defined with the help of orthogonal projections, cf. \cite[p. 77]{McLean2000} for $H^s(\Omega_0)$.  Note that the space $\widetilde H^s_\bullet(\Omega_0)$ coincides with the one denoted as $L_{s,z}^p(\Omega_0)$ in \cite[Eq. (5.2)]{Mitrea-Monniaux2008}, \cite[Eq. (2.212)]{MM2013Spr}, for $p=2$.

We denote by ${H}^s_{\pO}$ the following subspace of ${H}^s
(\R^n)$ (and $\widetilde{H}^s (\Omega_0)$),
\be\label{H_dO}
{H}^s_{\pO}:=\{g:\;g\in H^s(\R^n),\; \supp \,g \subset{\pO}\},
\ee
and by $\mathring H^s(\Omega_0)$ we denote the closure of  $\mathcal D(\Omega_0)$ in $H^s(\Omega_0)$.


\begin{definition}\label{E0}Let us denote  by $\mathring E_{\Omega_0}$ the operator of extension of functions $g\in H^s(\Omega_0)$, $s\ge 0$, to the whole $\R^n$  by zero outside $\Omega_0$. 
By, e.g., \cite[Lemma 3.32 and Theorem 3.33]{McLean2000}, see also \cite[Theorem 2.7]{MikJMAA2011}, the operator $\mathring E_{\Omega_0}: H^s(\Omega_0)\to\widetilde H^s(\Omega_0)$ is continuous if $0\le s<\ha$ and we will extend it also to the range $-\ha< s<\ha$ defining it for $-\ha<s<0$ as, cf. the proof of \cite[Theorem 2.16]{MikJMAA2011},
\begin{align}\label{E0s}
\langle \mathring E_{\Omega_0} g,v\rangle_{\Omega_0}:=\langle  g,\mathring E_{\Omega_0} v\rangle_{\Omega_0},\quad
\forall g\in H^s(\Omega),\ \forall v\in H^{-s}(\Omega).
\end{align}
\end{definition}
\begin{remark}\label{R2.1}
Note the following known or easily deduced facts.
\begin{enumerate}
\item\label{0} There hold the continuous embeddings 
$\widetilde H^s_\bullet(\Omega_0)\hookrightarrow \mathring H^s(\Omega_0)\hookrightarrow H^s(\Omega_0)$, cf. \cite[Eq. (2.123)]{Mitrea-Wright2012}. 

\item\label{i} $\widetilde H^s_\bullet(\Omega_0)= \mathring H^s(\Omega_0)$ for any $s>1/2$ such that $s-\frac{1}{2}$ is non-integer, by e.g. \cite[Theorem 3.3]{McLean2000}.  

\item\label{ii} $\mathring H^s(\Omega_0)=H^s(\Omega_0)$ for any $s\le 1/2$, by \cite[Theorem 2.12]{MikJMAA2011}.  

\item\label{iii} $\widetilde H^s_\bullet(\Omega_0)=\mathring H^s(\Omega_0)=H^s(\Omega_0)$ for any $s< 1/2$ such that $s-\frac{1}{2}$ is non-integer, by e.g. \cite[Lemma 2.15]{MikJMAA2011}. 

\item\label{iv} for any $s\in\R$, there evidently exist an extension from  $\widetilde H^s_\bullet(\Omega_0)$ to  $\widetilde H^s(\Omega_0)$, and for any $s\ge-1/2$ this extension is {\bf unique}, by e.g. \cite[Lemma 3.39]{McLean2000}, \cite[Theorem 2.10(i)]{MikJMAA2011}. 

\item\label{v} By \cite[Theorem 2.16]{MikJMAA2011}, for any $s\in(-1/2,1/2)$ the extension  from
$\widetilde H^s_\bullet(\Omega_0)=\mathring H^s(\Omega_0)=H^s(\Omega_0)$ to $\widetilde H^s(\Omega_0)$ and is given by the operator $\mathring E_{\Omega_0}$.
\end{enumerate}
\end{remark}
\begin{remark}\label{R2.2}
Due to Remark~\ref{R2.1}(\ref{iv}),  for $s\ge -1/2$ the space $\widetilde H^s_\bullet(\Omega_0)$ is isometrically isomorphic to the space $\widetilde H^s (\Omega_0)$ and sometimes these spaces are identified. 
Particularly, if $g_1,g_2\in \widetilde H^s_\bullet(\Omega_0)$, then denoting by $\tilde g_1,\tilde g_2\in \widetilde H^s (\Omega_0)$ the {\bf unique} distributions such that $g_i=r_{\Omega_0} \tilde g_i$ in $\Omega_0$, we have 
$\|g_i\|_{\widetilde H^s_\bullet(\Omega_0)}=\|\tilde g_i\|_{\widetilde H^s(\Omega_0)}$ and
$(g_1,g_2)_{\widetilde H^s_\bullet(\Omega_0)}=(\tilde g_1,\tilde g_2)_{\widetilde H^{s}(\Omega_0)}$.
Moreover, if $s\in(-1/2,1/2)$, then by Remark~\ref{R2.1}(\ref{v}), $\tilde g_i=\widetilde E_{\Omega_0}^s g_i$ hence implying
$\|g_i\|_{\widetilde H^s_\bullet(\Omega_0)}=\|\widetilde E_{\Omega_0}^s g_i\|_{\widetilde H^s(\Omega_0)}$. 

There is no such isomorphism for $s<-1/2$ since  the extension from $\widetilde H^s_\bullet(\Omega_0)$ to $\widetilde H^s(\Omega_0)$ is not unique then. But due to the definition of the spaces, 
there is still an isometric isomorphism between the space $\widetilde H^s_\bullet(\Omega_0)$ and the quotient space  $\widetilde H^s(\Omega_0)/H^s_{\partial\Omega_0}$. 
\end{remark}

Definition of the space $\widetilde H^s_\bullet(\Omega_0)$, Remark~\ref{R2.1} and Remark~\ref{R2.2} imply the following assertion.
\begin{corollary}\label{C2.3} 
The following restriction operators are isomorphisms,
\begin{align}
\label{r1}
&r_{\Omega_0}:\widetilde H^s(\Omega_0)\to \widetilde H^s_\bullet(\Omega_0),\quad  -\frac{1}{2}\le s,\\
\label{r2}
&r_{\Omega_0}:\widetilde H^s(\Omega_0)\to H^s(\Omega_0)=\widetilde H^s_\bullet(\Omega_0),\quad  -\frac{1}{2}<s<\frac{1}{2},\\
\label{r3}
&r_{\Omega_0}:\widetilde H^s(\Omega_0)/H^s_{\partial\Omega_0}\to \widetilde H^s_\bullet(\Omega_0),\quad s<-\frac{1}{2}.
\end{align}
The inverse to the operator \eqref{r2} is $r_{\Omega_0}^{-1}=\mathring E_{\Omega_0}$, see Definition \ref{E0}.
\end{corollary}

\begin{definition}\label{C+mu}
Let for a non-negative integer $m$ and $0<\theta\le 1$, $C^{m,\theta}(\overline\Omega_0)$ denote the H\"older-Lipschitz space in the closed domain $\overline\Omega_0$.
Similar to \cite[Definition 3.1]{MikJMAA2013},  $g\in C^\mu_+(\overline\Omega_0)$ for $\mu\ge 0$ will mean that\\
$g\in L_\infty(\Omega_0)$, when $\mu=0$;\\ 
$g\in C^{\mu-1,1}(\overline\Omega_0)$, when $\mu$ is a positive integer;\\
$g\in C^{m,\theta+\epsilon}(\overline\Omega_0)$ for some $\epsilon>0$, when $\mu=m+\theta$, where $m$ is a non-negative integer and $0<\theta<1$.
\end{definition}

Employing this definition, Theorem~\ref{GrL} 
from Section \ref{Appendix} 
can be reformulated as follows.
\begin{theorem}\label{GrL+}
Let $\Omega_0$ be an open set in $\R^n$, $\sigma\in\R$, $v\in H^\sigma(\Omega_0)$ and $g\in C^{|\sigma|}_+(\overline\Omega_0)$. Then 
$g$ is a multiplier in $H^\sigma(\Omega_0)$, i.e.,   $gv\in H^\sigma(\Omega_0)$
for every $v\in H^\sigma(\Omega_0)$, and the corresponding norm estimate holds.
\end{theorem}


Let us denote $\pa_{j}:=\pa_{x_j}:=\pa/\pa{x_j}$ $(j=1,2,...,n)$, $\nabla=(\pa_{1},\pa_{2},..., \pa_{n})$.
Let $1/2<s<3/2$ and  $a\in C_+^{|s-1|}(\overline \Omega_\pm)$, and 
\begin{align}\label{a>0}
0<a_{\min}\le a(x)\le a_{\max}<\infty \mbox{ for almost every } x\in{\Omega_\pm}.
\end{align} 
We consider the scalar elliptic differential equation, which for
sufficiently smooth $u$ has the following strong form,
\begin{equation}
\label{2.1} Au(x):=A(x,\nabla)\,u(x) := \nabla\cdot \left(a(x)\,\nabla u(x)\,\right) = f(x),  \quad x \in \Omega_\pm,
\end{equation}
where
 $u$ is an unknown
function and $f$ is a given function in $\Omega_\pm$.

For $u\in H^s(\Omega_\pm)$, the partial differential operator $A$ is understood in the sense of distributions,
\begin{equation}\label{Ldist}
    \langle Au,v \rangle_{\Omega_\pm}:=-\E_{\Omega_\pm}(u,v)\quad \forall v\in
    \mathcal D(\Omega_\pm),
\end{equation}
where 
$$
\E_{\Omega_\pm}(u,v):=\left\langle a\nabla u , \nabla v\right\rangle_{\Omega_\pm}
:=\sum\limits_{i=1}^{n} \left\langle a\pa_i u , \pa_i v\right\rangle_{\Omega_\pm},
$$
and the duality brackets $\langle \;g,\;\cdot\;\rangle_{\Omega_\pm} $ denote
value of a linear functional (distribution) $g$, extending the
usual $L_2$ dual product. If $s=1$, then 
 $$
\E_{\Omega_\pm}(u,v)=\int_{\Omega_\pm} a(x) \;\nabla u(x) \cdot\nabla v(x)dx.
 $$
Since the set $\mathcal D(\Omega_\pm)$ is dense in $\widetilde{H}^{2-s}(\Omega_\pm)$, \eqref{Ldist} defines, due to Theorem~\ref{GrL+} (cf. e.g. \cite[Theorem 3.4]{MikJMAA2013}), the continuous linear operator 
$A: H^s(\Omega_\pm)\to H^{s-2}(\Omega_\pm)=[\widetilde{H}^{2-s} (\Omega_\pm)]^*$, where
\begin{equation}\label{LH1}
    \langle Au,v \rangle_{\Omega_\pm}:=-\E_{\Omega_\pm}(u,v),\quad \forall\ u\in H^s(\Omega_\pm),\, v\in\widetilde{H}^{2-s}(\Omega_\pm).
\end{equation}

Let us consider also the different operators 
$\check{A}_{\Omega_\pm}:H^s({\Omega_\pm})\to \widetilde{H}^{s-2}({\Omega_\pm})=[{H}^{2-s} ({\Omega_\pm})]^*$, 
see \cite[Eq. (3.5)]{MikJMAA2011}, \cite[Eq. (5.1)]{MikJMAA2013},
\begin{multline}\label{Ltil}
    \langle \check{A}_{\Omega_\pm} u,v \rangle_{{\Omega_\pm}}:=-\check{\E}_{\Omega_\pm}(u,v)
    :=-\langle \mathring E_{\Omega_\pm}(a\nabla u),\nabla v\rangle_{{\Omega_\pm}}
        =-\langle \mathring E_{\Omega_\pm}(a\nabla u),\nabla v_e\rangle_{\R^n}\\
        =\langle \nabla\cdot\mathring E_{\Omega_\pm}(a\nabla u),v_e\rangle_{\R^n}
        =\langle \nabla\cdot\mathring E_{\Omega_\pm}(a\nabla u),v\rangle_{{\Omega_\pm}}, \quad 
        \forall\, u\in H^s({\Omega_\pm}),\  v\in {H}^{2-s}({\Omega_\pm}),
\end{multline}
which is evidently continuous.  
Here $v_e\in H^{2-s}(\R^n)$ is such that $r_{{\Omega_\pm}} v_e=v$. 
Evidently, weak definition \eqref{Ltil} can be also written (in the strong-looking form) as 
\be\label{checkA}
\check{A}_{\Omega_\pm} u=\nabla\cdot\mathring E_{\Omega_\pm} r_{\Omega_\pm}[a\nabla u].
\ee
For any $u\in H^s({\Omega_\pm})$, the
functional $\check{A}_{\Omega_\pm} u$ belongs to $\widetilde{H}^{s-2}({\Omega_\pm})$ and is a specific
extension of the functional ${A}u\in {H}^{s-2}({\Omega_\pm})$; recall that the functional ${A}u\in {H}^{s-2}({\Omega_\pm})$ is defined on 
$\widetilde{H}^{2-s}({\Omega_\pm})$, while the functional $\check{A}_{\Omega_\pm} u$ on ${H}^{2-s}({\Omega_\pm})$.

\begin{remark}\label{RE}
Note also that  Definition \ref{E0} for $\mathring E_{\Omega_\pm}$ and definition \eqref{Ltil} imply that 
$$
\langle \check{A}_{\Omega_\pm} u,v \rangle_{{\Omega_\pm}}=-\check{\E}_{\Omega_\pm}(u,v)=-\check{\E}_{\Omega_\pm}(v,u)=\langle u,\check{A}_{\Omega_\pm} v \rangle_{{\Omega_\pm}},\quad \forall\ u\in H^s({\Omega_\pm}),\ v\in {H}^{2-s}({\Omega_\pm}),\ 1/2<s<3/2.
$$
\end{remark}

From the trace theorem (see e.g. \cite{LiMa1, Costabel1988, DaLi4, McLean2000})
for $u\in H^s(\Omega_\pm)$, $1/2<s<3/2$, it follows that
 $\gamma^\pm\,u \in H^{s-\ha}(\pO)$, where
$\gamma^\pm=\gamma^\pm_{_\pO}$ is the trace operator on $\pO$ from $\Omega_\pm$.
If $\gamma^+u=\gamma^-u$, we will sometimes write $\gamma u$.
Let also $\gamma^{-1}:=\gamma_r^{-1}:H^{s-\ha}(\partial\Omega)\to H^s(\R^n)$ denote a (non-unique) continuous right inverse to the trace operator $\gamma$, i.e., $\gamma\gamma^{-1}w=w$ for any $w\in H^{s-\ha}(\partial\Omega)$. 
Hence also $\gamma^\pm\gamma^{-1}w=w$ for any $w\in H^{s-\ha}(\partial\Omega)$.

For $u\in H^s(\Omega_\pm)$, $s>\tha$, and $a\in C(\overline\Omega_\pm)$,  we can denote by $T^{c\pm}$ the corresponding classical (strong)
co-normal derivative operators on $\pO$ in the sense of
traces,
\begin{align}\label{Tcl}
 T^{c\pm}u(x) := 
a(x)\nu(x)\cdot\gamma^\pm\nabla u(x)
=a(x)\,\pa_\nu u(x),\quad x\in\partial\Omega,
\end{align}
where $\nu(x)=\nu^+(x)$ is the outward to $\Omega_+$ unit normal vector at
the point $x\in \pO$, and we will sometimes write $T^{c}u(x)$ if $T^{c+}u(x)=T^{c-}u(x)$.  
However the classical co-normal derivative is, generally, not well defined if 
$u\in H^s(\Omega_\pm )$, $1/2<s<3/2$, (cf.  an example in \cite[Appendix A]{MikArxiv2015} of a function from $H^1(\Omega)$, where the classical normal derivative does not exist at any boundary point).

Inspired by the first Green identity for smooth functions, we can define {\em the
generalised co--normal derivative} (cf., for
example, \cite[Lemma 4.3]{McLean2000}), \cite[Definition 3.1]{MikJMAA2011}, \cite[Definition 5.2]{MikJMAA2013}).
\begin{definition}\label{GCDd}
Let $1/2<s<3/2$, $u\in H^{s}(\Omega_\pm)$, $a\in C_+^{|s-1|}(\overline{\Omega}_\pm)$, and $r_{\Omega_\pm}Au=r_{\Omega_\pm}\tilde f_\pm$  for some $\tilde f_\pm\in\widetilde{H}^{s-2}(\Omega_\pm)$. 
Then the {\bf generalised co--normal derivatives} $T^\pm(\tilde f_\pm;u) \in H^{s-\tha}(\pO)$  are defined in the weak form as
\begin{align}
\label{Tgend} 
\pm\left\langle T^\pm(\tilde f_\pm;u)\,,\, w\right\rangle _{\pO}:=
 \langle \tilde f_\pm,\gamma^{-1}w \rangle_{\Omega_\pm} + \check\E_{\Omega_\pm}(u,\gamma^{-1}w)
 =\langle\tilde f_\pm-\check{A}_{\Omega_\pm}u,\gamma^{-1}w\rangle_{\Omega_\pm},
  \quad  
\forall\ w\in H^{\tha-s} (\partial\Omega),
\end{align}
i.e.,
$
 T^\pm(\tilde f_\pm,u):=\pm(\gamma^{-1})^*(\tilde f_\pm-\check{A}_{\Omega_\pm}u).
$

If $a\equiv 1$, then $A=\Delta$ and $T^\pm(\tilde f_\pm;u)$ become {\bf generalised normal derivatives} denoted as  $T^\pm_\Delta(\tilde f_\pm;u)$. 
\end{definition}
\begin{theorem}[Lemma 4.3 in \cite{McLean2000}), Theorem 3.2 in \cite{MikJMAA2011}, and Theorem 5.3 \cite{MikJMAA2013}]\label{T6.6} 
Under the hypotheses of Definition~\ref{GCDd}, the generalised co-normal derivatives $T^\pm u(\tilde f_\pm;u)$ are
independent of (non-unique) choice of the operator $\gamma^{-1}$,  
we have the estimate
\begin{equation}\label{estimate}
\|T^\pm(\tilde{f}_\pm;u)\|_{H^{s-\tha}(\pO)}\le
C_1\|u\|_{H^s({\Omega_\pm})} + C_2\|\tilde{f}_\pm\|_{\widetilde{H}^{s-2}(\Omega_\pm)},
\end{equation}
and the first Green identity holds in the following form for $u\in {H}^{s}(\Omega_\pm)$ such that $r_{\Omega_\pm} Au=r_{\Omega_\pm}\tilde f_\pm$ for some 
$\tilde f_\pm\in\widetilde{H}^{s-2}(\Omega_\pm)$,
\begin{equation}
\label{Tgen} 
\pm\left\langle T^\pm(\tilde f_\pm;u)\,,\, \gamma^\pm v
\right\rangle _{\pO}=\langle \tilde f_\pm,v \rangle_{\Omega_\pm} + \check\E_{\Omega_\pm}(u,v)=\langle
\tilde f_\pm-\check{A}_{\Omega_\pm}u,v \rangle_{\Omega_\pm}
  \quad  
\forall\ v\in H^{2-s} (\Omega_\pm).
\end{equation}
\end{theorem}
As follows from Definition~\ref{GCDd}, the generalized co-normal derivative is nonlinear with respect to $u$ for  fixed $\tilde{f}_\pm$, but still linear with respect to the couple $(\tilde f_\pm, u)$, i.e., for any complex numbers $\alpha_1$ and $\alpha_2$,
\begin{align}\label{GCDL}
\alpha_1 T^\pm(\tilde{f}_{1\pm};u_1) +\alpha_2 T^+(\tilde{f}_{2\pm};u_2)=
T^\pm(\alpha_1\tilde{f}_{1\pm};\alpha_1u_1) + T^\pm(\alpha_2\tilde{f}_{2\pm};\alpha_2u_2)=
T^\pm(\alpha_1\tilde{f}_{1\pm}+\alpha_2\tilde{f}_{2\pm};\alpha_1u_1+\alpha_2u_2).
\end{align}

Let us also define some subspaces of $H^s({\Omega_\pm})$, cf. \cite{Grisvard1985, Costabel1988, MikJMAA2011, MikJMAA2013}.
\begin{definition}\label{Hst}
Let $s\in\mathbb{R}$ and $A_*:H^s({\Omega_\pm})\to {\cal D}^*({\Omega_\pm})$
be a linear operator. For $t\in \R$, we introduce the space
 $$
 H^{s,t} ({\Omega_\pm};A_*):=\{g:\;g\in H^s ({\Omega_\pm}),\
 A_*g\in
\widetilde H^t_\bullet({\Omega_\pm})\}
 $$
endowed with the norm
 $ 
 \|g\|_{ H^{s,t} ({\Omega_\pm};A_*)}:= \left(\|g\|^2_{ H^s({\Omega_\pm})}+\|A_*g\|^2_{\widetilde H^t_\bullet({\Omega_\pm})}\right)^{1/2}
 $
 and the corresponding inner product. 
 \end{definition}
 
 \begin{definition} Let $\Omega_0$ be either $\Omega_+$ or $\Omega_-$. 
By Remark~\ref{R2.2}, if $g\in H^{s,t} ({\Omega_0};A_*)$ for some $s\in\mathbb{R}$ and $t\ge -\frac{1}{2}$, then
there exists a {\bf unique} distribution $\tilde f\in \widetilde{H}^{t}({\Omega_0})$ such that $r_{\Omega_0}\tilde f=A_*g$, and hence, $\tilde f=\tilde A_{*{\Omega_0}}g$, where $\tilde A_{*{\Omega_0}}:= r_{\Omega_0}^{-1}A_*$. The operator 
$\tilde A_{*{\Omega_0}}:H^{s,t} ({\Omega_0};A_*)\to \widetilde H^t({\Omega_0})$ is continuous by Corollary~\ref{C2.3}, is the {\bf canonical} extension of the operator $A_*:H^{s,t} ({\Omega_0};A_*)\to \widetilde H^t_\bullet({\Omega_0})$ and
moreover, if $-\ha<t<\ha$, then $\tilde A_{*{\Omega_0}}=\mathring E_{\Omega_0} A_*$.
\end{definition}
We will mostly use the operators $A$ or $\Delta$ as $A_*$ in the above definition. 
Note that since
$
Au =a\Delta u +\nabla a\cdot\nabla u 
$,
then for $1/2<s<3/2$,   we have
 $H^{s,-\ha}({\Omega_0};A)= H^{s,-\ha} ({\Omega_0};\Delta)$ if $a\in C_+^{\tha}(\overline{{\Omega_0}})$, with equivalent norms.
 
   Let us now define the {\em canonical} conormal derivative, cf. \cite[Definition 6.5]{MikJMAA2013}.
\begin{definition}\label{Dccd}
For $u\in H^{s,-\frac{1}{2}}(\Omega_\pm;A)$ and $a\in C_+^{|s-1|}(\overline\Omega_\pm)$,  $1/2<s<3/2$,  we define the
{\bf canonical  co-normal derivatives} $T^\pm u \in
H^{s-\frac{3}{2}}(\partial\Omega)$ as
\begin{multline}\label{Tcandef}
 \pm\left\langle  T^\pm u\,,\, w \right\rangle _{\partial\Omega}
 :=\langle \tilde{A}_{\Omega_\pm} u,\gamma^{-1}w \rangle_{\Omega_\pm} +
\check\E_{\Omega_\pm}(u,\gamma^{-1}w)
=\langle \tilde{A}_{\Omega_\pm} u-\check{A}_{\Omega_\pm}u,\gamma^{-1}w \rangle_{\Omega_\pm}\\
=\langle(\gamma^{-1})^*(\tilde A_{\Omega_\pm}u-\check{A}_{\Omega_\pm}u),w \rangle_{\partial\Omega}
  \quad
\forall\ w\in H^{\frac{3}{2}-s} (\partial\Omega),
\end{multline}
i.e,
$
 T^\pm u:=\pm (\gamma^{-1})^*(\tilde A_{\Omega_\pm}u-\check{A}_{\Omega_\pm}u).
$

If $a\equiv 1$, $T^\pm u$ becomes {\bf canonical normal derivative} denoted as  $T^\pm_\Delta u$. 
\end{definition}

\begin{theorem}[Theorem 3.9 in \cite{MikJMAA2011} and Theorem 6.6 in \cite{MikJMAA2013}]\label{T6.6c} 
Under the hypotheses of Definition~\ref{Dccd}, the canonical co-normal derivatives $T^\pm u$ are
independent of (non-unique) choice of the operator $\gamma^{-1}$, the operators $T^\pm :
H^{s,-\frac{1}{2}}({\Omega_\pm};A)\to H^{s-\frac{3}{2}}(\partial\Omega)$ are continuous, and
the first Green identity holds in the following form,
\begin{eqnarray} \label{Tcan}
 \pm\left\langle T^\pm u\,,\gamma^\pm  v\right\rangle _{\partial\Omega}
 =\langle \tilde{A}_{\Omega_\pm}  u,v \rangle_{{\Omega_\pm}}\ +\ \check\E_{\Omega_\pm} (u,v)
  =\langle \tilde{A}_{\Omega_\pm}  u-\check{A}_{\Omega_\pm}u,v \rangle_{{\Omega_\pm}}
  \quad
\forall\ v\in H^{2-s} ({\Omega_\pm}).
\end{eqnarray}
\end{theorem}
The canonical co-normal derivatives in Definition~\ref{Dccd} are completely defined by the function $u$
and operator $A$ only and do not depend explicitly on the right hand sides
$\tilde f_\pm$, unlike the generalised
co-normal derivatives defined in \eqref{Tgen}, while the operators $T^\pm$ are linear in $u$.
Note that the canonical co-normal derivatives coincides with the
classical co-normal derivatives $T^\pm u=T^{c\pm} u$ if the latter do exist, sf. \cite[Corollaries 6.11 and 6.14]{MikJMAA2013}, which is generally not the case for the generalised conormal derivatives even for smooth functions $u$, unless $\tilde f_\pm=\tilde{A}_{\Omega_\pm}  u$ is chosen.

Let $1/2<s<3/2$ and $a\in C_+^{|s-1|}(\overline\Omega_\pm)$. If  $u\in H^{s,-\ha}({\Omega_\pm};A)$, then Definitions \ref{GCDd} and \ref{Dccd} imply
that the generalised co-normal derivative for arbitrary extensions
$\tilde f_\pm\in \widetilde{H}^{s-2}({\Omega_\pm})$ of the distributions $r_{\Omega_\pm}Au$ can be
expressed as
\begin{equation}
\label{Tgentil} 
T^\pm(\tilde f_\pm;u)
=T^\pm u\, \pm(\gamma^{-1})^*(\tilde f_\pm-\tilde A_{\Omega_\pm} u).
\end{equation}
If $u\in H^s({\Omega_\pm})$ and $v\in H^{2-s,-\ha}({\Omega_\pm};A)$, then 
swapping over the roles of $u$ and $v$ in \eqref{Tcan},
we obtain the first Green identity for $v$,
\begin{equation} \label{Greentilde}
 \pm\left\langle T^\pm v\,,\, \gamma^\pm u\right\rangle_{\pO}= \check\E_{\Omega_\pm} (v,u) +\langle \tilde{A}_{\Omega_\pm}  v,u \rangle_{{\Omega_\pm}} .
\end{equation} 
If, in addition, $r_{\Omega_\pm}Au=r_{\Omega_\pm}\tilde f_\pm$, where $\tilde f_\pm\in
\widetilde{H}^{s-2}({\Omega_\pm})$, then subtracting \eqref{Greentilde} from \eqref{Tgen} and taking into account that $\check{\E}_{\Omega_\pm}(u,v)=\check{\E}_{\Omega_\pm}(v,u)$ by Remark~\ref{RE}, we obtain the following second Green identity,
\begin{equation} \label{2.5s}
\langle\tilde f_\pm,v \rangle_{\Omega_\pm} -\langle \tilde{A}_{\Omega_\pm}  v,u \rangle_{{\Omega_\pm}}=
 \pm\left\langle T^\pm(\tilde f_\pm;u)\,,\, \gamma^\pm  v\right\rangle _{\pO}
 \mp\left\langle T^\pm v\,,\, \gamma^\pm u\right\rangle _{\pO}
 .
\end{equation}
If, finally, $u\in H^{s,-\ha}({\Omega_\pm};A)$ and $v\in H^{2-s,-\ha}({\Omega_\pm};A)$, then we arrive at
the familiar form of the second Green identity for the canonical
extension and canonical co-normal derivatives
\begin{equation} \label{GreenCan}
\langle \tilde{A}_{\Omega_\pm}  u,v \rangle_{{\Omega_\pm}}- \langle \tilde{A}_{\Omega_\pm}  v,u \rangle_{{\Omega_\pm}}=
 \pm\left\langle T^\pm u\,,\, \gamma^\pm  v\right\rangle _{\pO}
 \mp\left\langle T^\pm v\,,\, \gamma^\pm u\right\rangle _{\pO}
 .
\end{equation}

\section{Parametrix and potential type operators on Lipschitz domains}\label{S3}
Unless stated otherwise, we will henceforth assume that $\Omega=\Omega_+$.


We will say, a function $P(x,y)$ of two variables $x,y\in \R^n$
is a parametrix (the Levi function)  for the operator $A(x,\pa_x)$
in $\R^n$ if (see, e.g., \cite{Levi1909, Hilbert1912, Miranda1970, Hellwig1977,
Pomp1998a, Pomp1998b, MikEABEM2002})
\begin{align}
\label{3.1} \dst A(x,\nabla_x)\,P(x,y)=\delta (x-y) +R(x,y),
\end{align}
where $\delta(\cdot)$ is the Dirac distribution and $R(x,y)$
possesses a weak (integrable) singularity at $x=y$, i.e.,
\begin{equation}
\label{3.2} \dst R(x,y)={\cal
O}\,(|x-y|^{-\varkappa})\;\;\;\mbox{\rm with}\;\;\;\; \varkappa<n.
\end{equation}

Let $\omega_n=\frac{2\,\pi^{n/2}}{\Gamma(n/2)}$ denote the area of the unit sphere in $\R^n$. 
It is well known that function 
\begin{equation}
\label{3.3D} P_\Delta(x,y)=\frac{-1}{(n-2)\omega_n\,|x-y|^{n-2}}\,,\;\;\;x,y\in \R^n
\end{equation}
is the fundamental solution of the Laplace equation, i.e., $\Delta_x P_\Delta(x,y)=\Delta_y P_\Delta(x,y)=\delta(x-y)$.

It  is easy to see that for the operator $A(x,\pa_x)$ given by the
left-hand side in \eqref{2.1}, the function
\begin{equation}
\label{3.3} \dst P(x,y)=\frac{1}{a(y)}P_\Delta(x,y)=\frac{-1}{(n-2)\omega_n\,a(y)\,|x-y|^{n-2}}\,,\;\;\;x,y\in \R^n,
\end{equation}
is a parametrix, while the corresponding remainder function is
\begin{equation}
\label{3.4} R(x,y)
=\nabla a(x)\cdot\nabla_x P(x,y)
=-\frac{1}{a(y)}\nabla a(x)\cdot\nabla_yP_\Delta(x,y)
=\frac{(x-y)\cdot\nabla a(x)}{\omega_n\,a(y)\,|x-y|^n} \,,\;\;\;x,y\in \R^n,
\end{equation}
and, if $a\in C_+^1(\R^n)$, satisfies estimate \eqref{3.2}  a.e., with $\varkappa=n-1$. 
Note also that 
\begin{align}
\label{3.1y} \dst A(y,\nabla_y)\,P(x,y)=\delta (x-y) +R_*(x,y),
\end{align}
where
\begin{equation}
\label{3.4*} R_*(x,y)
=-\nabla_y\cdot\left(P(x,y)\nabla a(y)\right)
=\frac{\Delta(\ln a(y))}{(n-2)\omega_n\,|x-y|^{n-2}}
-\frac{(x-y)\cdot\nabla a(y)}{\omega_n\,a(y)\,|x-y|^n} \,,\;\;\;x,y\in \R^n.
\end{equation}
Evidently, the parametrix $P(x,y)$ given by \eqref{3.3} is related with the
fundamental solution to the operator 
$A(y,\nabla_x):=a(y)\Delta_x$ 
with "frozen" coefficient $a(y)$, and
 $
  A(y,\nabla_x)\,P(x,y)=\delta (x-y).
 $

Note that parametrix \eqref{3.3} and remainders \eqref{3.4}, \eqref{3.4*} are not smooth enough for the corresponding potential operators to be directly treated as in \cite{McLean2000}, which thus need some additional consideration.

For $g\in \mathcal D(\R^n)$, 
the parametrix-based  volume potential operator and the remainder
potential operator, corresponding to parametrix \eqref{3.3} and to
remainders \eqref{3.4}, \eqref{3.4*}  for $y\in\R^n$ are
\begin{eqnarray}
{\bf P}g(y):=\langle P(\cdot,y),g\rangle_{\R^n}= \int\limits_{\R^n} P(x,y)\,g(x)\,dx,
&&\label{4.9bP}\\
{\bf R}g(y):=\langle R(\cdot,y),g\rangle_{\R^n}= \int\limits_{\R^n} R(x,y)\,g(x)\,dx,
&&{\bf R}_*g(y):=\langle R_*(\cdot,y),g\rangle_{\R^n}= \int\limits_{\R^n} R_*(x,y)\,g(x)\,dx,\label{4.9bR}
\end{eqnarray}
and from \eqref{3.1}-\eqref{4.9bR} we obtain, for sufficiently smooth coefficient $a$,
\begin{align}
{\bf P}Ag=g+{\bf R}g, \quad A{\bf P}g=g+{\bf R}_*g\quad \mbox{in } \R^n.
\end{align}

For the function $g$ defined on a domain $\Omega_+\subset\R^n$, e.g., $g\in \mathcal D(\overline\Omega_+)$, 
the corresponding potentials for $y\in \Omega_+$ are
\begin{eqnarray}
{\mathcal P}g(y):=\langle P(\cdot,y),g\rangle_{\Omega_+}= \int\limits_{\Omega_+} P(x,y)\,g(x)\,dx,
&&\label{4.9P}\\
{\mathcal R}g(y):=\langle R(\cdot,y),g\rangle_{\Omega_+}= \int\limits_{\Omega_+} R(x,y)\,g(x)\,dx,
&&{\mathcal R}_*g(y):=\langle R_*(\cdot,y),g\rangle_{\Omega_+}= \int\limits_{\Omega_+} R_*(x,y)\,g(x)\,dx.\label{4.9R}
\end{eqnarray}
From definitions \eqref{3.3}, \eqref{3.4}, \eqref{3.4*}, one can obtain
representations of the parametrix-based potential operators in terms of their counterparts for $a=1$ (i.e.
associated with the  Laplace operator $\Delta$, cf. e.g. \cite{Hsiao-Wendland2008}), which we equip with the subscript $\Delta$, cf. \cite{CMN-1},
\begin{align}
\label{3.d1Rn} 
{\mathbf P}\,g\,=\frac{1}{a}\;{\mathbf P}_\Delta\,g,\quad   
{\mathbf R}\,g=-\,\frac{1}{a}\; \nabla\cdot{\mathbf P}_\Delta\,(g\,\nabla a),\quad
&{\mathbf R}_*\,g=-\nabla\cdot\left(\frac{\nabla a}{a}{\mathbf P}_\Delta g\right),\\
\label{3.d1}  
{\mathcal P}\,g\,=\frac{1}{a}\;{\mathcal P}_\Delta\,g, \quad 
{\mathcal R}\,g=-\,\frac{1}{a}\; \nabla\cdot{\mathcal P}_\Delta\,(g\,\nabla a),\quad
&{\mathcal R}_*\,g=-\nabla\cdot\left(\frac{\nabla a}{a}{\mathcal P}_\Delta g\right).
\end{align}
Hence
\begin{equation}\label{DeltaPg}
    \Delta(a\mathbf P g)=g\ \text{in}\ \R^n,\quad    \Delta(a\P g)=g\ \text{in}\ \Omega. 
\end{equation}
Employing relations \eqref{3.d1} and the well known properties of the operator ${\bf P}_\Delta$ as the pseudo-differential operator of order $-2$ together with Theorem~\ref{GrL+},
definitions \eqref{4.9bP}-\eqref{4.9bR} can be extended to $g\in H^s(\R^n)$, $g\in\widetilde H^s(\Omega)$ and lower-smoothness coefficient $a$.  
For $g\in \widetilde H^s(\Omega)$ and $g\in H^s(\Omega)$, the potentials ${\mathcal P}$, ${\mathcal R}$, ${\mathcal R}_*$ defined on functions (or distributions) having support on $\overline\Omega$ are understood as 
\begin{align}\label{P1til}
&{\mathcal P}g:={\bf P}g,\quad\
{\mathcal R}g:={\bf R}g,\quad\
{\mathcal R_*}g:={\bf R}_*g,
\qquad\ g\in \widetilde H^s(\Omega),\ s\in \R,\\
\label{P2}
&{\mathcal P}g:=r_{\Omega}{\bf P}\mathring E_{\Omega} g,\ \
{\mathcal R}g:=r_{\Omega}{\bf R}\mathring E_{\Omega} g,\ \ 
{\mathcal R_*}g:=r_{\Omega}{\bf R}_*\mathring E_{\Omega} g, \ \ \quad g\in H^s({\Omega}),\ s> -\ha.
\end{align}

To prove mapping properties of the parametrix-based volume potential operators in Sobolev spaces, we first provide some well-known results for the classical Newtonian volume potential associated with the Laplace operator.
\begin{lemma}\label{T3.1PD}
Let  $\Omega$ be a bounded Lipschitz domain in $\R^n$. The
following operators are continuous
%
\begin{eqnarray} 
\mu{\bf P}_\Delta  &:& {H}^{s}(\R^n) \to H^{s+2}({\R^n}), \quad  s\in \R,\quad \forall\ \mu\in\mathcal D(\R^n);\label{T3.1P0D}\\ 
{\P}_\Delta  &:& \widetilde{H}^{s}(\Omega) \to H^{s+2}(\Omega), \quad  s\in \R;\label{T3.1P1D}\\
{\cal P}_\Delta &:&  {H}^{s}(\Omega) \to H^{s+2}( {\Omega}), \quad-\ha<s<\ha; \label{T3.1P2D}\\
{\cal P}_\Delta&:& \widetilde{H}^{s}(\Omega)\to H^{s+2,-\ha}(\Omega;\Delta),\quad s\ge -\ha;\label{T3.1P1haD}\\
\gamma^+{\P}_\Delta  &:&\  \widetilde{H}^{s}(\Omega) \to H^{s+\frac{3}{2}}(\pO), \quad -\frac{3}{2}<s<-\ha; \label{T3.1P1+D}\\
\gamma^+{\P}_\Delta &:&  {H}^{s}(\Omega) \to H^{1}( {\pO}), \quad
 -\ha<s; \label{T3.1P2+D}\\
T^+_\Delta{\P}_\Delta  &:&\  \widetilde{H}^{s}(\Omega) \to
L_2( {\pO}), \quad -\frac{1}{2}<s;\label{T3.1P1T+D}\\
T^+_\Delta{\P}_\Delta &:&  {H}^{s}(\Omega) \to L_2( {\pO}), \quad
-\ha<s. \label{T3.1P2T+D}
\end{eqnarray}
If $-\tha<s<-\ha$, $\tilde f\in \widetilde H^s(\Omega)$, and $\tilde f_0\in \widetilde H^s(\Omega)$ is such that $r_{\Omega} \tilde f_0=r_{\Omega} \tilde f$, then there exist constants $C_0,C_1>0$ such that
\begin{equation}\label{TPf0D}
\|T_\Delta^+(\tilde{f}_0;\P_\Delta\tilde f)\|_{H^{s+\ha}(\pO)}\le
C_1\|\tilde f\|_{\widetilde H^s(\Omega)} + C_2\|\tilde{f}_0\|_{\widetilde{H}^{s}(\Omega)}.
\end{equation}
\end{lemma}
\begin{proof}  
Operator \eqref{T3.1P0D} and hence \eqref{T3.1P1D} is continuous since ${\bf P}_\Delta$ is a pseudo-differential operator of order $-2$. 
Continuity of operator \eqref{T3.1P2D} follows from the first relation in \eqref{P2} for ${\cal P}_\Delta$ and ${\bf P}_\Delta$, and \eqref{T3.1P1D}.
Since $\Delta \mathcal P_\Delta g=g$ in $\Omega$, continuity of operator \eqref{T3.1P1D} implies continuity of operator \eqref{T3.1P1haD}.

Continuity of operator \eqref{T3.1P1+D} is implied by continuity of operator \eqref{T3.1P1D} and the trace theorem for Lipschitz domains, cf. e.g. \cite[Lemma 3.6]{Costabel1988}, \cite[Theorem 3.38]{McLean2000}.
Continuity of operator \eqref{T3.1P2+D} follows from continuity of operator \eqref{T3.1P2D} and e.g. \cite{Zolesio1977}, \cite[Lemma 2.5]{MikJMAA2011} for $-\ha<s<\ha$, and then by the embedding argument for $s\ge\ha$.

Continuity of operators \eqref{T3.1P1T+D} and \eqref{T3.1P2T+D} is implied by continuity of \eqref{T3.1P1D} and \eqref{T3.1P2D}, respectively, and by \cite[Corollary 3.14]{MikJMAA2011}  since in the both cases $s+2>\tha$. Estimate \eqref{TPf0D} follows from continuity of operator \eqref{T3.1P1D}, relation $\Delta \mathcal P_\Delta \tilde f=\tilde f$, and estimate \eqref{estimate}.
\end{proof} 

Now the following  mapping properties of the parametrix-based operators can be obtained.
\begin{theorem}\label{T3.1P}
Let  $\Omega$ be a bounded Lipschitz domain in $\R^n$. The
following operators are continuous
%
\begin{eqnarray} 
\mu{\bf P}  &:& {H}^{s}(\R^n) \to H^{s+2}({\R^n}), \quad  s\in \R, \quad  
a\in C_+^{|s+2|}(\R^n),\quad \forall\ \mu\in\mathcal D(\R^n);\label{T3.1P0}\\ 
{\P}  &:& \widetilde{H}^{s}(\Omega) \to H^{s+2}( {\Omega}), \quad  s\in \R, \quad  
a\in C_+^{|s+2|}(\overline\Omega);\label{T3.1P1}\\
{\cal P} &:&  {H}^{s}(\Omega) \to H^{s+2}( {\Omega}), \quad
-\ha<s<\ha, \quad  a\in C_+^{s+2}(\overline\Omega); \label{T3.1P2}\\
  {\cal P}&:& \widetilde{H}^{s}(\Omega)\to
H^{s+2,-\ha}(\Omega;A),\quad 
-\ha\le s,\quad a\in C_+^{s+2}(\overline\Omega);\label{T3.1P1ha}\\
\mu\mathbf R &:& H^{s-1}(\R^n) \to
H^{s}(\R^n),\quad s\in\R, \quad  a\in C_+^{|s-1|+1}(\R^n),\quad \forall\ \mu\in\mathcal D(\R^n);  \label{T3.1P3Rn}\\
{\cal R}  &:&  H^{s-1}(\Omega) \to
H^{s}( {\Omega}),\quad \ha<s<\tha, \quad  a\in C_+^{|s-1|+1}(\overline\Omega);  \label{T3.1P3}\\
{\cal R}&:& H^{s}(\Omega) \to H^{s}(\Omega), \quad \ha<s<\tha, 
\quad  a\in C_+^{s}(\overline\Omega);
\label{T3.1P3has}\\
{\cal R}&:& H^{s}(\Omega) \to H^{s,-\ha}(\Omega;A), \quad \ha<s<\tha, 
\quad  a\in C_+^{\tha}(\overline\Omega);
\label{T3.1P3ha}\\
\mu\mathbf R_*&:& H^{s}(\R^n) \to H^{s+1}(\R^n), \quad s\in\R, 
\quad  a\in C_+^{|s+2|+1}(\R^n),\quad \forall\ \mu\in\mathcal D(\R^n);
\label{T3.1P3ha*Rn}\\
{\cal R}_*&:& \widetilde H^{s}(\Omega) \to H^{s+1}(\Omega), \quad s\in\R, 
\quad  a\in C_+^{|s+2|+1}(\overline\Omega);
\label{T3.1P3ha*}\\
{\cal R}_*&:& \widetilde H^{s}(\Omega) \to H^{\sigma}(\Omega), \quad -\frac{3}{2}<s, 
\quad  a\in C_+^{\tha}(\overline\Omega), \mbox{ for some } \sigma>-\ha;
\label{T3.1P3ha**}\\
  \gamma^+{\P}  &:&\  \widetilde{H}^{s}(\Omega) \to H^{s+\frac{3}{2}}(\pO), \quad -\frac{3}{2}<s<-\ha, \quad  
a\in C_+^{s+2}(\overline\Omega);\label{T3.1P1+}\\
 \gamma^+{\P} &:&  {H}^{s}(\Omega) \to H^{1}( {\pO}), \quad
 -\ha<s, \quad  a\in C_+^{\tha}(\overline\Omega); \label{T3.1P2+}\\
\gamma^+{\cal R}  &:&  {H}^{s}(\Omega) \to
H^{s-\ha}( {\pO}), \quad \frac{1}{2}<s<\tha, \quad  a\in C_+^{s}(\overline\Omega) ;\label{T3.1P3+}\\
T^+{\P}  &:&  \widetilde{H}^{s}(\Omega) \to
L_2({\pO}), \quad  -\ha<s, \quad  a\in C_+^{\tha}(\overline\Omega); \label{T3.1P1T+}\\
T^+{\P}  &:&  {H}^{s}(\Omega) \to L_2({\pO}), \quad
 -\ha<s, \quad  a\in C_+^{\tha}(\overline\Omega);  \label{T3.1P2T+}\\
T^+{\cal R}  
&:&{H}^{s}(\Omega) \to H^{s-\ha}( {\pO}), \quad
 \ha<s,  
\quad  a\in C_+^{\tha}(\overline\Omega).\label{T3.1P4T+}
\end{eqnarray}
Moreover, operators \eqref{T3.1P3has}, \eqref{T3.1P3ha}, \eqref{T3.1P3+}, 
\eqref{T3.1P4T+} are compact.

If $-\tha<s<-\ha$,  $a\in C_+^{s+2}(\overline\Omega)$, $\tilde f\in \widetilde H^s(\Omega)$,  and $\tilde f_0\in \widetilde H^s(\Omega)$ is such that $r_{\Omega} \tilde f_0=r_{\Omega} A\mathcal P\tilde f$, then there exist constants $C_0,C_1>0$ such that
\begin{equation}\label{TPf0}
\|T^+(\tilde{f}_0;\P\tilde f)\|_{H^{s+\ha}(\pO)}\le
C_1\|\tilde f\|_{\widetilde H^s(\Omega)} + C_2\|\tilde{f}_0\|_{\widetilde{H}^{s}(\Omega)}.
\end{equation}
\end{theorem}
\begin{proof}  
Continuity of operators  \eqref{T3.1P0} -\eqref{T3.1P2} is implied by the first relations in \eqref{3.d1Rn}, \eqref{3.d1} and continuity of operators  \eqref{T3.1P0D} -\eqref{T3.1P2D} together with Theorem~\ref{GrL+}.

Continuity of operators \eqref{T3.1P1}, \eqref{T3.1P2}
and Remark~\ref{R2.1}(\ref{iii}) imply continuity of operator \eqref{T3.1P1ha} for
$s>-\ha$. 
Let us now prove \eqref{T3.1P1ha} for $s=-\ha$. For $g\in
\widetilde{H}^{-\ha}(\Omega)$, we
have, ${\cal P}\,g\in H^\tha(\Omega)$ due to \eqref{T3.1P1}, while
\begin{align}
A {\cal P}\,g\, =\nabla\cdot\left(a\nabla\left[\frac{1}{a }\;{\cal P}_\Delta\,g\,\right]\right)= 
g\, -\nabla\cdot\left[(\nabla\ln a){\cal P}_\Delta\,g\right]
\ \mbox{in}\ \Omega,\label{DPgB}
\end{align}
where we have taken into
account that $\Delta{\cal P}_\Delta\,g=g$. 
The first term in the right hand side of \eqref{DPgB}
belongs to $\widetilde{H}^{-\ha}_\bullet(\Omega)$, while, since $a\in
C_+^\tha(\overline\Omega)$, $a>0$, the  second term belongs to ${H}^{\ha}(\Omega)$
and can be extended by zero to $\widetilde{H}^0(\Omega)\subset
\widetilde{H}^{-\ha}(\Omega)$, which completes the proof of continuity for
operator \eqref{T3.1P1ha}.

Continuity of the operator \eqref{T3.1P3Rn}, follows from the second relation \eqref{3.d1Rn} together with Theorem~\ref{GrL+} and continuity of operator \eqref{T3.1P0D}. Indeed, let us take arbitrary $\mu\in\mathcal D(\R^n)$, let $B_\mu$ be a ball such that $\supp \mu\subset B_\mu$ and let $\mu_1\in\mathcal D(\R^n)$ be such that $\mu_1=1$ in $B_\mu$.
Then for any $g\in  H^{s-1}(\R^n)$, we have,
\begin{multline}\label{Rloc}
\|\mu{\mathbf R}\,g\|_{H^{s}(\R^n)}
=\left\|\frac{\mu}{a}\; \nabla\cdot\left(\mu_1{\mathbf P}_\Delta\,(g\,\nabla a)\right)\right\|_{H^{s}(\R^n)}
\le c_1\|\nabla\cdot\left(\mu_1{\mathbf P}_\Delta\,(g\,\nabla a)\right)\|_{H^{s}(\R^n)}
\le c_2\|\mu_1{\mathbf P}_\Delta\,(g\,\nabla a)\|_{H^{s+1}(\R^n)}\\
\le c_3\|g\,\nabla a\|_{H^{s-1}(\R^n)}
\le c_4\|g\|_{H^{s-1}(\R^n)},
\end{multline}
where $c_i$ are some positive constants (depending on $\mu$, $\mu_1$ and $a$), and we took into account that 
$C_+^{|s-1|+1}(\R^n)\subset  C_+^{|s|}(\R^n)$ since $|s|\le|s-1|+1$. 
This implies continuity of \eqref{T3.1P3Rn}.

To prove continuity of the operator \eqref{T3.1P3}, we similarly employ the second relation in \eqref{3.d1} together with Theorem~\ref{GrL+} and continuity of operator \eqref{T3.1P2D}. 
Then we obtain for any $g\in  H^{s-1}(\Omega)$, $1/2<s<3/2$, and some positive constants $c_i$,
\begin{multline}\label{Rs-1tos}
\|{\mathcal R}\,g\|_{H^{s}(\Omega)}
=\left\|\frac{1}{a}\; \nabla\cdot{\mathcal P}_\Delta\,(g\,\nabla a)\right\|_{H^{s}(\Omega)}
\le c_1\|\nabla\cdot{\mathcal P}_\Delta\,(g\,\nabla a)\|_{H^{s}(\Omega)}
\le c_2\|{\mathcal P}_\Delta\,(g\,\nabla a)\|_{H^{s+1}(\Omega)}\\
\le c_3\|g\,\nabla a\|_{H^{s-1}(\Omega)}
\le c_4\|g\|_{H^{s-1}(\Omega)}.
\end{multline}

Let us prove continuity and compactness of operator \eqref{T3.1P3has}. 
For $1\le s<\tha$, we have $s=|s-1|+1$ and then continuity of operator \eqref{T3.1P3} implies continuity and compactness of \eqref{T3.1P3has}.
For $\ha< s<1$, we need a sharper estimate of the norm 
$\|g\,\nabla a\|_{H^{s-1}(\Omega)}$.
First, by Definition~\ref{C+mu} the inclusion $ a\in C_+^{s}(\overline\Omega)$ implies that there exists 
$t\in (s, 1)$ such that,
$a\in C^{0,t}(\overline\Omega)=B^t_{\infty,\infty}(\Omega)=F^t_{\infty,\infty}(\Omega)$, see e.g. Proposition in  \cite[Section 2.1.2]{Runst-Sickel1996}, and hence 
$\nabla a\in  
F^{t-1}_{\infty,\infty}(\Omega)$. 
Then, by Theorems 1  from \cite[Section 4.4.3]{Runst-Sickel1996},
\begin{align}\label{3.48a}
\|g\,\nabla a\|_{F^{t-1}_{2,\infty}(\Omega)}
\le C\|\nabla a\|_{F^{t-1}_{\infty,\infty}(\Omega)}\, \|g\|_{H^{\sigma}(\Omega)}
\le C\|a\|_{C^{0,t}(\overline\Omega)}\, \|g\|_{H^{\sigma}(\Omega)},\quad 
\forall\ 
\sigma\in(1-t, s).
\end{align}
On the other hand, by \eqref{Rs-1tos}, item (ii) of Proposition from \cite[Section 2.2.1]{Runst-Sickel1996}, and \eqref{3.48a}, we obtain
$$
\|{\mathcal R}\,g\|_{H^{s}(\Omega)}\le
c_3\|g\,\nabla a\|_{H^{s-1}(\Omega)}=c_3\|g\,\nabla a\|_{F^{s-1}_{2,2}(\Omega)}
\le C_1\|g\,\nabla a\|_{F^{t-1}_{2,\infty}(\Omega)}
\le C_1 C\|a\|_{C^{0,t}(\overline\Omega)}\, \|g\|_{H^{\sigma}(\Omega)}.
$$
Thus the operator ${\mathcal R}: H^{\sigma}(\Omega)\to {H^{s}(\Omega)}$ is continuous, which implies continuity and, by the Rellich compact embedding theorem, also compactness of operator  \eqref{T3.1P3has} for $\ha< s<1$.

Let us prove continuity of operator
\eqref{T3.1P3ha}.
Since $a\in C_+^{\tha}(\overline\Omega)$, then by Definition~\ref{C+mu} there exists $\epsilon>0$ such that 
$a\in C^{1,\ha+\epsilon}(\overline\Omega)$, 
and let us chose any  $\sigma\in(\ha,\min\{s,\ha+\epsilon\})$. 
By continuity of \eqref{T3.1P3}, the operator  $\mathcal R:H^{\sigma}(\Omega) \to H^{s}(\Omega)$ is continuous. 
Now let us prove that the operator $A\mathcal R:H^{\sigma}(\Omega) \to \widetilde H^{-\ha}_\bullet(\Omega)$ is continuous as well. 
Indeed, for some positive constants $c_i$, we have,
\begin{multline*}
\|A{\mathcal R}\,g\|_{\widetilde H_\bullet^{-\ha}(\Omega)}
\le\|A{\mathcal R}\,g\|_{\widetilde H_\bullet^{\sigma-1}(\Omega)}
\le c_0\|A{\mathcal R}\,g\|_{H^{\sigma-1}(\Omega)}
=c_0\left\|\nabla\cdot\left[a\nabla\left(\frac{1}{a}\; \nabla\cdot{\mathcal P}_\Delta\,(g\,\nabla a)\right)\right]
\right\|_{H^{\sigma-1}(\Omega)}\\
=c_0\left\|-\nabla\cdot[(\nabla\ln a) \nabla\cdot{\mathcal P}_\Delta\,(g\,\nabla a)]
+\Delta(\nabla\cdot{\mathcal P}_\Delta\,(g\,\nabla a))\right\|_{H^{\sigma-1}(\Omega)} 
\le c_1\|-(\nabla\ln a) \nabla\cdot{\mathcal P}_\Delta\,(g\,\nabla a) 
+ 
(g\,\nabla a)\|_{H^{\sigma}(\Omega)}\\
\le c_2\|a\|_{C^{1,\ha+\epsilon}}\|{\mathcal P}_\Delta\,(g\,\nabla a)\|_{H^{\sigma+1}(\Omega)}
+c_1\|g\,\nabla a\|_{H^{\sigma}(\Omega)}
\le c_3\|g\,\nabla a\|_{H^{\sigma-1}(\Omega)}+c_1\|g\,\nabla a\|_{H^{\sigma}(\Omega)}\\
\le c_4\|g\,\nabla a\|_{H^{\sigma}(\Omega)}
\le c_5\|a\|_{C^{1,\ha+\epsilon}}\|g\|_{H^\sigma(\Omega)}.
\end{multline*}
Hence we proved continuity of the operator $H^\sigma(\Omega) \to H^{s,-\ha}(\Omega;A)$, which implies continuity  of operator  \eqref{T3.1P3ha}
and by the Rellich compact embedding theorem also its compactness.

Continuity of operator \eqref{T3.1P3ha*Rn} is implied by the last relation in \eqref{3.d1Rn}, continuity of operator \eqref{T3.1P0D} and Theorem~\ref{GrL+} in the chain of inequalities analogous to \eqref{Rloc}.
Similarly, continuity of operator \eqref{T3.1P3ha*} is implied by the last relation in \eqref{3.d1}, continuity of operator \eqref{T3.1P1D} and Theorem~\ref{GrL+}.
Continuity of operator \eqref{T3.1P3ha**} is implied by continuity of \eqref{T3.1P3ha*} since $a\in C_+^{\tha}(\overline\Omega)$ implies that there exists $\epsilon>0$ such that 
$a\in C^{1,1/2+\epsilon}(\overline\Omega)$, and we can  take $\sigma\in(\tha,\min\{s+1,\tha+\epsilon\})$.

Continuity of operator \eqref{T3.1P1+} is implied by continuity of operator \eqref{T3.1P1} and the trace theorem for Lipschitz domains, cf. e.g. \cite[Lemma 3.6]{Costabel1988}, \cite[Theorem 3.38]{McLean2000}.
Continuity of operator \eqref{T3.1P2+} follows from continuity of operator \eqref{T3.1P2} for $-\ha<s<-\ha+\epsilon$ with any sufficiently small $\epsilon>0$  together with e.g. \cite{Zolesio1977}, \cite[Lemma 2.5]{MikJMAA2011}, and then by the embedding argument for all $s>-\ha$.
Similarly, continuity of operators \eqref{T3.1P1T+} and \eqref{T3.1P2T+} is implied by continuity of \eqref{T3.1P1} and \eqref{T3.1P2}, respectively, and by \cite[Corollary 3.14]{MikJMAA2011} since in the both cases $s+2>\tha$. 

Continuity and compactness of operators \eqref{T3.1P3+} and \eqref{T3.1P4T+} are implied by continuity and compactness of operators \eqref{T3.1P3has} and \eqref{T3.1P3ha}, the trace theorem for Lipschitz domains and Theorem~\ref{T6.6}.

Estimate \eqref{TPf0} follows from continuity of operator \eqref{T3.1P1} and estimate \eqref{estimate}.
\end{proof} 


The parametrix-based single and the double layer surface potential operators are defined as
\begin{eqnarray}
\label{3.6} && 
Vg(y):=-\int\limits_{\pO} P(x,y)\, \psi(x)\,dS_x,   \;\;\;\; y\not\in \pO,  \\
\label{3.7} && Wg(y):=-\int\limits_{\pO}
\big[\,T^c(x,n(x),\pa_x)\,P(x,y)\,\big]\, \, \varphi(x)\,dS_x,   \;\;\;\;
y\not\in \pO,\quad
\end{eqnarray}
where the integrals are understood as duality forms if
$\psi$ and $\varphi$ are not integrable. Particularly, for 
$\psi\in H^{\ha-s}(\partial\Omega)$, $\varphi\in H^{\ha-s}(\partial\Omega)$, $\ha<s<\tha$, we have
\begin{eqnarray}
\label{3.6d} && 
V\psi(y):=-\langle \gamma P(\cdot,y), \psi\rangle_{\partial\Omega}=-\langle P(\cdot,y), \gamma^*\psi\rangle_{\R^n}
=-\mathbf P\gamma^*\psi\,(y)=-\frac{1}{a(y)}\mathbf P_\Delta\gamma^*\psi(y),   
\\
\label{3.7d} && W\varphi(y):=-\langle T^c P(\cdot,y), \varphi\rangle_{\partial\Omega}
=-\langle P(\cdot,y), T^{c*}\varphi\rangle_{\R^n}
=-\mathbf P T^{c*}\varphi\,(y)=-\frac{1}{a(y)}\mathbf P_\Delta T^{c*}\varphi\,(y),   
\end{eqnarray}
where $\gamma^*\psi$ and $T^{c*}\varphi$ are well defined for any $\psi\in\mathcal D^*(\partial\Omega)$  and for any $\varphi\in L_1(\partial\Omega)$, $a\in L_\infty(\partial\Omega)$, in the sense of distribution as
\begin{align}
&\langle\gamma^*\psi,\phi\rangle_{\R^n}:=\langle\psi,\gamma\phi\rangle_{\partial\Omega},\quad 
\langle T^{c*}\varphi,\phi\rangle_{\R^n}:=\langle\varphi,T^c\phi\rangle_{\partial\Omega}
=\langle \varphi,aT^c_\Delta\phi\rangle_{\partial\Omega},\quad \forall\phi\in \mathcal D(\R^n),
\end{align}
which evidently implies that 
$\supp \gamma^*\psi\subset \partial\Omega$ and $\supp T^{c*}\varphi\subset \partial\Omega$.  
Moreover,
\begin{align}\label{gT*}
\gamma^*: H^{\ha-s}(\partial\Omega)\to H^{-s}_{\partial\Omega},\quad 
T^{c*}: H^{\ha-s}(\partial\Omega)\to H^{-s-1}_{\partial\Omega}, \quad\ha<s<\tha,
\end{align}
 are  the continuous operators adjoint, respectively, to the continuous trace operator 
$\gamma:H^s_{\rm loc}(\R^n)\to H^{s-\ha}(\partial\Omega)$ and to the continuous classical conormal derivative operator 
$T^c:H^{s+1}_{\rm loc}(\R^n)\to H^{s-\ha}(\partial\Omega)$; for the continuity of $T^c$ and $T^{c*}$, it is also assumed that $a\in C_+^{s-\ha}(\partial\Omega)$.

When $a=1$, formulas \eqref{3.6}, \eqref{3.7} define the corresponding harmonic potentials that we denote as $V_\Delta$ and $W_\Delta$, respectively.
From definitions \eqref{3.6}, \eqref{3.7}, similar to \eqref{3.d1}, we have, cf. \cite{CMN-1},
\begin{eqnarray}
\label{VWa}
Vg=\frac{1}{a}V_{_\Delta} g,&&\quad Wg=\frac{1}{a}W_{_\Delta}(ag).
\end{eqnarray}
Hence
 \begin{equation}\label{DV,DW=0}
\Delta(aVg)=0, \quad \Delta(aWg)=0\ \text{in}\ \Omega_\pm. 
 \end{equation}

We will mainly need the restrictions of the layer potentials  to $\Omega$, i.e., $r_{\Omega}{V}$, $r_{\Omega}{W}$, but will often omit the restriction operator $r_{\Omega}$ if this is clear from the context. 

The mapping properties as well as jump relations for the single and double layer potentials are well known for the case
$a=const$. Employing \eqref{3.d1}-\eqref{VWab4}, they were extended to the case of infinitely smooth boundary and variable coefficient $a(x)$ in \cite{CMN-1,CMN-2}. 
Before proving the corresponding properties for the parametrix-based potentials on Lipschitz domains, we collect below the following well-know mapping and jump properties for the harmonic potentials on Lipschitz domains.
\begin{theorem}\label{T3.1s0D}
Let  $\Omega$ be a bounded Lipschitz domain in $\R^n$.

(i)  If $\ha\le s\le \tha$, then the following operators are continuous for any $\mu\in\mathcal D(\R^n)$,
 \begin{align}
\mu V_\Delta  &: H^{s-\frac{3}{2}}(\pO) \to H^{s}(\R^n),
\quad
r_{\Omega}W_\Delta  :  H^{s-\ha}(\pO)\to H^{s}(\Omega),
\quad
\mu\, r_{\Omega_-}W_\Delta  :  H^{s-\ha}(\pO)\to H^{s}(\overline\Omega_-).\label{WHs1D}
\end{align}

(ii) If  $\ha< s< \tha$,  then the following operators are continuous, 
 \begin{eqnarray}
\gamma^\pm V_\Delta : H^{s-\frac{3}{2}}(\pO) \to H^{s-\ha}(\partial\Omega), 
&&
\gamma^\pm W_\Delta :  H^{s-\ha}(\pO)\to H^{s-\ha}(\partial\Omega),\label{WHs1gD}\\
T_\Delta^\pm V_\Delta : H^{s-\frac{3}{2}}(\pO) \to H^{s-\tha}(\partial\Omega), 
&&
T_\Delta^\pm W_\Delta :  H^{s-\ha}(\pO)\to  H^{s-\tha}(\partial\Omega),\label{WHs1GaTD}
\end{eqnarray}

(iii) If  $\ha< s< \tha$, then for any $\varphi\in H^{s-\ha}(\partial\Omega)$ and $\psi\in H^{s-\tha}(\partial\Omega)$ the following jump properties hold,
 \begin{eqnarray}
\gamma^+ V_\Delta\psi -\gamma^- V_\Delta\psi=0, 
&&
\gamma^+ W_\Delta\varphi -\gamma^- W_\Delta\varphi=-\varphi,\label{WHs1gjD}\\
T_\Delta^+ V_\Delta\psi -T^- V_\Delta\psi=\psi, 
&&
T_\Delta^+ W_\Delta\varphi -T^- W_\Delta\varphi=0.\label{WHs1GaTjD}
\end{eqnarray}
\end{theorem}
\begin{proof}
Items (i) and (ii) follow e.g. from \cite[Theorem 1(i,ii) and Remark]{Costabel1988}, \cite{Verchota1984, JK1981BAMS, JK1981AM,JK1982}, cf. also \cite[Theorem 6.12]{McLean2000}, if take into account that the canonical conormal derivative operators in \eqref{WHs1GaTD} are well defined  since $\Delta V=0$ and $\Delta W=0$ in $\Omega_\pm$.
The jump properties of item (iii) for $s=1$ are implied e.g. by \cite[Lemma 4.1]{Costabel1988}, cf. also \cite[Theorem 6.11]{McLean2000}. 
Hence they evidently hold if $1\le s<\tha$, and by the density argument also if $\ha< s< 1$.
\end{proof}

Employing relations \eqref{VWa}, Theorem~\ref{T3.1s0D}, and Theorem~\ref{GrL+}, we obtain the following mapping  properties for the parametrix-based potentials on Lipschitz domains. 
\begin{theorem}\label{T3.1s0}
Let  $\Omega$ be a bounded Lipschitz domain.

(i) The following operators are continuous if  $\ha\le s\le \tha$,
 \begin{eqnarray}
\mu V  &:& H^{s-\frac{3}{2}}(\pO) \to H^{s}(\R^n),\quad a\in C_+^{s}(\R^n),\quad \forall\ \mu\in\mathcal D(\R^n);\label{VHs1}\\
r_{\Omega}W  &:&  H^{s-\ha}(\pO)\to H^{s}(\Omega),\quad a\in C_+^{s}(\overline\Omega);\label{WHs1}\\
\mu\,r_{\Omega_-}W  &:&  H^{s-\ha}(\pO)\to H^{s}(\overline\Omega_-),\quad a\in C_+^{s}(\overline\Omega_-),\quad \forall\ \mu\in\mathcal D(\R^n)\label{WHs1-}.
\end{eqnarray}
(ii) The following operators are continuous if  $\ha< s\le \tha$ and $a\in C_+^{\tha}(\overline\Omega)$,
 \begin{eqnarray}
r_{\Omega}V  &:& H^{s-\frac{3}{2}}(\pO) \to H^{s,-\ha}( {\Omega;A}); \label{VHs1Ga}\\
\mu\,r_{\Omega_-}V  &:& H^{s-\frac{3}{2}}(\pO) \to H^{s,-\ha}_{\rm loc}({\Omega_-;A}),\quad \forall\ \mu\in\mathcal D(\R^n);\label{VHs1Ga-}\\
r_{\Omega}W  &:&  H^{s-\ha}(\pO)\to  H^{s,-\ha}( {\Omega;A});\label{WHs1Ga}\\
\mu\,r_{\Omega_-}W  &:&  H^{s-\ha}(\pO)\to  H^{s,-\ha}( {\Omega_-;A}),\quad \forall\ \mu\in\mathcal D(\R^n).\label{WHs1Ga-}
\end{eqnarray}
(iii) The following operators are continuous if  $\ha< s< \tha$,
 \begin{eqnarray}
\gamma^\pm V  &:& H^{s-\frac{3}{2}}(\pO) \to H^{s-\ha}(\partial\Omega),\quad a\in C_+^{s}(\overline\Omega_\pm);\label{VHs1g}\\
\gamma^\pm W  &:&  H^{s-\ha}(\pO)\to H^{s-\ha}(\partial\Omega),\quad a\in C_+^{s}(\overline\Omega_\pm);\label{WHs1g}\\
T^\pm V  &:& H^{s-\frac{3}{2}}(\pO) \to H^{s-\tha}(\partial\Omega),\quad 
a\in C_+^{\tha}(\overline\Omega_\pm);\label{VHs1GaT}\\
T^\pm W  &:&  H^{s-\ha}(\pO)\to  H^{s-\tha}(\partial\Omega),\quad a\in C_+^{\tha}(\overline\Omega_\pm).\label{WHs1GaT}
\end{eqnarray}
\end{theorem}
\begin{proof} 
Relations \eqref{VWa}, Theorem~\ref{T3.1s0D}(i), and Theorem~\ref{GrL+} immediately imply continuity of operators \eqref{VHs1} and \eqref{WHs1}. 
Further, if $a\in C_+^{\tha}(\overline\Omega)$ then there exists $\epsilon>0$ such that 
$a\in C^{1,\ha+\epsilon}(\overline\Omega)$. For  $\ha< s\le \tha$, $g\in{H}^{s-\tha}(\Omega)$ and any 
$\sigma\in(\ha, \min\{s,\ha+\epsilon\})$, we
have, 
\begin{align*}
\|A Vg\|_{H^{\sigma-1}(\Omega)} 
&=\left\|\nabla\cdot\left(a\nabla\left[\frac{1}{a }\;V_\Delta\,g\,\right]\right)\right\|_{H^{\sigma-1}(\Omega)}
=\|\nabla\cdot\left[(\nabla\ln a)V_\Delta\,g\right]\|_{H^{\sigma-1}(\Omega)}\\
&\le\|\left[(\nabla\ln a)V_\Delta\,g\right]\|_{H^\sigma}
\le C\|a\|_{ C^{1,\ha+\epsilon}(\overline\Omega)}\|V_\Delta\,g\|_{H^{\sigma}(\Omega)}
\le C\|a\|_{ C^{1,\ha+\epsilon}(\overline\Omega)}\|V_\Delta\,g\|_{H^{s}(\Omega)},
\end{align*}
where we have taken into
account that $\Delta V_\Delta\,g=0$ in $\Omega$. 
Hence along  with continuity of the first operator in \eqref{WHs1D} this implies $A Vg\in {H}^{\sigma-1}(\Omega)$.
Therefore $r_{\Omega}A Vg$ can be extended by zero to $\widetilde{H}^{\sigma-1}(\Omega)\subset
\widetilde{H}^{-\ha}(\Omega)$ with the corresponding norm estimate, from which  continuity of operator \eqref{VHs1Ga} follows.
Continuity of operator \eqref{WHs1Ga} is proved in a similar fashion.

Continuity of operators \eqref{WHs1Ga}, \eqref{VHs1Ga-}, \eqref{WHs1-},  and \eqref{WHs1Ga-} immediately follows from their counterparts for the interior domain.

Continuity of operators \eqref{VHs1g}, \eqref{WHs1g}, for the potential traces, is implied by continuity of operators \eqref{VHs1}, \eqref{WHs1},  \eqref{WHs1-} and the trace theorem, while continuity of operators \eqref{VHs1GaT}, \eqref{WHs1GaT}, for the potential conormal derivatives, is implied by continuity of operators \eqref{VHs1Ga},  \eqref{VHs1Ga-} \eqref{WHs1Ga},  \eqref{WHs1Ga-} and Theorem~\ref{T6.6}.
\end{proof}

Now we can prove the jump properties for the parametrix-based potentials on Lipschitz domains. 
\begin{theorem}\label{T3.1sJ}
Let  $\partial\Omega$ be a compact Lipschitz boundary,  $\ha< s< \tha$, $\varphi\in H^{s-\ha}(\partial\Omega)$ and $\psi\in H^{s-\tha}(\partial\Omega)$. Then
 \begin{eqnarray}
\gamma^+ V\psi -\gamma^- V\psi=0, 
&&
\gamma^+ W\varphi -\gamma^- W\varphi=-\varphi,\quad 
\mbox{if } a\in C_+^{s}(\R^n);\label{WHs1gj}\\
T^+ V\psi -T^- V\psi=\psi, 
&&
T^+ W\varphi -T^- W\varphi=(\partial_\nu a)\varphi,\quad 
\mbox{if } a\in C_+^{\tha}(\R^n).\label{WHs1GaTj}
\end{eqnarray}
\end{theorem}
\begin{proof}
Relations \eqref{VWa} and \eqref{WHs1gjD} along with Theorem~\ref{GrL+} immediately imply jump relations \eqref{WHs1gj}. 
To prove the first jump relation in \eqref{WHs1GaTj}, we generalise to the parametrix-based potentials the arguments from the proof of Lemma 4.1 in \cite{Costabel1988}. 
Let $\psi\in H^{s-\tha}(\partial\Omega)$. From \eqref{3.6d}, we obtain, in the sense of distributions,
\begin{align}
A V\psi=-A\left(\frac{1}{a(y)}\mathbf P_\Delta\gamma^*\psi\right)
&=-\gamma^*\psi+\nabla\cdot\left(\frac{\nabla a}{a }\mathbf P_\Delta\,\gamma^*\psi\,\right)
=-\gamma^*\psi-\nabla\cdot\left((\nabla a) V\psi\right)\quad\mbox{in }\R^n,
\end{align}
where we have taken into account that $\Delta{\mathbf P}_\Delta\,\gamma^*\psi=\gamma^*\psi$. 
Note that  $\gamma^*\psi\in H^{s-2}(\R^n)$ by \eqref{gT*}, and hence 
$\mathbf R_*\gamma^*\psi\in  H^{s-1}_{\rm loc}(\R^n)$ by \eqref{T3.1P3ha*Rn}.
Then, since the operator $A$ is formally self-adjoint, for any test function $\phi\in\mathcal D(\R^n)$ we obtain,
\begin{align}\label{3.71-}
\int_{\R^n}V\psi(y)A\phi(y) dy=\langle A V\psi,\phi\rangle_{\R^n}
=-\langle \gamma^*\psi+\mathbf R_*\gamma^*\psi,\phi\rangle_{\R^n}
=-\langle \psi,\gamma\phi\rangle_{\partial\Omega}- \langle\nabla\cdot\left((\nabla a) V\psi\right),\phi\rangle_{\R^n}.
\end{align}
Note that for $a\in C_+^{\tha}(\R^n)$ and $\psi\in H^{s-\tha}(\partial\Omega)$,  $\ha< s< \tha$, continuity of operator \eqref{VHs1} and Theorem~\ref{GrL+} imply that  $V\psi\in H^{s}_{\rm loc}(\R^n)$ but 
$(\nabla a)V\psi\in H^{\ha+\epsilon}_{\rm loc}(\R^n)$ for some $\epsilon\in(0,1)$. 
Hence, from the second Green identity \eqref{GreenCan} for $v=V\psi$ and $u=\phi$, we have,
\begin{multline}
\int_{\Omega_\pm}V\psi(y)A\phi(y) dy-\langle \tilde A_{\Omega_\pm} V\psi,\phi\rangle_{\Omega_\pm}
=\int_{\Omega_\pm}V\psi(y)A\phi(y) dy
+\langle\widetilde E^{-\ha+\epsilon}_{\Omega_\pm}r_{\Omega_\pm}\nabla\cdot\left((\nabla a)V\psi\right),\phi\rangle_{\Omega_\pm}\\
=\pm\left\langle T^{+}\phi, \gamma^+ V\psi\right\rangle _{\pO}
 \mp\left\langle T^{+}V\psi, \gamma^+\phi\right\rangle _{\pO},\label{3.71}
\end{multline}
Here we employed that $r_{\Omega_\pm}\gamma^*\psi=0$ since $\supp\gamma^*\psi\subset \partial\Omega$.
Let us take into account that $\gamma^+\phi=\gamma^-\phi=\gamma\phi$, $T^+\phi=T^-\phi=T^c\phi$ due to smoothness of $\phi$, while  $\gamma^+ V\psi=\gamma^- V\psi=\gamma V\psi$ by the first relation in \eqref{WHs1gj}.
Moreover,  we also have 
$$
\langle\widetilde E^{-\ha+\epsilon}_{\Omega_\pm}r_{\Omega_\pm}\nabla\cdot\left((\nabla a)V\psi\right),\phi\rangle_{\Omega_\pm}
=\langle r_{\Omega_\pm}\nabla\cdot\left((\nabla a)V\psi\right),\mathring E_{\Omega_\pm}\phi\rangle_{\Omega_\pm}\\
=\pm\langle (\partial_\nu a)\gamma^\pm V\psi,\gamma\phi\rangle_{\partial\Omega}
-\langle(\nabla a)V\psi,\nabla\phi\rangle_{\Omega_\pm}.
$$  
Then summing up \eqref{3.71} for $\Omega$ and $\Omega_-$, we obtain
\begin{align} \label{3.73}
\int_{\R^n}V\psi(y)A\phi(y)dy
&=-\left\langle T^+V\psi-T^-V\psi, \gamma\phi\right\rangle _{\pO}-\langle(\nabla a)V\psi,\nabla\phi\rangle_{\R^n}.
\end{align}
Comparing \eqref{3.73} and \eqref{3.71-}, we obtain 
$\left\langle T^+V\psi-T^-V\psi, \gamma\phi\right\rangle _{\pO}=\langle \psi,\gamma\phi\rangle_{\partial\Omega}$ for arbitrary  $\phi\in\mathcal D(\R^n)$, which implies the first jump relation in \eqref{WHs1GaTj}.

Let us similarly prove the second jump relation in \eqref{WHs1GaTj}.
Let $\varphi\in H^{s-\ha}(\partial\Omega)$. From \eqref{3.7d}, we obtain, in the sense of distributions,
\begin{align}
A W\varphi=-A\left(\frac{1}{a(y)}\mathbf P_\Delta T^{c*}\varphi\right)
&=- T^{c*}\varphi+\nabla\cdot\left(\frac{\nabla a}{a }\mathbf P_\Delta\, T^{c*}\varphi\,\right)
=- T^{c*}\varphi-\nabla\cdot\left((\nabla a) W\varphi\right)
\quad\mbox{in }\R^n,
\end{align}
where we have taken into account that $\Delta{\mathbf P}_\Delta\, T^{c*}\varphi= T^{c*}\varphi$. 
Then for any test function $\phi\in\mathcal D(\R^n)$ we obtain,
\begin{multline}\label{3.71-W}
\int_{\R^n}W\varphi(y)A\phi(y) dy=\langle A W\varphi,\phi\rangle_{\R^n}
=-\langle  T^{c*}\varphi+\nabla\cdot\left((\nabla a) W\varphi\right),\phi\rangle_{\R^n}\\
=-\langle \varphi,T^c\phi\rangle_{\partial\Omega} + \langle(\nabla a) W\varphi,\nabla\phi\rangle_{\R^n}.
\end{multline}
Note that for $a\in C_+^{\tha}(\R^n)$ and $\varphi\in H^{s-\ha}(\partial\Omega)$,  $\ha< s< \tha$, continuity of operators \eqref{WHs1}, \eqref{WHs1-} and Theorem~\ref{GrL+} imply that 
$r_{\Omega}W\varphi\in H^{s}(\Omega)$, $r_{\Omega_-}W\varphi\in H^{s}_{\rm loc}(\overline\Omega_-)$ but $(\nabla a)r_{\Omega}W\varphi\in H^{\ha+\epsilon}(\Omega)$, 
$(\nabla a)r_{\Omega_-}W\varphi\in H^{\ha+\epsilon}_{\rm loc}(\overline\Omega_-)$ for some $\epsilon\in(0,1)$. 
Hence from the second Green identity \eqref{GreenCan} for $v=W\varphi$ and $u=\phi$, we have,
\begin{align}
\int_{\Omega_\pm}W\varphi(y)\, A\phi(y) dy-\langle \tilde A_{\Omega_\pm} W\varphi,\phi\rangle_{\Omega_\pm}
&=\int_{\Omega_\pm}W\varphi(y)\, A\phi(y) dy
+\langle\widetilde E^{-\ha+\epsilon}_{\Omega_\pm}r_{\Omega_\pm}\nabla\cdot\left((\nabla a)W\varphi\right),\phi\rangle_{\Omega_\pm} \nonumber\\
&=\pm\left\langle T^{+}\phi, \gamma^+ W\varphi\right\rangle _{\pO}
 \mp\left\langle T^{+}W\varphi, \gamma^+\phi\right\rangle _{\pO}.\label{3.71W}
\end{align}
Here we employed that $r_{\Omega_\pm}T^{c*}\varphi=0$ since $\supp T^{c*}\varphi\subset \partial\Omega$.
Let us also take into account that $\gamma^+\phi=\gamma^-\phi=\gamma\phi$, $T^+\phi=T^-\phi=T^c\phi$ due to smoothness of $\phi$, while  $\gamma^+ W\varphi-\gamma^- W\varphi=-\varphi$ by the second relation in \eqref{WHs1gj}.
Moreover,  we also have 
$$
\langle\widetilde E^{-\ha+\epsilon}_{\Omega_\pm}r_{\Omega_\pm}\nabla\cdot\left((\nabla a)W\varphi\right),\phi\rangle_{\Omega_\pm}
=\langle r_{\Omega_\pm}\nabla\cdot\left((\nabla a)W\varphi\right),\mathring E_{\Omega_\pm}\phi\rangle_{\Omega_\pm}\\
=\pm\langle (\partial_\nu a)\gamma^\pm W\varphi,\gamma\phi\rangle_{\partial\Omega}
-\langle(\nabla a)W\varphi,\nabla\phi\rangle_{\Omega_\pm}.
$$ 
Then summing up \eqref{3.71W} for $\Omega$ and $\Omega_-$, we obtain
\begin{multline} \label{3.73W}
\int_{\R^n}W\varphi(y)A\phi(y)dy-\langle (\partial_\nu a)\varphi,\gamma\phi\rangle_{\partial\Omega}
-\langle(\nabla a)W\varphi,\nabla\phi\rangle_{\R^n}\\
=-\left\langle T^c\phi, \varphi\right\rangle _{\pO}
-\left\langle T^+W\varphi-T^-W\varphi, \gamma\phi\right\rangle _{\pO}.
\end{multline}
Comparing \eqref{3.73W} and \eqref{3.71-W}, we obtain 
$\left\langle T^+W\varphi-T^-W\varphi, \gamma\phi\right\rangle _{\pO}
=\langle (\partial_\nu a)\varphi,\gamma\phi\rangle_{\partial\Omega}$ for arbitrary  $\phi\in\mathcal D(\R^n)$, which implies the second jump relation in \eqref{WHs1GaTj}.
\end{proof}

Theorem \ref{T3.1s0}(iii) and the first relation in \eqref{WHs1gj} imply the following assertion.
\begin{corollary}\label{T3.3}
Let  $\partial\Omega$ be a compact Lipschitz boundary,  $\ha< s< \tha$.
The following operators are continuous.
 \begin{align}
&\mathcal V:=\gamma^+ V=\gamma^- V  : H^{s-\frac{3}{2}}(\pO) \to H^{s-\ha}(\partial\Omega),\quad 
a\in C_+^{s}(\overline\Omega_\pm);\label{3.11}\\
&\mathcal W:=\ha(\gamma^+ W+\gamma^- W)  :  H^{s-\ha}(\pO)\to H^{s-\ha}(\partial\Omega),\quad 
a\in C_+^{s}(\overline\Omega_\pm);\label{3.12}\\
&\mathcal W':=\ha(T^+V +T^-V)  : H^{s-\frac{3}{2}}(\pO) \to H^{s-\tha}(\partial\Omega),\quad 
a\in C_+^{\tha}(\overline\Omega_\pm);\label{3.13}\\
&\mathcal L:=\ha(T^+W +T^-W)   :  H^{s-\ha}(\pO)\to  H^{s-\tha}(\partial\Omega),\quad a\in C_+^{\tha}(\overline\Omega_\pm).\label{3.14}
 \end{align}
\end{corollary}
For the case of smooth boundary, the boundary operators defined in Corollary~\ref{T3.3} (cf. \cite[Eq. (7.3)]{McLean2000} for the fundamental solution - based potentials on Lipschitz domains) correspond to the  boundary integral (pseudodifferential) operators
of direct surface values of the single layer potential, the double layer potential $\W$, and the co-normal derivatives of the single layer potential $\Wp$ and of the double layer potential, cf. \cite[Eq. (3.6)-(3.8)]{CMN-1} for the parametrix-based potentials on smooth domains. See also \cite[Theorems 7.3, 7.4]{McLean2000} about integral representations on Lipschitz domains of the boundary operators associated with the layer potentials, based on fundamental solutions.

If $a=1$, we will equip the operators defined in Corollary~\ref{T3.3}  with the subscript $\Delta$. Then under the hypotheses of Corollary~\ref{T3.3} we have (cf. \cite[Eq. (3.10)-(3.13)]{CMN-1} for the potentials on smooth domains),
\begin{align}
\label{VWab1}&\mathcal V g=\frac{1}{a}\mathcal V_{_\Delta} g,\quad \hspace{5.2em}
\mathcal W g=\frac{1}{a}\mathcal W_{_\Delta}
(ag),\\
& \mathcal W\,^\prime g= {\mathcal W^{\,\prime}_{_\Delta}}g -\frac{\partial_\nu a}{a}\,\mathcal V_{_\Delta}g
,\quad
 {\mathcal L}g ={\mathcal L}_{_\Delta}(ag) -\frac{\partial_\nu a}{a}\, \mathcal W_{_\Delta} (ag).
 \label{VWab4}
\end{align}
Indeed, relations \eqref{VWab1} immediately follow from \eqref{3.11}, \eqref{3.12}, and \eqref{VWa}. 
Further, $T^+ Vg=T^+ (\frac{1}{a}V_\Delta g)$. Let $\{v_k\}\subset\mathcal D(\overline\Omega)$ be a sequence such that $\|v_k-V_\Delta g\|_{H^{s,-\ha}(\Omega;\Delta)}\to 0$ as $k\to\infty$, which implies that 
also $\|\frac{1}{a}v_k-V g\|_{H^{s,-\ha}(\Omega;A)}\to 0$ as $k\to\infty$. 
Then, cf. \cite[Lemma 6.10]{MikJMAA2013},
$$
T^+ Vg=\lim_{k\to\infty}T^c\left(\frac{1}{a}v_k\right)=\lim_{k\to\infty}aT^c_\Delta\left(\frac{1}{a}v_k\right)
=\lim_{k\to\infty}\left(\partial_\nu v_k - \frac{\partial_\nu a}{a}\gamma^+v_k\right)
=T^+_\Delta V_\Delta g -\frac{\partial_\nu a}{a}\gamma^+V_\Delta g.
$$
Similarly, $T^- Vg=T^-_\Delta V_\Delta g - \frac{\partial_\nu a}{a}\gamma^-V_\Delta g$, which together with \eqref{3.13} implies the first relation in \eqref{VWab4}. 
The second relation in \eqref{VWab4} is proved by a similar arguments.

Employing definitions \eqref{3.11}-\eqref{3.14}, the jump properties \eqref{WHs1gj}-\eqref{WHs1GaTj} can be re-written as follows for $\psi\in H^{s-\tha}(\pO)$, and  $\varphi\in H^{s-\ha}(\pO)$,  $\ha< s< \tha$, 
\begin{align}
&\gamma^\pm V\psi= {\cal V}\psi, \quad\hspace{4.5em}
\gamma^\pm W\varphi= \mp \ha\,\varphi  + {\cal W}\varphi,\quad \hspace{2em}
\mbox{if } a\in C_+^{s}(\R^n);\label{3.8}
\\
&T^\pm V\psi= \pm\ha\,\psi +
{\Wp}\psi,\quad
T^\pm W\varphi= \pm \ha\,(\partial_\nu a)\varphi  + {\cal L}\varphi,\quad 
\mbox{if } a\in C_+^{\tha}(\R^n).\label{3.11a}
\end{align}

\section{The third Green identity and integral relations}\label{G3G}
We will apply in this section some limiting procedures 
to obtain the parametrix-based third Green identities. 

\begin{theorem}\label{3rdGA}
Let  $\Omega$ be a bounded Lipschitz domain, $u\in {H}^{s}(\Omega)$, $\ha< s< \tha$, and 
$a\in C_+^{s}(\overline\Omega)$.

(i) The following {\bf generalised third Green identity} holds,
\begin{equation}
u+{\cal R}u +W\gamma^+u=\P\check{A}u
 \quad \mbox{in}\ \Omega,\label{4.G3til}
\end{equation}
where, by \eqref{Ltil}, \eqref{checkA},
\begin{multline}\label{LPtil}
 \P\check{A}u(y):=   \langle \check{A}u,P(\cdot,y) \rangle_{\Omega} =-\check\E_{\Omega}(u,P(\cdot,y))
=-\langle \widetilde E_{\Omega}^{s-1}(a\nabla u),\nabla P(\cdot,y)\rangle_{\Omega}\\
=-\frac{1}{a(y)}\nabla\cdot\P_\Delta \widetilde E_{\Omega}^{s-1}(a\nabla u)(y),\quad\mbox{a.e. }y\in\Omega,
\end{multline}
and particularly, if $s=1$,
\begin{eqnarray}\label{LPtil1}
 \P\check{A}u(y) =-\int_{\Omega} a(x) \;\nabla u(x)\cdot\nabla_x P(x,y)\,dx,\quad\mbox{a.e. }y\in\Omega.\qquad
\end{eqnarray}

(ii) Moreover, if $Au=r_{\Omega}\tilde f$ in $\Omega$, where $\tilde f\in \widetilde{H}^{s-2}(\Omega)$,
then the {\bf generalised third Green identity} takes form,
\begin{eqnarray}
u+{\cal R}u - VT^+(\tilde f;u)  +W\gamma^+u=
 {\cal P}\tilde f  &&\mbox{in }\Omega. \qquad\label{4.2Gen}
 \end{eqnarray}
\end{theorem} 
\begin{proof}
(i) Let first $u\in\mathcal D(\overline \Omega)$. Let $y\in\Omega$,  $B_\epsilon(y)\subset \Omega$ be a ball centred in $y$ with sufficiently small radius $\epsilon$, and $\Omega_\epsilon:=\Omega\setminus\overline{B_\epsilon}(y)$. 
For any fixed $y$, evidently, 
$P(\cdot,y)=\frac{1}{a(y)}P_\Delta(\cdot,y)\in \mathcal D(\overline{\Omega_\epsilon})\subset H^{1,0}(A;\Omega_\epsilon)$ 
and has the coinciding classical and canonical conormal derivatives on $\pO_\epsilon$. 
Then from  the first Green identity \eqref{Greentilde} employed for $\Omega_\epsilon$ with $v=P(\cdot,y)$ we obtain
 \begin{align}\label{4.G3tile}
  - \langle T^{+}_x P(\cdot,y),\gamma^+u(x)\rangle_{\partial B_\epsilon(y)} 
  - \langle T^{+} P(\cdot,y),\gamma^+u\rangle_\pO
  +\langle R(\cdot,y),u\rangle_{\Omega_\epsilon} 
  =-\langle \nabla P(\cdot,y), a \nabla u\rangle_{\Omega_\epsilon}
  \end{align}
Since
$$
 \lim_{\epsilon\to 0}\langle T^{+}_x P(\cdot,y),\gamma^+u(x)\rangle_{\partial B_\epsilon(y)}=
\frac{1}{a(y)} \lim_{\epsilon\to 0}\int_{\partial B_\epsilon(y)} [\partial_{\nu(x)} P_\Delta(x,y)]a(x)\gamma^+u(x) dS(x)=-u(y),
$$ 
by passing to the limits as $\epsilon\to 0$, equation \eqref{4.G3tile} reduces to the third Green identity \eqref{4.G3til} for any $u\in \mathcal D(\overline \Omega)$.
Taking into account density of $\mathcal D(\overline \Omega)$ in $H^{s}(\Omega)$, and the mapping properties of the volume potentials \eqref{T3.1P1}, \eqref{T3.1P3has} in Theorem \ref{T3.1P} and of the double layer potential \eqref{WHs1} in Theorem \ref{T3.1s0}(i), we obtain that \eqref{4.G3til}-\eqref{LPtil} hold true also for any $u\in H^{s}(\Omega)$, $\ha< s< \tha$, in the sense of $H^{s}(\Omega)$, which implies also \eqref{LPtil1} for $s=1$.

(ii) 
Let $\{u_k\}\in\mathcal D(\overline\Omega)$ be a sequence converging to $u$ in $H^{1}(\Omega)$.  
By \eqref{LPtil}, \eqref{LPtil1} and \eqref{Tgen}, we have,
 \begin{multline} \label{LPtilc}
  \P\check{A}u_k(y)
=-\lim_{\epsilon\to 0}\int_{\Omega_\epsilon} a(x) \nabla u_k(x) \cdot\nabla_x P(x,y)\,dx 
=-\lim_{\epsilon\to 0} \check{\mathcal E}_{\Omega_\epsilon}(u_k,P(\cdot, y))\\
=\lim_{\epsilon\to 0}\left[\int_{\Omega_\epsilon} (\tilde A_{\Omega}u_k)(x) P(x,y)\,dx 
 - \int_{\partial B_\epsilon(y)} P(x,y)T^+u_k(x) dS(x) - \int_\pO P(x,y)T^+u_k(x)  dS(x)\right]\\  
={\cal P}\tilde A_{\Omega}u_k(y) + VT^+u_k(y).\quad
 \end{multline} 
Let now $\tilde f_k:= \widetilde E_{\Omega}^{s-2}r_{\Omega}(\tilde A_{\Omega}u_k-\check A_{\Omega}u) 
+ \tilde f,$
where $\widetilde E_{\Omega}^{s-2}:H^{s-2}(\Omega)\to \widetilde H^{s-2}(\Omega)$ is a (non-unique) continuous extension operator, which exists by \cite[Theorem 2.16]{MikJMAA2011}.
Since $r_{\Omega}\check A_{\Omega}u = r_{\Omega}\tilde f$, we obtain 
$r_{\Omega}\tilde f_k=r_{\Omega}\tilde A_{\Omega}u_k=r_{\Omega}\check A_{\Omega}u_k$.
Hence $r_{\Omega}\tilde A_{\Omega}u_k-r_{\Omega}\check A_{\Omega}u
=r_{\Omega}\check A_{\Omega}(u_k-u)\to 0$ in $H^{s-2}(\Omega)$ and $\tilde f_k\to\tilde f$ in $\widetilde H^{s-2}(\Omega)$, as $k\to\infty$.
Then by \eqref{LPtilc}, \eqref{3.6d} and \eqref{Tgentil} we obtain,
\begin{align*}
\P\check{A}u_k={\cal P}\tilde A_{\Omega}u_k+ VT^+u_k
={\cal P}\tilde A_{\Omega}u_k+VT^+(\tilde f_k;u_k)+\mathcal P(\tilde f_k-\tilde A_{\Omega}u_k)
=VT^+(\tilde f_k;u_k)+\mathcal P\tilde f_k.
\end{align*}
Passing to the limits as $k\to\infty$, we obtain 
$
\P\check{A}u(y)={\cal P}\tilde f + VT^+(\tilde f;u),
$
which substitution to \eqref{4.G3til} gives \eqref{4.2Gen}.
\end{proof}

For some functions $\tilde f$, $\Psi$, $\Phi$, let us consider a more
general "indirect" integral relation, associated with
\eqref{4.2Gen},
\begin{eqnarray}
 u+{\cal R}u - V\Psi +W\Phi &=&
 {\cal P}\tilde f\ \mbox{ in }\Omega. \label{4.2nd}
\end{eqnarray}
The following two lemmas extend Lemma 4.1 from \cite{CMN-1}, where
the corresponding assertion was proved for $\tilde f\in L_2(\Omega)$, $s=1$,  $a\in C^\infty(\Omega)$ and the infinitely smooth boundary.

\begin{lemma}\label{IDequivalenceGen}
Let $\ha< s< \tha$ and $a\in C_+^{s}(\overline \Omega)$. 
Let $u\in H^s(\Omega)$, $\Psi\in H^{s-\tha}(\pO)$, $\Phi\in
H^{s-\ha}(\pO)$, and $\tilde f\in \widetilde{H}^{s-2}(\Omega)$ satisfy
(\ref{4.2nd}). Then
\begin{align}
\label{2.6'gen}Au= r_{\Omega} \tilde f& \text{ in }  \Omega,
\\
\label{difference}
 V(\Psi -T^+(\tilde f;u) ) -  W(\Phi - \gamma^+u) = 0& \text{ in }  \Omega.
\end{align}
\end{lemma}
\begin{proof} 
Subtracting  \eqref{4.2nd} from identity \eqref{4.G3til}, we obtain
\begin{equation}
\label{4.12} V\Psi-W(\Phi-\gamma^+u) ={\cal P}[\check{A}u-\tilde f]
\; \text{in }  \Omega.
\end{equation}
Multiplying equality \eqref{4.12} by $a$, applying the Laplace
operator $\Delta$ and taking into account \eqref{DV,DW=0},
\eqref{DeltaPg}, we get
$r_{\Omega}\tilde f=r_{\Omega}(\check{A}u)=Au$ in $\Omega.$
This means $\tilde f$ is an extension of the distribution $Au\in H^{s-2}(\Omega)$ to
$\widetilde{H}^{s-2}(\Omega)$, and $u$ satisfies  \eqref{2.6'gen}. 
Then \eqref{Tgen} implies
 \begin{eqnarray}
 \label{Tgen1}
{\cal P}[\check{A}u-\tilde f](y)&=&\langle \check{A}u-\tilde f,
P(\cdot,y)\rangle_{\Omega}
=-\langle
 T^+(\tilde f;u)\,,\,P(\cdot,y)
\rangle _{\pO} =VT^+(\tilde f;u)(y), \quad y\in\Omega.\qquad
\end{eqnarray}
Substituting \eqref{Tgen1} into \eqref{4.12} leads to
\eqref{difference}.
\end{proof}

For  $\ha< s< \tha$ and $a\in C_+^s(\overline \Omega)$, and $g\in H^{s-1}(\Omega)$ let us introduce the operator $A^\nabla$ as
\be\label{Aring}
A^\nabla  g:=-\nabla\cdot\mathring E_{\Omega}\,(g\nabla a).
\ee
\begin{lemma}\label{AringMap}
Let $\ha< s< \tha$. 

(i) If $a\in C_+^{|s-1|+1}(\overline \Omega)$ the following operator is continuous,
\begin{align}\label{AringMap1}
A^\nabla: H^{s-1}(\Omega)\to \widetilde H^{s-2}(\Omega).
\end{align} 

(ii) If $a\in C_+^{s}(\overline \Omega)$, the following operator is continuous and compact, 
\begin{align}\label{AringMap2}
A^\nabla: H^{s}(\Omega)\to \widetilde H^{s-2}(\Omega).
\end{align}
\end{lemma}
\begin{proof}
(i) If $a\in C_+^{|s-1|+1}(\overline \Omega)$, then $\nabla a\in C_+^{|s-1|}(\overline \Omega)$, and by Theorem~\ref{GrL+}, $\nabla a$ is a multiplier in $H^{s-1}(\Omega)$, which implies continuity of operator \eqref{AringMap1}.

(ii) For $1\le s<\tha$, we have $s=|s-1|+1$, which by item (i) implies continuity of operator \eqref{AringMap1} and thus continuity and compactness of operator \eqref{AringMap2}.

For $\ha< s<1$, we need an estimate of the norm $\|g\,\nabla a\|_{H^{s-1}(\Omega)}$.
First, by Definition~\ref{C+mu} the inclusion $ a\in C_+^{s}(\overline\Omega)$ implies that there exists 
$t\in (s, 1)$ such that,
$a\in C^{0,t}(\overline\Omega)=B^t_{\infty,\infty}(\Omega)=F^t_{\infty,\infty}(\Omega)$, see e.g. Proposition in  \cite[Section 2.1.2]{Runst-Sickel1996}, and hence 
$\nabla a\in  
F^{t-1}_{\infty,\infty}(\Omega)$. 
Then, by Theorems 1  from \cite[Section 4.4.3]{Runst-Sickel1996},
\begin{align}\label{3.48aa} 
\|g\,\nabla a\|_{F^{t-1}_{2,\infty}(\Omega)}
\le C\|\nabla a\|_{F^{t-1}_{\infty,\infty}(\Omega)}\, \|g\|_{H^{\sigma}(\Omega)}
\le C\|a\|_{C^{0,t}(\overline\Omega)}\, \|g\|_{H^{\sigma}(\Omega)},\quad 
\forall\ 
\sigma\in(1-t, s).
\end{align}
On the other hand, by \eqref{Rs-1tos}, item (ii) of Proposition from \cite[Section 2.2.1]{Runst-Sickel1996}, and \eqref{3.48aa}, we obtain
$$
\|{A^\nabla}\,g\|_{\widetilde H^{s-2}(\Omega)}\le
c_3\|g\,\nabla a\|_{H^{s-1}(\Omega)}=c_3\|g\,\nabla a\|_{F^{s-1}_{2,2}(\Omega)}
\le C_1\|g\,\nabla a\|_{F^{t-1}_{2,\infty}(\Omega)}
\le C_1 C\|a\|_{C^{0,t}(\overline\Omega)}\, \|g\|_{H^{\sigma}(\Omega)}.
$$
Thus the operator $A^\nabla: H^{\sigma}(\Omega)\to {\widetilde H^{s-2}(\Omega)}$ is continuous, which implies continuity and, by the Rellich compact embedding theorem, also compactness of operator  \eqref{AringMap2} for $\ha< s<1$.
\end{proof}

In accordance with notation \eqref{checkA}, let us also denote
$\check\Delta_{\Omega}g:=\nabla\cdot\mathring E_{\Omega}\, r_{\Omega}\nabla g$.

Let us now discuss the trace and two forms of co-normal derivative associated with equation  \eqref{4.2nd}. 
\begin{lemma}\label{L4.3}
(i) Under the hypotheses of Lemma~\ref{IDequivalenceGen}, 
\begin{align}
\gamma^+u +\gamma^+{\cal R}u - \V \Psi -\ha\Phi +\W \Phi&=
 \gamma^+{\cal P}\tilde f
 \text{on}\partial\Omega, 
\label{4.2u+gen}\\
T^+(\tilde f;u)+T^+_\Delta(A^\nabla u;{a\cal R} u) -\ha\Psi - {\mathcal W}'_\Delta\Psi +\mathcal L_\Delta(a \Phi) &=
T^+_\Delta(\tilde f;\P_\Delta\tilde f)
 \text{on}  \partial\Omega. 
\label{4.2T+genD} 
\end{align}
(ii) If, moreover,  $a\in C_+^{\tha}(\overline\Omega)$ then  
\begin{align}
 T^+(\tilde f;u)+T^+{\cal R}u -\ha\Psi - {\Wp}\Psi +T^+W \Phi &=
  T^+(\tilde f+\mathring E\,{\mathcal R}_*\tilde f;{\cal P}\tilde f)  \text{on}  \partial\Omega,\quad \label{4.2T+gen}
\end{align}
where ${\cal R}_*$ is defined in \eqref{4.9R}, \eqref{3.d1}. 
\end{lemma}
\begin{proof}
(i) Equation \eqref{4.2u+gen} is implied by \eqref{4.2nd}, \eqref{3.8}.

To prove \eqref{4.2T+genD}, let us first multiply \eqref{4.2nd} by $a$ to obtain
\begin{eqnarray}
 - V_\Delta\Psi +W_\Delta(a\Phi) =
 {\cal P}_\Delta\tilde f - au-a{\cal R}u\ \mbox{ in }\Omega. \label{4.2nd1}
\end{eqnarray}
Since ${\Delta} \{ - V_\Delta\Psi +W_\Delta(a\Phi)\}=0$, for the both sides of \eqref{4.2nd1} the canonical conormal derivative, $T^+_\Delta$, is well defined,
\begin{eqnarray}
 - T^+_\Delta V_\Delta\Psi +T^+_\Delta W_\Delta(a\Phi) &=&
 T^+_\Delta({\cal P}_\Delta\tilde f - au-a{\cal R}u) \label{4.2nd2}
\end{eqnarray}
and
\begin{eqnarray}
 T^+_\Delta({\cal P}_\Delta\tilde f - au-a{\cal R}u)
= -(\gamma^{-1})^*\check\Delta_{\Omega}({\cal P}_\Delta\tilde f - au-a{\cal R}u).\label{4.2nd3}
\end{eqnarray}
Note that by the second equality in \eqref{3.d1},
\begin{align}\label{4.18a}
\Delta(a{\cal R}u)=-\nabla\cdot\left[ \Delta{\mathcal P}_\Delta\,(u\,\nabla a)\right]
=-\nabla\cdot(u\,\nabla a)=r_{\Omega}A^\nabla u\ \mbox{ in }\Omega
\end{align}
which implies that $A^\nabla u\in \widetilde H^{-1}(\Omega)$ is an extension of $\Delta(a{\cal R}u)\in H^{-1}(\Omega)$. 
Further, cf. \eqref{checkA},
\begin{eqnarray*}
 \check\Delta_{\Omega}(au)=\nabla\cdot\mathring E_{\Omega}\,r_{\Omega}\nabla(ua)
 =\nabla\cdot\mathring E_{\Omega}\,r_{\Omega}(u\nabla a)+\nabla\cdot\mathring E_{\Omega}\,r_{\Omega}(a\nabla u)
 =-A^\nabla u+\check{A}u\ \mbox{ in }\R^n. 
\end{eqnarray*}
Then 
\begin{eqnarray*}
-\check\Delta_{\Omega}({\cal P}_\Delta\tilde f - au-a{\cal R}u)
=\tilde f -\check\Delta_{\Omega}{\cal P}_\Delta\tilde f -(\tilde f -\check{A}u)
-[A^\nabla u-\check\Delta_{\Omega}(a{\cal R}u)]\ \mbox{ in }\R^n
\end{eqnarray*}
and
\begin{eqnarray*}
T^+_\Delta({\cal P}_\Delta\tilde f - au-a{\cal R}u)
=  T^+_\Delta(\tilde f;{\cal P}_\Delta\tilde f) - T^+(\tilde f;u)
-T^+_\Delta(A^\nabla u;a{\cal R}u). 
\end{eqnarray*}
Substituting this in \eqref{4.2nd2}, we obtain
\bes 
T^+(\tilde f;u)+T^+_\Delta(A^\nabla u;{a\cal R} u) - T^+_\Delta V_\Delta\Psi + T^+_\Delta W_\Delta(a\Phi)
=T^+_\Delta(\tilde f;\P_\Delta\tilde f)\ \text{on} \partial\Omega.
\ees
Taking into account jump relation \eqref{WHs1GaTjD} and \eqref{3.13}, we arrive at \eqref{4.2T+genD}.

(ii) To prove \eqref{4.2T+gen}, let us first remark that
\begin{equation}\label{LPf}
 A\P\tilde f=\tilde f+{\cal R}_*\tilde f\  \text{in}\ {\Omega},
\end{equation}
which implies, due to \eqref{2.6'gen},
$
    A(\P\tilde f-u)={\cal R}_*\tilde f\ \text{in}\ {\Omega},
$
where ${\cal R}_*$ is defined in \eqref{4.9R}, \eqref{3.d1} and since $a\in C_+^{\tha}(\overline\Omega)$, we obtain by \eqref{T3.1P3ha**} that ${\cal R}_*\tilde f\in H^{\sigma}(\Omega)$ for some $\sigma>-\ha$. 
Then $A(\P\tilde f-u)$ can be canonically extended to 
$\tilde A(\P\tilde f-u)=\mathring E_{\Omega} {\mathcal R}_*\tilde f\in\s{H}^\sigma({\Omega})\subset\s{H}^{s-2}({\Omega})$. 
This implies that there exists a
canonical co-normal derivative of $(\P\tilde f-u)$, for which, due to
\eqref{Tcandef} and \eqref{checkA}, we have
 \begin{align}
 T^+(\P\tilde f-u)=&(\gamma^{-1})^*[ \tilde A(\P\tilde f-u)-\check{A}\P\tilde f +\check{A}u]
 = (\gamma^{-1})^*[\mathring E_{\Omega}\,{\mathcal R}_*\tilde f-\check{A}\P\tilde f +\check{A}u]\nonumber\\
 =&(\gamma^{-1})^*[ \mathring E_{\Omega}\,{\mathcal R}_*\tilde f+\tilde f-\tilde f-\check{A}\P\tilde f +\check{A}u]
 =(\gamma^{-1})^*[ \tilde f+\mathring E_{\Omega}\,{\mathcal R}_*\tilde f-\check{A}\P\tilde f +\check{A}u-\tilde f]\nonumber\\
 =& T^+(\tilde f+\mathring E_{\Omega}\,{\mathcal R}_*\tilde f,\P\tilde f) - T^+(\tilde f,u),
 \label{TPf-u}
 \end{align}
where  $\tilde f+\mathring E_{\Omega}\,{\mathcal R}_*\tilde f\in\s{H}^{s-2}({\Omega})$ is an
extension of $A\P\tilde f$ associated with \eqref{LPf}. 
From \eqref{4.2nd} we have ${\cal P}\tilde f-u={\cal R}u - V\Psi +W\Phi$
in ${\Omega}$. Substituting this in the left hand side of
\eqref{TPf-u} and taking into account Theorem~\ref{T3.1s0}(ii) and jump relation \eqref{WHs1GaTj} and \eqref{3.13}, we
arrive at \eqref{4.2T+gen}.
\end{proof} 

Note that unlike \eqref{4.2T+genD}, the co-normal derivative form \eqref{4.2T+gen} of relation \eqref{4.2nd} is written without referring to the corresponding constant-coefficient potentials. 

\begin{remark}
If $\ha<s<\tha$ but $\tilde f\in \widetilde H^{-1/2}(\Omega)\subset \widetilde H^{s-2}(\Omega)$, then 
${\cal P}_\Delta\tilde f\in H^{s,-1/2}(\Omega,\Delta)$ and
\begin{equation}\label{TfPTP}
T^+_\Delta(\tilde f;\P_\Delta\tilde f)=T^+_\Delta{\cal P}_\Delta\tilde f.
\end{equation}

(i) Furthermore, if the hypotheses of Lemma~\ref{IDequivalenceGen} are satisfied, then \eqref{2.6'gen} implies $u\in H^{s,-1/2}(\Omega,A)$ and $ T^+(\tilde f;u)= T^+(\tilde Au;u)= T^+u$.  
Henceforth, \eqref{4.2T+genD} takes a simpler form, 
\begin{align}\label{4.22a}
 T^+u + T^+_\Delta(A^\nabla u;{a\cal R} u) -\ha\Psi - {\mathcal W}'_\Delta\Psi +\L^+_\Delta(a \Phi) =
 T^+_\Delta\P_\Delta\tilde f \text{ on } \partial\Omega.
\end{align}
If, in addition, $au\in H^{s,-1/2}(\Omega,\Delta)$, then by \eqref{4.18a},
$$
\Delta (a\mathcal R u)=r_{\Omega}A^\nabla u=-\nabla\cdot(u\,\nabla a)
=Au-\Delta(au)\in \widetilde H^{-\ha}_\bullet(\Omega).
$$
Hence the canonical conormal derivative $T^+_\Delta({a\cal R} u)$ is well defined and
\begin{multline}\label{4.23b}
T^+_\Delta(A^\nabla u;{a\cal R} u)
=(\gamma^{-1})^*[A^\nabla u-\check\Delta_{\Omega}(a\mathcal R u)]
=(\gamma^{-1})^*[A^\nabla u-\tilde{\Delta}(a\mathcal R u)] + T^+_\Delta({a\cal R} u)\\
=(\gamma^{-1})^*[-\nabla\cdot\mathring E_{\Omega}\,(u\nabla a)+\mathring E_{\Omega}\nabla\cdot\,(u\nabla a)]+ T^+_\Delta({a\cal R} u)\\
=(\gamma^{-1})^*[\check A u- \tilde Au]+(\gamma^{-1})^*[-\check\Delta(au)+\tilde\Delta(au)]+ T^+_\Delta({a\cal R} u)
= -T^+u + T^+_\Delta(au) + T^+_\Delta({a\cal R} u).
\end{multline}
This reduces \eqref{4.22a} to the relation
\begin{align}\label{4.22b}
 T^+_\Delta(au) + T^+_\Delta(a\mathcal R u) -\ha\Psi - {\mathcal W}'_\Delta\Psi +\L^+_\Delta(a \Phi) =
 T^+_\Delta\P_\Delta\tilde f \text{ on } \partial\Omega
\end{align}
with only canonical normal  derivatives associated with the Laplace operator.

(ii)  If the hypotheses of Lemma~\ref{IDequivalenceGen} are satisfied and moreover,  $a\in C_+^{\tha}(\overline\Omega)$, then by \eqref{T3.1P1ha} and \eqref{T3.1P3ha**}, 
${\cal P}\tilde f\in H^{s,-\ha}(\Omega;A)$ and ${\mathcal R}_*\tilde f\in \widetilde H^{-1/2}(\Omega)$ implying
$T^+(\tilde f+\mathring E_{\Omega}\,{\mathcal R}_*\tilde f;{\cal P}\tilde f)=T^+({\cal P}\tilde f)$.
Henceforth, \eqref{4.2T+gen} reduces to the relation
\begin{align*}
 T^+u + T^+{\cal R}u -\ha\Psi - {\Wp}\Psi +T^+W \Phi &=
  T^+{\cal P}\tilde f  \text{on}  \partial\Omega
\end{align*}
with only canonical conormal derivatives associated with the operator $A$.
\end{remark}

\begin{remark}\label{R4.5}
(i) Let the hypotheses of Lemma~\ref{IDequivalenceGen} be satisfied and a sequence $\{\tilde f_j\}\in \widetilde H^{-\ha}(\Omega)$ converge to $\tilde f$ in  $\widetilde H^{s-2}(\Omega)$. By the continuity of operators \eqref{T3.1P1} and \eqref{T3.1P3}, estimate \eqref{estimate} and relation \eqref{TfPTP} for $\tilde f_j$, we obtain that 
$$
T^+_\Delta(\tilde f;\P_\Delta\tilde f)
=\lim_{j\to\infty}T^+_\Delta{\cal P}_\Delta\tilde f_j\quad \mbox{in } H^{s-\tha}(\partial\Omega),
$$
 cf. also Theorem~\ref{Tseq}.

(ii)  If, moreover,  $a\in C_+^{\tha}(\overline\Omega)$, then, similarly, 
$$
T^+(\tilde f+\mathring E_{\Omega}\,{\mathcal R}_*\tilde f,{\cal P}\tilde f)
=\lim_{j\to\infty}T^+(\tilde f_j+\mathring E_{\Omega}\,{\mathcal R}_*\tilde f_j,{\cal P}\tilde f_j)
=\lim_{j\to\infty}T^+{\cal P}\tilde f_j.
$$
\end {remark}


Lemma \ref{L4.3}(ii) and the third Green identity \eqref{4.2Gen}
imply the following assertion.
\begin{corollary}\label{IDequivGenDir}
Let $\Omega$ be a bounded simply-connected Lipschitz domain,  $\ha< s< \tha$, $a\in C_+^{s}(\overline \Omega)$ and
$u\in H^{s}(\Omega)$, $\tilde f\in \widetilde{H}^{s-2}(\Omega)$ be such that $Au=r_{\Omega}\tilde f$ in $\Omega$.

(i) Then
 \begin{align}
\frac{1}{2}&\gamma^+u +\gamma^+{\cal R}u - \V T^+(\tilde f;u)  +\W \gamma^+u=
 \gamma^+{\cal P}\tilde f \mbox{ on }\pO, 
 \label{4.2u+Gen}\\
\frac{1}{2}&T^+(\tilde f,u)+T^+_\Delta(A^\nabla u;{a\cal R} u) 
- {\mathcal W}'_\Delta T^+(\tilde f,u) +\mathcal L_\Delta(a \gamma^+u) =
T^+_\Delta(\tilde f;\P_\Delta\tilde f) \mbox{ on } \pO.
\label{4.2T+GenD}
\end{align}

(ii) If, moreover,  $a\in C_+^{\tha}(\overline\Omega)$, then  
\begin{align}
 \ha T^+(\tilde f,u)+T^+{\cal R}u  - {\Wp}T^+(\tilde f,u) +T^+W \gamma^+u &=
  T^+(\tilde f+\mathring E_{\Omega}\,{\mathcal R}_*\tilde f,{\cal P}\tilde f)  \text{ on }  \partial\Omega,\quad \label{4.2T+Gen}
\end{align}
where ${\cal R}_*$ is defined in \eqref{4.9R}, \eqref{3.d1}. 
\end{corollary}

Let us extend to Lipschitz domains and $s\in(\ha,\tha)$ Lemma 4.2(i,ii) proved in \cite{CMN-1} for smooth domains and $s=1$.
\begin{lemma}\label{VW0s}\quad
Let $\Omega$ be a bounded simply-connected Lipschitz domain,  $\ha< s< \tha$, 
$a\in C_+^{s}(\overline \Omega)$.
\begin{itemize}
\item[(i)]  If $\Psi^*\in {H}^{s-\tha}(\pO)$ and
$r_{\Omega} V\Psi^*=0, $
then $\Psi^*=0$.
 \item[(ii)]  If $\Phi^*\in {H}^{s-\ha}(\pO)$ and
$r_ {\Omega}  W\Phi^*=0,$
then $\Phi^*=0$.
\end{itemize}
\end{lemma}
\begin{proof}
To prove (i), let us multiply equation $r_{\Omega} V\Psi^*=0$ by $a$, which by the first relation in \eqref{VWa} reduces it to $r_{\Omega} V_\Delta\Psi^*=0$ in $\Omega$. 
Taking trace of this equation  on $\partial\Omega$ and using  the first relation in \eqref{3.8} (for the case $a=1$), by  Theorem ~\ref{T3.6} we obtain point (i).

Similarly, multiplying equation $r_ {\Omega}  W\Phi^*=0$ by $a$, the second relation in \eqref{VWa} reduces it to  $r_ {\Omega}  W_\Delta(a\Phi^*)=0$ in $\Omega$.
Taking trace of this equation  on $\partial\Omega$ and using  the first jump relation in \eqref{3.8} (for the case $a=1$), we obtain $\dst -\ha\hat{\Phi}^* + {\cal W}_\Delta\hat{\Phi}^*=0$ on $\partial\Omega$, where $\hat{\Phi}^*=a\Phi^*$.
Since this equation for $\hat{\Phi}^*$ is uniquely solvable (see Theorem ~\ref{T3.6}),
by condition \eqref{a>0}
this implies point (ii).
%
\end{proof}

\begin{theorem}\label{GIequivalenceH-1}
Let $\Omega$ be a bounded simply-connected Lipschitz domain,  $\ha< s< \tha$, 
$a\in C_+^{s}(\overline \Omega)$.
Let $\tilde f\in \widetilde{H}^{s-2}( {\Omega})$. A function $u\in H^s( {\Omega})$ is a
solution of PDE $Au=r_ {\Omega}\tilde f$ in $ {\Omega}$ if and only if it is a
solution of boundary-domain integro-differential equation (\ref{4.2Gen}).
\end{theorem}
\begin{proof} If $u\in H^s( {\Omega})$ solves PDE $Au=r_ {\Omega}\tilde f$ in
$ {\Omega}$, then by Theorem \ref{3rdGA}(ii) it satisfies (\ref{4.2Gen}).  On the other hand, if
$u $ solves boundary-domain integro-differential equation \eqref{4.2Gen}, then using Lemma
\ref{IDequivalenceGen} for $\Psi=T^+(\tilde f;u) $, $\Phi=\gamma^+u$
completes the proof.
\end{proof}

\section{Segregated BDIE systems for the Dirichlet problem}

For $\ha<s<\tha$, let us consider the {\bf Dirichlet Problem:} \\
{\em Find a function $u\in H^s( {\Omega})$ satisfying equations}
\begin{eqnarray}
\label{2.6} && A\,u=f  \;\;\; \mbox{\rm in}\;\;\;\;  {\Omega},
\\ 
 \label{2.7} &&  \gamma^+u=\varphi_0  \;\;\; \mbox{\rm on}\;\;\;\; \pO,
\end{eqnarray}
{\em where}   $f\in {H}^{s-2}( {\Omega})$, $ \varphi_0 \in H^{s-\ha}(\pO)$.

Equation \eqref{2.6} is understood in the distributional sense
\eqref{Ldist} and the Dirichlet boundary condition \eqref{2.7} in the trace sense.
The following uniqueness assertion 
is well known for $s=1$   and follows from the first Green identity; hence it also holds true for $1\le s<3/2$.
\begin{theorem}\label{Rem1}
Let $1\le s<\tha$ and $a\in C_+^{|s-1|}(\overline \Omega)$.
The Dirichlet  problem \eqref{2.6}-\eqref{2.7} has at most one solution in ${H}^{s}( {\Omega})$. 
\end{theorem} 

\subsection{BDIE formulations and equivalence to the Dirichlet problem}

Let $\ha<s<\tha$. 
In this section we reduce the Dirichlet problem \eqref{2.6}-\eqref{2.7}  to three different {\em segregated}  Boundary-Domain Integral Equation (BDIE)
systems.
Two of these formulations, for $s=1$ and infinitely smooth coefficients and infinitely smooth boundary, were analysed in \cite{MikArxiv2015}.

Let $\tilde f\in \widetilde{H}^{s-2}( {\Omega})$ be an extension of $f\in {H}^{s-2}( {\Omega})$ (i.e., $f=r_ {\Omega}\tilde f$), which always exists, see \cite[Lemma 2.15 and Theorem 2.16]{MikJMAA2011}. 
Let us substitute in  \eqref{4.2Gen}, \eqref{4.2u+Gen}, \eqref{4.2T+GenD} and \eqref{4.2T+Gen} the generalised co-normal derivative and the trace of the function $u$ as
$$
T^+(\tilde f;u)=\psi,\qquad \gamma^+u=\varphi_0,
$$ 
where $\varphi_0$ is the known right hand side of the Dirichlet boundary condition \eqref{2.7}, and  $\psi\in{H}^{s-\frac{3}{2}}( \partial\Omega)$ is
a new unknown function that will be regarded as formally {\em segregated} from $u$.
Thus we will look for the unknown couple
$
(u, \psi)
\in H^s( {\Omega})\times {H}^{s-\frac{3}{2}}( \partial\Omega).
$

\paragraph{BDIE system (D1).} 
Let $a\in C_+^{s}(\overline \Omega)$.
To reduce the Dirichlet BVP \eqref{2.6}-\eqref{2.7} to the BDIE system (D1), we will use equation \eqref{4.2Gen} in $ {\Omega}$ and
equation \eqref{4.2u+Gen} on $\pO$.
Then we
arrive at the following system of the boundary-domain integral equations, (D1), which is similar the corresponding system in \cite{MikArxiv2015},
\begin{align}
\label{4.5GT1}
 u+{\cal R}u - V\psi =&\F^{D1}_1 \text{ in}\  {\Omega},     \\
 \gamma^+{\cal R} u -  {\cal V}\psi =& \F^{D1}_2 
\text{ on}\ \pO, \label{4.5GT2}
\end{align}
where
\begin{equation}\label{FD1}
\F^{D1}=\left[\begin{array}{l}
   \F^{D1}_1\\[1ex]
   \F^{D1}_2
   \end{array}\right]
   =\left[\begin{array}{l}
  F_0^D\\[1ex]
  \gamma^+F_0^D-\varphi_0 
  \end{array}\right]
\text{ and }
 F_0^D:={\cal P}\tilde f -W\varphi_0
\;\;\;\text{in}\  {\Omega}.
\end{equation}
Note that for $\varphi_0\in H^{s-\ha}(\pO)$ and $\tilde f\in \widetilde{H}^{s-2}( {\Omega})$, we have the inclusion
$F_0^D\in H^{s}( {\Omega})$  due to the mapping properties of
the Newtonian (volume) and layer potentials, cf.  
(\ref{T3.1P1}), (\ref{WHs1}).
%

\paragraph{BDIE system (D2$_\Delta$).}
Let
$a\in C_+^{s}(\overline \Omega)$.
To obtain a segregated BDIE system of {\em the second kind}, we
will use equation \eqref{4.2Gen} in $ {\Omega}$ and  equation
\eqref{4.2T+GenD} on $\pO$.
Then we arrive at the following BDIE system (D2$_\Delta$),
\begin{eqnarray}
\label{2.44D} u+{\cal R}u - V\psi  =\F^{D2\Delta}_1 &\text{in}&
   {\Omega},   \\
\label{2.45D} \frac{1}{2}\,\psi + T^+_\Delta(A^\nabla u;{a\cal R} u)  - {\mathcal W}'_\Delta\psi =
\F^{D2\Delta}_2&\text{on}& \pO,
\end{eqnarray}
where
\begin{equation}\label{FD2D}
\F^{D2\Delta}=\left[\begin{array}{l}
   \F^{D2\Delta}_1\\[1ex]
   \F^{D2\Delta}_2
   \end{array}\right]
   =\left[\begin{array}{l}
  {\cal P}\tilde f-W\varphi_0\\[1ex]
  T^+_\Delta(\tilde f;\P_\Delta\tilde f)-\mathcal L_\Delta(a\varphi_0)
  \end{array}\right].
\end{equation}
Due to the mapping properties of the operators involved in
\eqref{FD2} we have $\F^{D2\Delta}\in H^{s}( {\Omega} )\times H^{s-\tha}(\pO)$.

\paragraph{BDIE system (D2).}
Let the coefficient be smoother than in the first two cases, 
$a\in C_+^{\tha}(\overline \Omega)$.
Now we will use equation \eqref{4.2Gen} in $ {\Omega}$ and  equation
\eqref{4.2T+Gen} on $\pO$.
Then we arrive at another BDIE system of {\em the second kind}, (D2), which is similar to the corresponding system in \cite{MikArxiv2015},
\begin{eqnarray}
\label{2.44} u+{\cal R}u - V\psi  =\F^{D2}_1 &\text{in}&
   {\Omega},   \\
\label{2.45} \frac{1}{2}\,\psi+T^+{\cal R}u  - {\Wp}\psi =
\F^{D2}_2&\text{on}& \pO,
\end{eqnarray}
where
\begin{equation}\label{FD2}
\F^{D2}=\left[\begin{array}{l}
   \F^{D2}_1\\[1ex]
   \F^{D2}_2
   \end{array}\right]
   =\left[\begin{array}{l}
  {\cal P}\tilde f-W\varphi_0\\[1ex]
  T^+(\tilde f+\mathring E_{\Omega}\,r_ {\Omega}{\mathcal R}_*\tilde f;{\cal P}\tilde f)-T^+ W \varphi_0
  \end{array}\right].
\end{equation}
Due to the mapping properties of the operators involved in
\eqref{FD2} we have $\F^{D2}\in H^{s}( {\Omega} )\times H^{s-\tha}(\pO)$.

Let us prove that BVP \eqref{2.6}--\eqref{2.7} in $ {\Omega}$ is
equivalent to both systems of BDIEs, (D1), (D2$_\Delta$) and (D2).

\begin{theorem}
\label{Deqin}
Let $\ha<s<\tha$ and $a\in C_+^{s}(\overline \Omega)$.
Let 
$ \varphi_0 \in H^{s-\frac{1}{2}}( \partial\Omega)$,
$f\in  H^{s-2}({\Omega})$, and   $\tilde f\in  \widetilde H^{s-2}( {\Omega})$ be such that $r_{_ {\Omega}}\tilde f=f$.
%
 \begin{enumerate}
 \item[\rm (i)]
 If a function $u\in {H}^{s}( {\Omega})$ solves the Dirichlet BVP \eqref{2.6}--\eqref{2.7}, then the couple $(u, \psi)\in H^s ( {\Omega})\times {H}^{s-\tha}(\pO)$,
    where 
    \begin{equation}
    \label{upsiphiGT}
    \psi=T^+ (\tilde f;u) \;\;\;\;\text{on}\;\;\;\; \pO, \qquad
    \end{equation}
      solves the BDIE systems (D1), (D2$_\Delta$) and, if $a\in C_+^{\tha}(\overline \Omega)$, also the BDIE system (D2).
 \item[\rm (ii)]
 Vice versa, if a a couple $(u, \psi)\in H^s ( {\Omega})\times {H}^{s-\tha}(\pO)$ solves one of the BDIE systems, (D1), (D2$_\Delta$), or (D2) (if $a\in C_+^{\tha}(\overline \Omega)$), then this solution  solves the other BDIE systems, while $u$ solves the Dirichlet BVP, and $\psi$ satisfies (\ref{upsiphiGT}).
 \end{enumerate}
 \end{theorem}
 \begin{proof} 
(i) Let 
$u\in H^s( {\Omega})$ be a solution to BVP
\eqref{2.6}--\eqref{2.7}. 
Setting $\psi$  by \eqref{upsiphiGT} evidently implies $\psi\in
{H}^{s-\tha}(\pO)$. 
Then it immediately follows from Theorem~\ref{GIequivalenceH-1} and relations
 \eqref{4.2u+Gen} and \eqref{4.2T+GenD} that the couple $(u,\psi)$ solves
systems (D1), (D2)$_\Delta$ and, if $a\in C_+^{\tha}(\overline \Omega)$, also (D2), with the right
hand sides \eqref{FD1}, \eqref{FD2D} and \eqref{FD2}, respectively,  which completes the proof of
item (i).

(ii) Let now a couple $(u,\psi)\in H^s( {\Omega})\times { H}^{s-\tha}(\pO)$
solve BDIE system \eqref{4.5GT1}-\eqref{4.5GT2}. Taking trace of
equation \eqref{4.5GT1} on $\pO$  and subtracting equation
\eqref{4.5GT2} from it, we obtain,
\begin{equation}\label{DGT}
     \gamma^+u=\varphi_0\quad \mbox{on } \pO,
\end{equation}
i.e. $u$ satisfies the Dirichlet condition \eqref{2.7}.
Equation \eqref{4.5GT1} and Lemma \ref{IDequivalenceGen} with
$\Psi=\psi$, $\Phi=\varphi_0$ imply that $u$ is a solution of PDE \eqref{2.6}, and
 $$
V\Psi^*- W\Phi^* =0\quad \mbox{in } {\Omega},
$$
where $\Psi^*=\psi -T^+ (\tilde f;u) $ and $\Phi^*=\varphi_0 - \gamma^+u$. Due to equation \eqref{DGT}, $\Phi^*=0$. Then Lemma \ref{VW0s}(i)
implies $\Psi^*=0$, which completes the proof of condition
\eqref{upsiphiGT}.
Thus $u$ obtained from solution of BDIE system (D1) solves the Dirichlet problem and hence, by item (i) of the theorem, $(u,\psi)$ solve also BDIE system (D2$_\Delta$) and, if $a\in C_+^{\tha}(\overline \Omega)$,  also (D2).  

      
Let now  a couple $(u,\psi)\in H^{s}( {\Omega})\times H ^{s-\tha}(\pO)$ solve BDIE system \eqref{2.44D}-\eqref{2.45D}.
Lemma \ref{IDequivalenceGen} for equation \eqref{2.44D} implies that
$u$ is a solution of PDE \eqref{2.6}, and  equation
\eqref{difference} holds for $\Psi=\psi$ and $\Phi=\varphi_0$, while Corollary~\ref{IDequivGenDir} gives equation \eqref{4.2T+GenD}.   
Multiplication of \eqref{difference} by $a$ reduces it to
\begin{align}
\label{differenceD}
 V_\Delta(\psi -T^+(\tilde f;u) ) -  W_\Delta(a(\varphi_0 - \gamma^+u)) = 0& \text{ in }  \Omega.
\end{align} 
Subtracting \eqref{4.2T+GenD} from equation \eqref{2.45D} and taking into account \eqref{differenceD} gives
\begin{equation}
\psi -T^+ (\tilde f;u) =0 \text{ on } \pO,\label{PsiGenD}
\end{equation}
that is, equation (\ref{upsiphiGT}) is proved.
Equations \eqref{differenceD} and \eqref{PsiGenD} give
 $ W_\Delta\Phi^* =0$ in $ {\Omega}$,
 where
$\Phi^*=a(\varphi_0 - \gamma^+u)$. Then Lemma \ref{VW0s}(ii) implies
$\Phi^*=0$ on $\pO$. This means that $u$ satisfies the Dirichlet
condition \eqref{2.7}.
Thus $u$ obtained from solution of BDIE system (D2$_\Delta$) solves the Dirichlet problem and hence, by item (i) of the theorem, the couple $(u,\psi)$ solve also BDIE system (D1) and, if $a\in C_+^{\tha}(\overline \Omega)$,  also (D2).       
      
 Let, finally, $a\in C_+^{\tha}(\overline \Omega)$ and a couple $(u,\psi)\in H^{s}( {\Omega})\times H ^{s-\tha}(\pO)$
solve BDIE system \eqref{2.44}-\eqref{2.45}.
Lemma \ref{IDequivalenceGen} for equation \eqref{2.44} implies that
$u$ is a solution of PDE \eqref{2.6}, and equation
\eqref{difference} holds for $\Psi=\psi$ and
$\Phi=\varphi_0$, while Corollary~\ref{IDequivGenDir} gives equation \eqref{4.2T+Gen}. 
Subtracting \eqref{4.2T+Gen} from equation \eqref{2.45} leads again to \eqref{differenceD},
that is, equation (\ref{upsiphiGT}) is proved.
Equations \eqref{difference} and \eqref{differenceD} imply
 $ W\Phi^* =0$ in $ {\Omega}$,
 where
$\Phi^*=\varphi_0 - \gamma^+u$. 
Then by Lemma \ref{VW0s}(ii) we deduce
$\Phi^*=0$ on $\pO$. This means that $u$ satisfies the Dirichlet
condition \eqref{2.7}.
Thus $u$ obtained from solution of BDIE system (D2) solves the Dirichlet problem and hence, by item (i) of the theorem, the couple $(u,\psi)$ solves also BDIE systems (D1) and (D2$_\Delta$).  

\end{proof} 

\subsection{Properties of BDIE system operators for the Dirichlet problem}

BDIE systems (D1), (D2$_\Delta$) and (D2) can be written as
\begin{equation*}
  \mathfrak{D}^1\U^D=\F^{D1},\quad \mathfrak{D}^{2\Delta}\U^D=\F^{D2\Delta},  \text{ and }\  \mathfrak{D}^2\U^D=\F^{D2},
\end{equation*}
respectively. Here $\U^D:=(u, \psi)^\top\in H^s( {\Omega})\times { H}^{s-\tha}(\pO)$,
\begin{equation}\label{D1-D2}
\hspace{-0.8ex}\mathfrak{D}^1:= \left[
\begin{array}{cc}
I-{\cal R} & -V  \\
\gamma^+{\cal R}  & -{\cal V}
\end{array}
\right],\quad 
   \mathfrak{D}^{2\Delta}:= \left[
 \begin{array}{ccc}
 I+{\cal R} & -V  \\[1ex]
   T^+_\Delta(A^\nabla;a{\cal R}) & \displaystyle \,\frac{1}{2}\,I- \mathcal W'_\Delta
 \end{array}
 \right],\quad 
    \mathfrak{D}^2:= \left[
  \begin{array}{ccc}
  I+{\cal R} & -V  \\[1ex]
    T^+{\cal R} & \displaystyle \,\frac{1}{2}\,I- {\Wp}
  \end{array}
  \right],
\end{equation}
while $\F^{D1}$, $\F^{D2\Delta}$ and $\F^{D2}$ are given by \eqref{FD1}, \eqref{FD2D} and \eqref{FD2}, respectively.
Note that 
\begin{align}\label{TDAr}
T^+_\Delta(A^\nabla;a\mathcal R)u:=(\gamma^{-1})^*(A^\nabla u-\check{\Delta}_{\Omega} (a\mathcal Ru)).
\end{align}

Let $\ha<s<\tha$. The operators
  \begin{eqnarray}
\mathfrak{D}^1 
&:& H^s ( {\Omega})\times { H}^{s-\tha}(\pO)\to H^s ( {\Omega})\times H^{s-\ha}(\pO)\quad\mbox{if }
a\in C_+^{s}(\overline \Omega), \label{AGcGcont1}\\
\mathfrak{D}^{2\Delta}
&:& H^s ( {\Omega})\times { H}^{s-\ha}(\pO)\to H^s ( {\Omega})\times H^{s-\tha}(\pO)\quad\mbox{if } 
a\in C_+^{s}(\overline \Omega),\label{AGcTcont2D}\\
\mathfrak{D}^2
&:& H^s ( {\Omega})\times {H}^{s-\ha}(\pO)\to H^s ( {\Omega})\times H^{s-\tha}(\pO)\quad\mbox{if } 
a\in C_+^{\tha}(\overline \Omega),\label{AGcTcont2}
  \end{eqnarray}
are continuous due to the mapping properties of the operators constituting them, see Section~\ref{S3},  
while for the right hand sides of the BDIE systems we have the following inclusions
$$\F^{D1}\in H^s({\Omega})\times H^{s-\ha}(\pO),\quad 
\F^{D2\Delta}\in H^{s}({\Omega} )\times H^{s-\tha}(\pO),\quad
\F^{D2}\in H^s({\Omega})\times H^{s-\tha}(\pO).$$


\begin{theorem}\label{AGT-Fred}
Let $\Omega$ be a bounded simply-connected Lipschitz domain and $\ha< s< \tha$.
Operators \eqref{AGcGcont1}-\eqref{AGcTcont2} are Fredholm operators with zero index.
\end{theorem}
\begin{proof} 
The operators continuity has been proved above already. 

To prove the Fredholm property of operator \eqref{AGcGcont1}, let us consider the  operator
\begin{equation*}
\mathfrak{D}^1_0:= \left[
\begin{array}{cc}
I &        -V            \\
 0 & \,-{\cal V}
\end{array}
\right]\ .
\end{equation*}
As a result of compactness properties of the operators ${\cal R}$
and $\gamma^+{\cal R}$ given by \eqref{T3.1P3has} and \eqref{T3.1P3+} in Theorem \ref{T3.1P}), the operator
$\mathfrak{D}^1_0$ is a compact perturbation of operator \eqref{AGcGcont1}.
The operator $\mathfrak{D}^1_0$ is an upper triangular matrix operator
with the following scalar diagonal invertible operators
\begin{eqnarray*}
I&:& H^s ( {\Omega})\to H^s ( {\Omega}),\\
{\cal V} & :&   H^{s-\tha}(\pO)\to H^{s-\ha}(\pO),
\end{eqnarray*}
where the invertibility of the operator $\V$ is implied by invertibility of operator $\V_\Delta$ in \eqref{3.15} and by the first relation in \eqref{VWa}. 
This implies that
$$
\mathfrak{D}^1_0 \;:\; H^s ( {\Omega})\times { H}^{s-\tha}(\pO)\to H^s ( {\Omega})\times H^{s-\ha}(\pO)
$$
is an invertible  operator. Thus   \eqref{AGcGcont1} is a Fredholm operator with zero index. 

The operator
\begin{align}\label{D20}
  \mathfrak{D}^{2}_0:= \left[
\begin{array}{ccc}
I  & -V  \\
0  & \displaystyle \frac{1}{2}\,I- \mathcal W'_\Delta
\end{array}
\right].
\end{align}
is a compact perturbation of  operator \eqref{AGcTcont2D}.
Indeed the operators $\mathcal R:H^s({\Omega})\to H^s ( {\Omega})$ is compact due to Theorem \ref{T3.1P}.      
Compactness of the operator 
$
T^+_\Delta(A^\nabla;a\mathcal R) : H^s({\Omega})\to H^{s-\tha}(\pO),
$
defined by \eqref{TDAr}, follows from compactness of the operator $A^\nabla: H^{s}(\Omega)\to \widetilde H^{s-2}(\Omega)$ given by Lemma~\ref{AringMap}(ii) and of the operator $\mathcal R: H^{s}(\Omega)\to H^{s}(\Omega)$,
i.e., operator \eqref{T3.1P3has} in Theorem~\ref{T3.1P}.
Consider the diagonal operators of the upper triangular matrix operator $\mathfrak{D}^2_0$.
The operator $I: H^s ( {\Omega})\to H^s ( {\Omega})$ is evidently invertible, 
while invertibility of the operator $\frac{1}{2}\,I- \mathcal W'_\Delta : H^{s-\tha}(\pO)\to H^{s-\tha}(\pO)$ is stated by Theorem \ref{T3.6}.
This implies that the operator 
\begin{align}\label{D20map}
\mathfrak{D}^{2}_0 : H^s ( {\Omega})\times { H}^{s-\tha}(\pO)\to H^s ( {\Omega})\times H^{s-\tha}(\pO)
\end{align}
is invertible and hence operator \eqref{AGcTcont2D} is Fredholm with zero index.

The operator $\mathfrak{D}^2_0$, defined by \eqref{D20}, is also a compact perturbation of  operator \eqref{AGcTcont2}.
Indeed the operators $\mathcal R:H^s({\Omega})\to H^s ( {\Omega})$ and  
$T^+{\cal R}: H^s({\Omega})\to H^{s-\tha}(\pO)$ are compact due to Theorem \ref{T3.1P}. 
From the first representation in \eqref{VWab4}, for $a\in C_+^{\tha}(\overline \Omega)$, the operator  
$\mathcal W^\prime_\Delta - \mathcal W^{\prime}= \frac{\partial_\nu a}{a}\,\mathcal V_{_\Delta}: H^{s-\tha}(\pO)\to H^{\sigma}(\pO)$,
 where $\sigma=\min\{\ha,s-\ha\}$, is continuous, which implies that the operator 
$\mathcal W^\prime_\Delta - \mathcal W^{\prime}: H^{s-\tha}(\pO)\to H^{s-\tha}(\pO)$ is compact. 
Since operator \eqref{D20map}
is invertible, this implies that operator \eqref{AGcTcont2} is Fredholm with zero index.
\end{proof} 

\begin{theorem}\label{AGTinvs}
Let $\Omega$ be a bounded simply-connected Lipschitz domain,  $\ha< s< \tha$, and  $\sigma=\max\{1, s\}$.
The following operators  are continuously invertible.
\begin{eqnarray}
\mathfrak{D}^1 
&:& H^s ( {\Omega})\times { H}^{s-\tha}(\pO)\to H^s ( {\Omega})\times H^{s-\ha}(\pO)\quad\mbox{if }
a\in C_+^{\sigma}(\overline \Omega), \label{AGcGcont1a}\\
\mathfrak{D}^{2\Delta}
&:& H^s ( {\Omega})\times { H}^{s-\ha}(\pO)\to H^s ( {\Omega})\times H^{s-\tha}(\pO)\quad\mbox{if } 
a\in C_+^{\sigma}(\overline \Omega),\label{AGcTcont2Da}\\
\mathfrak{D}^2
&:& H^s ( {\Omega})\times {H}^{s-\ha}(\pO)\to H^s ( {\Omega})\times H^{s-\tha}(\pO)\quad\mbox{if } 
a\in C_+^{\tha}(\overline \Omega),\label{AGcTcont2a}
  \end{eqnarray}
\end{theorem}
\begin{proof} 
Let $1\le s<\tha$ first. 
Then $\sigma=s$ and injectivity of operators \eqref{AGcGcont1a}-\eqref{AGcTcont2a} is implied by the equivalence Theorem~\ref{Deqin}(ii) and the BVP uniqueness Theorem~\ref{Rem1}. 
Indeed, consider, for example, injectivity of operator \eqref{AGcGcont1a}. 
For the homogeneous equation $\mathfrak{D}^1\U^D=0$, its zero right hand side $\F^{D1}=0$  can be represented as in \eqref{FD1} in terms of $\tilde f=0$ and $\varphi_0=0$. 
Then by Theorem~\ref{Deqin}(ii), $\U^D=(u, T^+(0;u))^\top$, where $u$ is a solution of the Dirichlet problem \eqref{2.6}-\eqref{2.7} with the right hand sides $f=0$ and $\varphi_0=0$, which has only the trivial solution, $u=0$,  due to Theorem~\ref{Rem1}. 
The arguments for injectivity of operators \eqref{AGcTcont2Da} and \eqref{AGcGcont1a} are similar.

Since, by Theorem~\ref{AGT-Fred}, operators \eqref{AGcGcont1a}-\eqref{AGcTcont2a} are Fredholm with zero index, this implies their invertibility  for $1\le s<\tha$.

Let now $\ha< s\le 1$.
Then $\sigma=1$, i.e., $a\in C_+^{1}(\overline \Omega)$ for operators \eqref{AGcGcont1a}-\eqref{AGcTcont2Da} and $a\in C_+^{\tha}(\overline \Omega)$ for operator \eqref{AGcTcont2a}. Hence, for a fixed function $a$, satisfying the corresponding conditions in \eqref{AGcGcont1a}-\eqref{AGcTcont2Da}, all these operators are continuous for $\ha< s\le 1$. 
By Theorem~\ref{AGT-Fred} they are also Fredholm with zero index. Since, as already proved, at $s=1$ these operators are also invertible,    Lemma~\ref{Lem2-Fredholm} implies that their kernels (null-spaces) consist of only the zero element for any $s\in(\ha,1]$, which implies that the operators are invertible for all $s$ from this interval.
\end{proof}

Theorems \ref{AGTinvs} and \ref{Deqin} imply the following assertion.
\begin{corollary}\label{Cor5.5} Let $\Omega$ be a bounded simply-connected Lipschitz domain,  $\ha< s< \tha$, $f\in{H}^{s-2}({\Omega})$, $ \varphi_0 \in H^{s-\ha}(\pO)$, and 
$a\in C_+^{\sigma}(\overline \Omega)$, $\sigma=\max\{1, s\}$.
Then the Dirichlet  problem \eqref{2.6}-\eqref{2.7} is uniquely solvable in ${H}^{s}( {\Omega})$. 
The solution is $u=(\mathcal A^D)^{-1}(f,\varphi_0)^\top$, where the inverse operator, 
$(\mathcal A^D)^{-1}: H^{s-2}( {\Omega})\times H^{s-\ha}(\pO) \to H^{s}( {\Omega})$, to the left hand side operator, $\mathcal A^D: H^{s}( {\Omega})\to H^{s-2}( {\Omega})\times H^{s-\ha}(\pO)$, of the Dirichlet problem \eqref{2.6}-\eqref{2.7}, is continuous.
\end{corollary}

\begin{remark}
For a given function $f\in  H^{s-2}( {\Omega})$, its extension  $\tilde f\in  \widetilde H^{s-2}( {\Omega})$ is not unique. Nevertheless, since solution of the Dirichlet BVP \eqref{2.6}--\eqref{2.7} does not depend on this extension, equivalence Theorem~\ref{Deqin}(ii) implies that $u$ in the solution of BDIE systems (D1) and (D2) does not depend on the particular  choice of extension $\tilde f$, however, $\psi$ obviously does, see \eqref{upsiphiGT}.
\end{remark}

\section{Segregated BDIE systems for the Neumann Problem}
Let us consider the 
{\bf Neumann Problem:} {\em Find a function $u\in H^s( {\Omega})$ satisfying equations}
\begin{align}
\label{2.6N} & A\,u=r_ {\Omega}\tilde f  \;\;\; \mbox{\rm in}\;\;\;\;  {\Omega},
\\ 
 \label{2.8} &  T^+(\tilde f;u)=\psi_0  \;\;\; \mbox{\rm on}\;\;\;\; \pO,
\end{align}
{\em where} $ \psi_0 \in H^{s-\tha}(\pO)$,  $\tilde f\in \widetilde{H}^{s-2}( {\Omega})$. 

Equation \eqref{2.6N} is understood in the distribution sense
\eqref{Ldist}, and the Neumann boundary condition \eqref{2.8} in the weak sense \eqref{Tgen}.
The following assertion is well-known and can be proved, e.g., using the variational settings and the Lax-Milgram lemma.
\begin{theorem}\label{Rem1N} 
Let $s=1$ and $a\in L_\infty(\Omega)$.

 (i) The homogeneous Neumann problem, \eqref{2.6N}-\eqref{2.8}, admits only one linearly independent solution $u^0=1$ in ${H}^{1}( {\Omega})$.
 
 (ii) The non-homogeneous Neumann problem \eqref{2.6N}-\eqref{2.8} is solvable if and only if the following condition is satisfied
\be\label{3.suf}
\langle \tilde f,u^0\rangle_ {\Omega} -\langle\psi_0,\gamma^+u^0\rangle_\pO=0.
\ee
\end{theorem}
\begin{remark}\label{Rem1Ns}
Item (i) in Theorem~\ref{Rem1N} evidently implies that for $1\le s<\tha$ and $a\in C_+^{|s-1|}(\overline \Omega)$, the homogeneous Neumann problem, associated with \eqref{2.6N}-\eqref{2.8}, also admits only one linearly independent solution $u^0=1$ in ${H}^{s}( {\Omega})$. 
\end{remark}
\subsection{BDIE formulations and equivalence to the Neumann problem}

We will explore different possibilities of
reducing the Neumann problem \eqref{2.6N}-\eqref{2.8} to a BDIE system. 
Let $\ha<s<\tha$.
Let us represent in  \eqref{4.2Gen}, \eqref{4.2u+Gen}, \eqref{4.2T+GenD} and \eqref{4.2T+Gen} the generalised co-normal derivative and the trace of the function $u$ as
$$
T^+(\tilde f;u)=\psi_0,\qquad \gamma^+u=\varphi,
$$
where $\psi_0$ is the known right hand side of the Neumann boundary condition \eqref{2.8}, and  $\varphi\in{H}^{s-\frac{1}{2}}( \partial\Omega)$ is
a new unknown function that will be regarded as formally {\em segregated} from $u$.
Thus we will look for the unknown couple
$
(u, \varphi)
\in H^s( {\Omega})\times {H}^{s-\frac{1}{2}}(\partial\Omega).
$

\paragraph{BDIE system (N1$_\Delta$).}
Let $a\in C_+^{s}(\overline \Omega)$. 
Using equation \eqref{4.2Gen} in $ {\Omega}$ and
equation \eqref{4.2T+GenD} on $ \partial\Omega$,  we arrive
at the following BDIE system (N1$_\Delta$) of two equations for the couple of unknowns,
$(u, \varphi)$,
\begin{eqnarray}
\label{Basic1D}
 u +{\mathcal R}u  +W\varphi &=& \F^{N1\Delta}_1\quad \mbox{in}\quad    {\Omega},
\\[1ex]
 \label{Basic2D}
T^+_\Delta(A^\nabla u;{a\cal R} u) +\mathcal L_\Delta(a \varphi)
  &=&
  \F^{N1\Delta}_2 \quad  \mbox{on}\quad   \partial\Omega,
\end{eqnarray}
where      
\begin{equation}\label{FN1D}
\F^{N1\Delta}=\left[\begin{array}{l}
  \F^{N1\Delta}_1\\[1ex]
  \F^{N1\Delta}_2
  \end{array}\right]
=\left[\begin{array}{l}
  {\cal P}\tilde f + V\psi_0\\[1ex]
  T^+_\Delta(\tilde f;\P_\Delta\tilde f)-\dfrac{1}{2}\psi_0+{\W}'_\Delta\psi_0
  \end{array}\right].
\end{equation}
Due to the mapping properties of the operators involved in
\eqref{FN1} we have $\F^{N1\Delta}\in H^{s}( {\Omega} )\times H^{s-\tha}(\pO)$.

\paragraph{BDIE system (N1).}
Let the coefficient be smoother than in the previous case, 
$a\in C_+^{\tha}(\overline \Omega)$.
Now, using equation \eqref{4.2Gen} in $ {\Omega}$ and
equation \eqref{4.2T+Gen} on $ \partial\Omega$,  we arrive
at the following BDIE system (N1) of two equations for the couple of unknowns,
$(u, \varphi)$, which is similar to the corresponding system in \cite{MikArxiv2015},
\begin{eqnarray}
\label{Basic1}
 u +{\mathcal R}u  +W\varphi &=& \F^{N1}_1\quad \mbox{in}\quad    {\Omega},
\\[1ex]
 \label{Basic2}
  T^+{\mathcal R} u +T^+ W\varphi
  &=&
  \F^{N1}_2 \quad  \mbox{on}\quad   \partial\Omega,
\end{eqnarray}
where
\begin{equation}\label{FN1}
\F^{N1}=\left[\begin{array}{l}
  \F^{N1}_1\\[1ex]
  \F^{N1}_2
  \end{array}\right]
=\left[\begin{array}{l}
  {\cal P}\tilde f + V\psi_0\\[1ex]
T^+(\tilde f+\mathring E_{\Omega}\,r_ {\Omega}{\mathcal R}_*\tilde f;{\cal P}\tilde f) -\dfrac{1}{2}\psi_0+{\Wp}\psi_0
  \end{array}\right].
\end{equation}
Due to the mapping properties of the operators involved in
\eqref{FN1} we have $\F^{N1}\in H^{s}( {\Omega} )\times H^{s-\tha}(\pO)$.
\paragraph{BDIE system (N2).}
Let again $a\in C_+^{s}(\overline \Omega)$.
If we use equation \eqref{4.2Gen} in $ {\Omega}$ and equation \eqref{4.2u+Gen} on $ \partial\Omega$, we arrive for the couple $(u,  \varphi)$ at the following  BDIE system (N2) of two equations {\em of the second kind}, which is also similar to the corresponding system in \cite{MikArxiv2015},
\begin{eqnarray}
\label{Basic1-2}
 u +{\mathcal R}u     +W  \varphi
 &=& \mathcal F^{N2}_1  \qquad \mbox{in}\quad    {\Omega},
\\[1ex]
 \label{Basic2-2}
 \frac{1}{2}\varphi+\gamma^+{\mathcal R} u   +{\mathcal W} \varphi
 &=&    \mathcal F^{N2}_2,  \quad   \mbox{on}\quad   \partial\Omega.
\end{eqnarray}
where 
\begin{align}
 \label{FN2}
\mathcal F^{N2}=\left[\begin{array}{l}
   \F^{N2}_1\\[1ex]
   \F^{N2}_2
   \end{array}\right]
   =\left[\begin{array}{l}
      F_0^N\\[1ex]
      \gamma^+F_0^N
      \end{array}\right],\quad
F_0^N:={\mathcal P} \tilde f + V\psi_0  \quad   \mbox{in}\quad    {\Omega}.
 \end{align}
Due to the mapping properties of the operators involved in
\eqref{FN2}, we have $\F^{N2}\in H^{s}( {\Omega} )\times H^{s-\ha}(\pO)$.

\begin{theorem}\label{equivalenceN}
Let $\ha<s<\tha$, $a\in C_+^{s}(\overline \Omega)$,
$ \psi_0 \in H^{s-\frac{3}{2}}( \partial\Omega)$  and   $\tilde f\in  \widetilde H^{s-2}( {\Omega})$.

(i) If a function $u\in H^s( {\Omega} )$ solves the Neumann  problem \eqref{2.6N}-\eqref{2.8} then the couple $(u,\varphi)$ with $\varphi=\gamma^+u\in H^{s-\frac{1}{2}}({ \partial\Omega})$ solves 
BDIE systems {\rm(N1$_\Delta$)}, {\rm(N2)} and, if $a\in C_+^{\tha}(\overline \Omega)$, also {\rm(N1)}.

(ii) Vice versa, if a couple 
$(u,\varphi)\in H^s( {\Omega} )\times H^{s-\frac{1}{2}}({ \partial\Omega})$ solves one of the
BDIE systems, {\rm(N1$_\Delta$)}, {\rm(N2)}, or {\rm(N1)} (if $a\in C_+^{\tha}(\overline \Omega)$), then the couple solves the other two BDE systems, while $u$ solves the Neumann
problem  \eqref{2.6N}-\eqref{2.8} and $\gamma^+u=\varphi$.
\end{theorem} 
\begin{proof} 
(i) Let $u\in H^s( {\Omega} )$ be a solution of the Neumann  problem \eqref{2.6N}-\eqref{2.8}. 
Then from Theorem~\ref{GIequivalenceH-1} and relations \eqref{4.2u+Gen}-\eqref{4.2T+Gen} we see
that the couple $(u,\varphi)$ with $\varphi= \gamma^+u$ solves BDIE systems (N1$_\Delta$), (N2) and (N1) with the right hand sides \eqref{FN1D}, \eqref{FN2} and \eqref{FN1}, respectively,   which proves item (i).

(ii) Let a couple $(u,\varphi)\in H^s( {\Omega} )\times
H^{s-\frac{1}{2}}({ \partial\Omega})$ solve BDIE system (N1$_\Delta$). 
Lemma \ref{IDequivalenceGen} for equation \eqref{Basic1D} implies that
$u$ is a solution of PDE \eqref{2.6N}, and equation
\eqref{difference} holds for $\Psi=\psi_0$ and
$\Phi=\varphi$, while Corollary~\ref{IDequivGenDir} gives equation \eqref{4.2T+Gen}. 
Multiplication of \eqref{difference} by $a$ reduces it to
\begin{align}
\label{differenceDN}
 V_\Delta(\psi_0 -T^+(\tilde f;u) ) -  W_\Delta(a(\varphi - \gamma^+u)) = 0& \text{ in }  \Omega.
\end{align} 
Subtracting \eqref{4.2T+Gen} from equation \eqref{Basic2D}, we get $T^+ (\tilde f;u)=\psi_0$ on $ \partial\Omega$, i.e., $u$ satisfies the Neumann condition \ref{2.8}. 
Further, from  \eqref{differenceDN} we derive
$ W_\Delta(a(\varphi-\gamma^+u))=0$  in $ \Omega,$
whence  $\gamma^+u = \varphi$  on $\partial  {\Omega}$ by Lemma
 \ref{VW0s}, completing item (ii) for BDIE system (N1$_\Delta$).

Let a couple $(u,\varphi)\in H^1( {\Omega} )\times
H^{\frac{1}{2}}({ \partial\Omega})$ solve BDIE system (N1). 
Lemma \ref{IDequivalenceGen} for equation \eqref{Basic1} implies that
$u$ is a solution of PDE \eqref{2.6N}, and equation
\eqref{difference} holds for $\Psi=\psi_0$ and
$\Phi=\varphi$, while Corollary~\ref{IDequivGenDir} gives equation \eqref{4.2T+Gen}. 
Subtracting \eqref{4.2T+Gen} from equation \eqref{Basic2} gives $T^+ (\tilde f;u)=\psi_0$ on $ \partial\Omega$, i.e., $u$ satisfies the Neumann condition \ref{2.8}. 
Further, from  \eqref{difference} we derive
$ W(\gamma^+u - \varphi)=0$  in $\Omega,$
whence  $\gamma^+u = \varphi$  on $\partial  {\Omega}$ by Lemma
 \ref{VW0s}, completing item (ii) for BDIE system (N1).

Let now a couple $(u,\varphi)\in H^1( {\Omega} )\times
H^{\frac{1}{2}}({ \partial\Omega})$ solve BDIE system
(N2). 
Further, taking the trace of \eqref{Basic1-2} on ${ \partial\Omega}$ and
comparing the result with \eqref{Basic2-2}, we easily derive that
$\gamma^+u=\varphi$ on ${ \partial\Omega}$. 
Lemma \ref{IDequivalenceGen} for equation \eqref{Basic1-2} implies that
$u$ is a solution of PDE \eqref{2.6N}, and equations
\eqref{difference} holds for $\Psi=\psi_0$ and
$\Phi=\varphi$. 
Further, from  \eqref{difference} and relation $\gamma^+u=\varphi$ we derive
$$
 V(\psi_0 -T^+(\tilde f;u) )=0 \quad   \mbox{\rm in}\;\;\;\;
   \Omega,
$$
whence $(\tilde f;u)=\psi_0$ on ${ \partial\Omega}$ by Lemma
 \ref{VW0s}, i.e., $u$ solves the
Neumann problem
 \eqref{2.6N}-\eqref{2.8}, which completes the proof of item (ii) for BDIE system (N2).
\end{proof} 
\subsection{Properties of BDIE system operators for the Neumann problem}
BDIE systems (N1$_\Delta$), (N1) and (N2) can be written, respectively, as
\begin{equation*} 
{\mathfrak N}^{1\Delta}\mathcal U^N=\mathcal F^{N1\Delta},\quad
{\mathfrak N}^1\mathcal U^N=\mathcal F^{N1},\quad
{\mathfrak N}^2\mathcal U^N=\mathcal F^{N2},
\end{equation*}
where 
$
{\mathcal U}^N=(u, \varphi)^\top
\in H^s( {\Omega})\times {H}^{s-\frac{1}{2}}(\partial{\Omega}),
$
\begin{equation*}
{\mathfrak N}^{1\Delta}:= \left[
\begin{array}{cc}
I+{\mathcal R} &\ W  \\[1ex]
T^+_\Delta(A^\nabla;{a\cal R})  &\ \mathcal L_0
\end{array}\right],\quad
{\mathfrak N}^1:= \left[
\begin{array}{cc}
I+{\mathcal R} &\ W  \\[1ex]
 T^+{\mathcal R}  &\ T^+ W
\end{array}\right],\quad
{\mathfrak N}^2:= \left[
\begin{array}{cc}
I+{\mathcal R} & W  \\[1ex]
\,\gamma^+ {\mathcal R}& \quad \displaystyle \frac{1}{2}I + {\mathcal W}
\end{array}
\right],
\end{equation*}
and we denoted $\L_0g:={\L}_{\Delta}(ag)$.
Let $\ha<s<\tha$. 
Due to the mapping properties of the potentials, see Section~\ref{S3}, the operators
\begin{eqnarray}
\mathfrak N^{1\Delta}&:& H^s( {\Omega})\times H^{s-\frac{1}{2}}(\partial\Omega) 
 \to H^s( {\Omega} )\times H^{s-\frac{3}{2}}( \partial\Omega)\quad\mbox{if }
 a\in C_+^{s}(\overline \Omega),\label {mfN1D}\\
\mathfrak N^1&:& H^s( {\Omega})\times H^{s-\frac{1}{2}}(\partial\Omega) 
 \to H^s({\Omega} )\times H^{s-\frac{3}{2}}( \partial\Omega)\quad\mbox{if }
 a\in C_+^{\tha}(\overline \Omega),\label {mfN1}\\
\label{mfN2}
\mathfrak N^2&:&H^s( {\Omega})\times H^{s-\frac{1}{2}}({\partial\Omega})
\to H^s( {\Omega} )\times H^{s-\frac{1}{2}}({ \partial\Omega})\quad\mbox{if }
a\in C_+^{s}(\overline \Omega).
  \end{eqnarray}
are continuous, while for the right hand sides of the BDIE systems we have the following inclusions 
$\mathcal F^{N1\Delta}\in  H^s ({\Omega})\times H^{s-\tha}(\pO),$
$\mathcal F^{N1}\in  H^s ({\Omega})\times H^{s-\tha}(\pO),$
$\mathcal F^{N2}\in  H^s ({\Omega})\times H^{s-\ha}(\pO).$

\begin{theorem}\label{FredholmN}
Let $\Omega$ be a bounded simply-connected Lipschitz domain and $\ha< s< \tha$. 
The operators \eqref{mfN1D}-\eqref{mfN2} are Fredholm operators with zero index. 
\end{theorem}
\begin{proof} 
The operators continuity is already proved above. 

Let us consider operator \eqref{mfN1D}. 
Due to estimate \eqref{a>0} and Theorem~\ref{T3.6}, the operator $\L_0: H^{s-\frac{1}{2}}( \partial\Omega) \to
H^{s-\frac{3}{2}}( \partial\Omega)$ is a Fredholm operator with zero index.
Therefore the operator
\begin{equation}
\label{d27}
\mathfrak N^{1}_0:=
 \left[
 \begin{array}{cc}
I     &   W \\
 0      & \L_0
  \end{array}
 \right]  \,:\, H^s( {\Omega} )\times H^{s-\frac{1}{2}}( \partial\Omega) 
\to H^s( {\Omega} )\times H^{s-\frac{3}{2}}( \partial\Omega).
  \end{equation}
is also Fredholm with zero index.
Operator \eqref{mfN1D} is a compact perturbation of $\mathfrak N^{1}_0$ since the operators 
\begin{align}
{\cal R}\, &:\, H^{s}( {\Omega})\to H^{s}( {\Omega}),\label{Rcomp}\\
T^+_\Delta(A^\nabla;{a\cal R})  \, &:\, H^s( {\Omega} )\to H^{s-\frac{3}{2}}( \partial\Omega)
\end{align}
are compact, due to Theorem \ref{T3.1P}, as has been shown in the compactness proof for operator \eqref{D20}.
Thus operator  \eqref{mfN1D} is Fredholm with zero index as well. 

The operator $\mathfrak{N}^1_0$, defined by \eqref{d27}, is also a compact perturbation of  operator  \eqref{mfN1}. 
Indeed, the operators \eqref{Rcomp},
\begin{align*}
T^+ W-\L_0\, &:\, H^{s-\frac{1}{2}}( \partial\Omega) \to
H^{s-\frac{3}{2}}( \partial\Omega),\\
 T^+{\cal R} \, &:\, H^s( {\Omega} )\to H^{s-\frac{3}{2}}( \partial\Omega)
\end{align*}
are compact, due to relations \eqref{VWab4}, \eqref{3.11a}, and Theorem \ref{T3.3}.
Thus operator  \eqref{mfN1} is Fredholm with zero index as well. 

To analyse operator \eqref{mfN2}, let us consider the auxiliary operator
\begin{equation}
\label{N20}
{\mathfrak N}^2_0:= \left[
\begin{array}{cc}
I& W  \\[1ex]
0& \quad  \frac{1}{2}I + {\mathcal W}
\end{array}
\right]:H^s( {\Omega})\times H^{s-\frac{1}{2}}({\partial\Omega})
\to H^s( {\Omega} )\times H^{s-\frac{1}{2}}({ \partial\Omega}).
\end{equation}
For any function $g$, we can represent 
$(\frac{1}{2}I + {\mathcal W})g= \frac{1}{a}(\frac{1}{2}I + {\mathcal W}_\Delta)(ag)$ which, by Theorem~\ref{T3.6}, implies that the operator $\frac{1}{2}I + {\mathcal W}: H^{s-\frac{1}{2}}({ \partial\Omega})\to H^{s-\frac{1}{2}}({ \partial\Omega})$, and hence operator \eqref{N20} are Fredholm with zero index.
Due to compactness of operator \eqref{Rcomp}, operator \eqref{mfN2} is a compact perturbation of operator \eqref{N20} and thus is Fredholm with zero index as well.
\end{proof} 

\begin{theorem}\label{KernelN}
Let $\Omega$ be a bounded simply-connected Lipschitz domain and $\ha< s< \tha$, and  $\sigma=\max\{1, s\}$. 
The following operators have one--dimensional null--spaces,
$\ker \mathfrak N^{1\Delta}=\ker \mathfrak N^1=\ker \mathfrak N^2$, in $H^s( {\Omega} )\times
H^{s-\frac{1}{2}}({ \partial\Omega})$, spanned over the element $(u^0,\varphi^0)=(1,1)$.
\begin{eqnarray}
\mathfrak N^{1\Delta}&:& H^s( {\Omega})\times H^{s-\frac{1}{2}}(\partial\Omega) 
 \to H^s( {\Omega} )\times H^{s-\frac{3}{2}}( \partial\Omega)\quad\mbox{if }
 a\in C_+^{\sigma}(\overline \Omega),\label {mfN1Da}\\
\mathfrak N^1&:& H^s( {\Omega})\times H^{s-\frac{1}{2}}(\partial\Omega) 
 \to H^s({\Omega} )\times H^{s-\frac{3}{2}}( \partial\Omega)\quad\mbox{if }
 a\in C_+^{\tha}(\overline \Omega),\label {mfN1a}\\
\mathfrak N^2&:&H^s( {\Omega})\times H^{s-\frac{1}{2}}({\partial\Omega})
\to H^s( {\Omega} )\times H^{s-\frac{1}{2}}({ \partial\Omega})\quad\mbox{if }
a\in C_+^{\sigma}(\overline \Omega).
\label{mfN2a}
  \end{eqnarray}
\end{theorem}
\begin{proof} 
The conditions on the coefficient $a$ imply that for $s=1$ operators \eqref{mfN1Da}-\eqref{mfN2a} are continuous. 
Then the equivalence Theorem \ref{equivalenceN} and Theorem~\ref{Rem1N}(i) imply that 
the homogeneous BDIE systems (N1$_\Delta$), (N1), and (N2) have only one linear independent solution 
$\mathcal U^0=(u^0,\varphi^0)^\top=(1,1)^\top$ in 
 $H^1( {\Omega} )\times H^{\frac{1}{2}}({ \partial\Omega})$. 
Indeed, consider, for example, 
the homogeneous equation $\mathfrak{N}^{1\Delta}\U^N=0$. 
Its zero right hand side $\F^{N1\Delta}=0$  can be represented as in \eqref{FN1D} in terms of $\tilde f=0$ and $\psi_0=0$. 
Then by Theorem~\ref{equivalenceN}(ii), $\U^N=(u, \gamma^+u)^\top$, where $u$ is a solution of the Neumann problem \eqref{2.6N}-\eqref{2.8} with the right hand sides $f=0$ and $\psi_0=0$, which has only the one linearly independent  solution, $u=1$,  due to Theorem~\ref{Rem1N}.   
This proves the theorem for $s=1$, and then Lemma~\ref{Lem2-Fredholm} and Theorem~\ref{FredholmN} complete the proof for $\ha< s< \tha$.
\end{proof}

\begin{lemma}\label{L4D}
Let $\Omega$ be a bounded simply-connected Lipschitz domain, $\ha<s<\tha$, and $a\in C_+^{\sigma}(\overline\Omega)$, $\sigma=\max\{1, s\}$. 
For any couple $(\F_1,\F_2)\in H^{s}( {\Omega})\times {H}^{s-\tha}(\pO)$, 
there exists a unique couple $(\tilde f_{*}, \Phi_*)\in \s{H}^{s-2} ({\Omega})\times 
{H}^{s-\ha}(\pO)$
 such that
\begin{eqnarray}
\label{4.8Psi1Delta} \F_1 &=&{\cal P}\tilde f_*-W\Phi_* \quad
\text{in } {\Omega},\\
\label{4.8Psi2Delta}
\F_2&=&T^+_\Delta(\tilde f_*;{\cal P}\tilde f_*)
-\mathcal L_\Delta(a\Phi_*) 
  \quad
\text{on } \partial\Omega.
\end{eqnarray}
Moreover, $(\tilde f_*, \Phi_*)=\mathcal C_*(\F_1,\F_2)$ and
$\mathcal C_*: H^{s}( {\Omega})\times {H}^{s-\tha}(\pO)\to \s{H}^{s-2} ( {\Omega})\times {H}^{s-\ha}(\pO)$ 
is a linear continuous operator given by 
\begin{align}\label{fDelta1}
\tilde f_*=&\check\Delta_{\Omega}(a\F_1)+\gamma^*\F_2,\\
\label{Phi*FDelta0}
 \Phi_* 
=&\frac{1}{a}\left(-\ha I + {\cal W}_\Delta\right)^{-1}\gamma^+
\{ -a\F_1+{\cal P}_\Delta[\check\Delta_{\Omega}(a\F_1)+\gamma^*\F_2]\}.
\end{align}
\end{lemma}
\begin{proof}
Let us first assume that there exist $(\tilde f_*, \Phi_*)\in \s{H}^{s-2} ( {\Omega})\times {H}^{s-\ha}(\pO)$
satisfying equations \eqref{4.8Psi1Delta}, \eqref{4.8Psi2Delta} and find their expressions in terms of $\F_1$ and $\F_2$.
Multiplying \eqref{4.8Psi1Delta}  by $a$, we get, 
\begin{align}
\label{4.8Psi1aDelta} a\F_1-{\cal P}_\Delta\tilde f_* = - W_\Delta(a\Phi_*) 
\text{ in } {\Omega}.
\end{align}
Applying the Laplace operator to \eqref{4.8Psi1aDelta}, we obtain, 
\begin{align}
\label{2.6'genPhi=Delta}
& \Delta(a\F_1-\P_\Delta \tilde f_*)=\Delta(a\F_1)- \tilde f_*=-\Delta W_\Delta(a\Phi_*)=0 
\text{ in }   {\Omega},
\end{align}
which means
\begin{align}
\label{2.6'genPhiDelta}
& \Delta(a\F_1)=r_ {\Omega} \tilde f_* \text{ in }   {\Omega}
\end{align}
and $a\F_1-\P_\Delta \tilde f_*\in H^{s,0}( {\Omega};\Delta)$. 
Applying the canonical conormal derivative operator $T^+_\Delta$ to the both sides of 
equation \eqref{4.8Psi1aDelta} and taking into account that 
$\tilde\Delta W_\Delta(a\Phi_*)=\tilde\Delta (a\F_1-{\cal P}_\Delta\tilde f_*)=0$ because $W_\Delta(a\Phi_*)$ is a harmonic function in $\Omega$, we obtain,
\begin{multline}\label{6.32Delta}
-\mathcal L_\Delta(a\Phi_*) = -T^+_\Delta W_\Delta (a\Phi_*)=T^+_\Delta(a\F_1-\P_\Delta \tilde f_*)
=(\gamma^{-1})^*[\tilde\Delta_{\Omega}(a\F_1-\P_\Delta \tilde f_*)
-\check\Delta_{\Omega}(a\F_1-\P_\Delta \tilde f_*)]\\
=-(\gamma^{-1})^*\check\Delta_{\Omega} (a\F_1-{\cal P}_\Delta\tilde f_*)
=T^+_\Delta(0;a\F_1-\P_\Delta \tilde f_*),
\end{multline}
where \eqref{2.6'genPhiDelta} and the third relation in \eqref{3.d1} were taken into account.
Substituting this to \eqref{4.8Psi2Delta}, we obtain
\begin{eqnarray}
\label{4.8Psi2TDelta}
\F_2&=&T^+_\Delta(\tilde f_*,a\F_1)
  \quad
\text{on } \partial\Omega.
\end{eqnarray}

Due to \eqref{2.6'genPhiDelta}, we can represent
\begin{align}\label{fDelta}
\tilde f_*=\check\Delta_{\Omega}(a\F_1)+\tilde f_{1*}=\nabla\cdot\mathring E_{\Omega} \nabla (a\F_1)-\gamma^*\Psi_*,
\end{align}
where $\tilde f_{1*}\in H^{s-2}_{ \partial\Omega}$ and hence, due to e.g. \cite[Theorem 2.10]{MikJMAA2011}, can be 
in turn represented as $\tilde f_{1*}=-\gamma^*\Psi_*$, with some $ \Psi_*\in H^{s-\tha}( \partial\Omega)$.
Then \eqref{2.6'genPhiDelta} is satisfied and 
\begin{align}
\label{4.8Psi2TTDelta}
\F_2=T^+_\Delta(\tilde f_*,a\F_1)
=(\gamma^{-1})^*[\tilde f_*-\check \Delta(a\F_1)]=(\gamma^{-1})^*\tilde f_{1*}=-(\gamma^{-1})^*\gamma^*\Psi_*=-\Psi_*,
\end{align}
because
$
\left\langle(\gamma^{-1})^*\gamma^*\Psi_*, w\right\rangle_{ \partial\Omega}
=\left\langle\gamma^*\Psi_*,\gamma^{-1}w\right\rangle_{ {\Omega}}
=\left\langle\Psi_{*},w\right\rangle_{ \partial\Omega}
$
for any $w\in H^{\tha-s}(\pO)$.
Hence \eqref{fDelta} reduces to \eqref{fDelta1}.

Now \eqref{4.8Psi1aDelta} can be written in the form
\begin{eqnarray}
W_\Delta (a\Phi_*)  = \F_\Delta \quad
\text{in } {\Omega},
\label{WFDelta}
\end{eqnarray}  
where 
\begin{eqnarray}
\label{4.8Psi1cDelta} \F_\Delta:=-a\F_1+{\cal P}_\Delta\tilde f_*
= -a\F_1+{\cal P}_\Delta[\check\Delta_{\Omega}(a\F_1)+\gamma^*\F_2], 
\end{eqnarray}  
is a harmonic function in $ {\Omega}$ due to \eqref{2.6'genPhi=Delta}.
The trace of equation \eqref{WFDelta} gives 
\begin{align}\label{Phi*eqDelta}
 -\ha a\Phi_* + {\cal W}_\Delta(a\Phi_*)=\gamma^+\F_\Delta\  \mbox{ on }  \partial\Omega.
\end{align}Since the operator 
$ -\ha I + {\cal W}_\Delta : H^{s-\frac{1}{2}}( \partial\Omega) \to H^{s-\frac{1}{2}}( \partial\Omega)$  is an 
isomorphism (see
Theorem~\ref{T3.6}), this implies 
\begin{align*} 
 \Phi_* =&\frac{1}{a}\left(-\ha I + {\cal W}_\Delta\right)^{-1}\gamma^+\F_\Delta 
=\frac{1}{a}\left(-\ha I + {\cal W}_\Delta\right)^{-1}\gamma^+
\{ -a\F_1+{\cal P}_\Delta[\check\Delta_{\Omega}(a\F_1)+\gamma^*\F_2]\},
\end{align*}
which coincides with \eqref{Phi*FDelta0}.

Relations \eqref{fDelta1},  \eqref{Phi*FDelta0} can be written as 
$(\tilde f_*,\Phi_{*})=\mathcal C_*(\F_1,\F_2)$,
where 
$\mathcal C_*: H^{1}( {\Omega})\times {H}^{-\ha}(\pO)\to \s{H}^{-1} ( {\Omega})\times {H}^{\ha}(\pO)$
is a linear continuous operator, as claimed.
We still have to check that the functions $\tilde f_*$ and $\Phi_{*}$, given by \eqref{fDelta1} and  \eqref{Phi*FDelta0}, 
satisfy equations \eqref{4.8Psi1Delta} and \eqref{4.8Psi2Delta}.
Indeed, $\Phi_{*}$ given by \eqref{Phi*FDelta0} satisfies equation \eqref{Phi*eqDelta} with $\F_\Delta$ given by \eqref{4.8Psi1cDelta}, and thus 
$\gamma^+W_\Delta (a\Phi_{*})  = \gamma^+\F_\Delta$.
Since both $W_\Delta (a\Phi_{*}) $ and  $\F_\Delta$ are harmonic functions belonging to the space $H^s(\Omega)$, this implies \eqref{WFDelta} 
and by \eqref{fDelta1} also \eqref{4.8Psi1Delta}.
Finally, \eqref{fDelta1} implies by \eqref{4.8Psi2TTDelta} that \eqref{4.8Psi2TDelta} is satisfied, and adding \eqref{6.32Delta} to it 
leads to \eqref{4.8Psi2Delta}.

Let us now prove that the operator $\mathcal C_*$ is unique.
Indeed, let a couple
$(\tilde f_*,\Phi_*)\in \s{H}^{s-2} ( {\Omega})\times {H}^{s-\ha}(\pO)$ 
be a solution of linear system \eqref{4.8Psi1Delta}-\eqref{4.8Psi2Delta} with $\F_1=0$ and $\F_2=0$. 
Then \eqref{2.6'genPhiDelta} implies that $r_ {\Omega} \tilde f_*=0$ in $ {\Omega}$, i.e., 
$\tilde f_*\in{H}^{s-2}_{\partial  {\Omega}}\subset \s{H}^{s-2} ( {\Omega})$.
Hence, \eqref{4.8Psi2TDelta} reduces to
$
0=T^+_\Delta(\tilde f_*,0)
$
on $\partial\Omega.$
By the first Green identity \eqref{Tgen}, this gives,
\begin{equation*}
0=\left\langle T^+_\Delta(\tilde f_*,0), \gamma^{+}v \right\rangle _{\pO}
=\langle \tilde f_*,v \rangle_ {\Omega} \quad  \forall\ v\in H^{2-s} ( {\Omega}),
\end{equation*}
which implies $\tilde f_*=0$ in $\R^n$. 
Finally, \eqref{Phi*FDelta0} gives $\Phi_*=0$. 
Hence, any solution of non-homogeneous linear system \eqref{4.8Psi1Delta}-\eqref{4.8Psi2Delta} has only one solution, which implies uniqueness of the operator $\mathcal C_*$. 
\end{proof}

\begin{theorem}\label{T3.521HFD}
Let $\Omega$ be a bounded simply-connected Lipschitz domain, $\ha<s<\tha$, 
and $a\in C_+^{\sigma}(\overline\Omega)$, $\sigma=\max\{1, s\}$.
The cokernel of operator \eqref{mfN1D} is spanned over the functional
\begin{align}\label{SSH10CD}
g^{*1\Delta}:=(0,1)^\top 
\end{align}
in  $[H^{s}(\Omega )\times H^{s-\tha}(\pO)]^*=\widetilde H^{-s}( {\Omega} )\times H^{\tha-s}(\pO)$, i.e.,
$
g^{*1\Delta}(\mathcal F_1,\mathcal F_2)
=\langle \F_2,\gamma^+ u^0\rangle_{ \partial\Omega},
$
where $u^0=1$.
\end{theorem} 
\begin{proof} 
Let us consider the equation $\mathfrak N^{1\Delta}\mathcal U=(\mathcal F_1,\mathcal F_2)^\top$, i.e., the BDIE system (N1$_\Delta$) for $(u,\varphi)\in H^{1}( {\Omega} )\times H^{\ha}(\pO)$,
\begin{align}
\label{2.44HFD}
 u +{\mathcal R}u  +W \varphi &= \mathcal F_1  \quad \mbox{in}\quad {\Omega},
\\[1ex]
 \label{2.45HFD}
T^+_\Delta(A^\nabla u;{a\cal R} u) +\mathcal L_\Delta(a \varphi)
  &=   \mathcal F_2  \quad   \mbox{on}\quad   \partial\Omega,
\end{align}
with arbitrary $(\mathcal F_1,\mathcal F_2)\in H^{s}( {\Omega} )\times H^{s-\tha}(\pO)$. 
By Lemma~\ref{L4D},  the right hand side of the system can be presented in form \eqref{4.8Psi1Delta}-\eqref{4.8Psi2Delta}, i.e., 
system \eqref{2.44HFD}-\eqref{2.45HFD} reduces to
\begin{align}
\label{2.44HFrD}
 u +{\mathcal R}u  +W (\varphi+\Phi_{*}) &= {\cal P}\tilde f_{*}  \quad \mbox{in}\quad    {\Omega},
\\[1ex]
 \label{2.45HFrD}
T^+_\Delta(A^\nabla u;{a\cal R} u) +\mathcal L_\Delta(a \varphi+a\Phi_{*})
  &=  T^+_\Delta(\tilde f_*;{\cal P}\tilde f_*)  \quad   \mbox{on}\quad   \partial\Omega,
\end{align}
where the couple $(\tilde f_{*}, \Phi_{*})\in \s{H}^{s-2} ( {\Omega})\times {H}^{s-\ha}(\pO)$ is given by \eqref{fDelta1}, \eqref{Phi*FDelta0}.
Up to the notations, system \eqref{2.44HFrD}-\eqref{2.45HFrD} is the same as \eqref{Basic1D}-\eqref{Basic2D} with the right hand side given by \eqref{FN1D}, where $\psi_0=0$. 

Let $s=1$ first. 
Then Theorems \ref{Rem1N} and \ref{equivalenceN} 
imply that BDIE system \eqref{2.44HFrD}-\eqref{2.45HFrD} and hence \eqref{2.44HFD}-\eqref{2.45HFD} is solvable if and only if
\begin{multline} \label{3.suf*}
\langle \tilde f_{*},u^0\rangle_ {\Omega} 
=
\langle \check\Delta_{\Omega}(a\F_1)+\gamma^*\F_2,u^0\rangle_ {\Omega}
=
\langle \nabla\cdot\mathring E_{\Omega} \nabla (a\F_1) +\gamma^*\F_2,u^0\rangle_{\R^n}\\
=
-\langle \mathring E_{\Omega} \nabla (a\F_1) ,\nabla u^0\rangle_{\R^n}
+\langle \F_2,\gamma^+ u^0\rangle_{ \partial\Omega} 
=\langle \F_2,\gamma^+ u^0\rangle_{ \partial\Omega}
=0,
\end{multline}
where we took into account that $u^0=1$ in $\R^n$.
Thus the functional $g^{*1\Delta}$ defined by \eqref{SSH10CD} generates the necessary and sufficient solvability condition of equation $\mathfrak N^{1\Delta}\mathcal U=(\mathcal F_1,\mathcal F_2)^\top$. 
Hence $g^{*1\Delta}$ is a basis of the cokernel of  $\mathfrak N^{1\Delta}$ (and thus the kernel of the operator $\mathfrak N^{1\Delta*}$ adjoint to $\mathfrak N^{1\Delta}$), for $s=1$.

Let us now choose any $s\in(\ha,\tha)$. 
By Theorem~\ref{FredholmN}, operator \eqref{mfN1D} and thus its adjoint are Fredholm with zero index. We already proved that at $s=1$ the kernel of the adjoint operator is spanned over $g^{*1\Delta}$.
For any fixed coefficient $a\in C_+^{\sigma}(\overline \Omega)$, the operator 
\begin{align}\label{mfN1a'}
\mathfrak N^{1\Delta}: H^{s'}( {\Omega})\times H^{s'-\frac{1}{2}}(\partial\Omega) 
\to H^{s'}({\Omega} )\times H^{s'-\frac{3}{2}}( \partial\Omega)
\end{align}
 is continuous for any $s'\in(\ha, \sigma]$ and particularly  for $s'=s$ and $s'=1$. 
Then Lemma~\ref{Lem2-Fredholm} implies that the  co-kernel of operator \eqref{mfN1a'} will be the same for $s'=s$ and $s'=1$  and is spanned over $g^{*1\Delta}$.
\end{proof}
\begin{lemma}\label{L4}
Let $\Omega$ be a bounded simply-connected Lipschitz domain, $\ha<s<\tha$, and $a\in C_+^{\tha}(\overline\Omega)$. 
For any couple $(\F_1,\F_2)\in H^{s}( {\Omega})\times {H}^{s-\tha}(\pO)$, 
there exists a unique couple $(\tilde f_{**}, \Phi_{**})\in \s{H}^{s-2} ({\Omega})\times 
{H}^{s-\ha}(\pO)$
 such that
\begin{eqnarray}
\label{4.8Psi1**} \F_1 &=&{\cal P}\tilde f_{**}-W\Phi_{**} \quad
\text{in } {\Omega},\\
\label{4.8Psi2**}
\F_2&=&T^+(\tilde f_{**}+\mathring E_{\Omega}\,{\mathcal R}_*\tilde f_{**};{\cal P}\tilde f_{**})-T^+ W\Phi_{**} 
  \quad
\text{on } \partial\Omega.
\end{eqnarray}
Moreover, $(\tilde f_{**}, \Phi_{**})={\cal C}_{**}(\F_1,\F_2)$ and
${\cal C}_{**}: H^{s}( {\Omega})\times {H}^{s-\tha}(\pO)\to \s{H}^{s-2} ( {\Omega})\times {H}^{s-\ha}(\pO)$ 
is a linear continuous operator given by 
\begin{align}\label{f**1}
\tilde f_{**}=&\check\Delta_{\Omega}(a\F_1)+\gamma^*(\F_2+(\gamma^+\F_1)\partial_n a),\\
\label{Phi*F**0}
 \Phi_{**} 
=&\frac{1}{a}\left(-\ha I + {\cal W}_\Delta\right)^{-1}\gamma^+
\{ -a\F_1+{\cal P}_\Delta[\check\Delta_{\Omega}(a\F_1)+\gamma^*(\F_2+(\gamma^+\F_1)\partial_n a)]\}.
\end{align}
\end{lemma}
\begin{proof}
Let us first assume that there exist $(\tilde f_{**}, \Phi_{**})\in \s{H}^{s-2} ( {\Omega})\times {H}^{s-\ha}(\pO)$
satisfying equations \eqref{4.8Psi1**}, \eqref{4.8Psi2**} and find their expressions in terms of $\F_1$ and $\F_2$.
Let us re-write \eqref{4.8Psi1**} as 
\begin{align}
\label{4.8Psi1a**} \F_1-{\cal P}\tilde f_{**} = - W\Phi_{**} \quad
\text{in } {\Omega},
\end{align}
Multiplying \eqref{4.8Psi1a**} by $a$ and applying the Laplace operator to it, we obtain, 
\begin{align}
\label{2.6'genPhi=**}
& \Delta(a\F_1-\P_\Delta \tilde f_{**})=\Delta(a\F_1)- \tilde f_{**}=-\Delta W_\Delta(a\Phi_{**})=0 
\text{in}   {\Omega},
\end{align}
which means
\begin{align}
\label{2.6'genPhi**}
& \Delta(a\F_1)=r_ {\Omega} \tilde f_{**} \text{in}   {\Omega}
\end{align}
and $a\F_1-\P_\Delta \tilde f_{**}\in H^{s,0}( {\Omega};\Delta)$. 
By equality \eqref{4.8Psi1a**} and continuity of operator \eqref{WHs1Ga} in Theorem~\ref{T3.1s0}, we also have 
$\F_1-\P \tilde f_{**}\in H^{1,0} ( {\Omega};A)$, which implies that the canonical conormal derivative 
$T^+(\F_1-\P \tilde f_{**})$ is well defined. 
Applying the canonical conormal derivative operator $T^+$ to the both sides of 
equation \eqref{4.8Psi1a**}, we obtain,
\begin{multline}\label{6.32**}
 -T^+ W \Phi_{**}=T^+(\F_1-\P \tilde f_{**})
=T^+(\tilde A(\F_1-\P\tilde f_{**});\F_1-\P \tilde f_{**})
=T^+(\mathring E_{\Omega} A(\F_1-\P\tilde f_{**});\F_1-\P \tilde f_{**})\\
=T^+(\mathring E_{\Omega} \nabla\cdot (a\nabla(\F_1-\P\tilde f_{**}));\F_1-\P \tilde f_{**})\\
=T^+(\mathring E_{\Omega} \Delta(a\F_1-\P_\Delta\tilde f_{**})
-\mathring E_{\Omega} \nabla\cdot ((\F_1-\P\tilde f_{**})\nabla a);\F_1-\P \tilde f_{**})\\
=T^+(-\mathring E_{\Omega}\nabla\cdot(\F_1\nabla a)
-\mathring E_{\Omega} \mathcal R_* \tilde f_{**};\F_1-\P \tilde f_{**}),
\end{multline}
where \eqref{2.6'genPhi**} and the third relation in \eqref{3.d1} were taken into account.
Substituting this into \eqref{4.8Psi2**}, we obtain
\begin{eqnarray}
\label{4.8Psi2T**}
\F_2&=&T^+(\tilde f_{**}-\mathring E_{\Omega}\nabla\cdot(\F_1\nabla a),\F_1)
  \quad
\text{on } \partial\Omega.
\end{eqnarray}

Due to \eqref{2.6'genPhi**}, we can represent
\begin{align}\label{f**}
\tilde f_{**}=\check\Delta_{\Omega}(a\F_1)+\tilde f_{1*}=\nabla\cdot\mathring E_{\Omega} \nabla (a\F_1)-\gamma^*\Psi_{**},
\end{align}
where $\tilde f_{1*}\in H^{s-2}_{ \partial\Omega}$ and hence, due to e.g. \cite[Theorem 2.10]{MikJMAA2011}, can be 
in turn represented as $\tilde f_{1*}=-\gamma^*\Psi_{**}$, with some $ \Psi_{**}\in H^{s-\tha}( \partial\Omega)$.
Then \eqref{2.6'genPhi**} is satisfied and 
\begin{multline}
\label{4.8Psi2TT**}
\F_2=T^+(\tilde f_{**}-\mathring E_{\Omega}\nabla\cdot(\F_1\nabla a),\F_1)
=(\gamma^{-1})^*[\tilde f_{**}-\mathring E_{\Omega}\nabla\cdot(\F_1\nabla a)-\check 
A\F_1]\\
=(\gamma^{-1})^*[\nabla\cdot\mathring E_{\Omega} \nabla (a\F_1) -\gamma^*\Psi_{**}
-\mathring E_{\Omega}\nabla\cdot(\F_1\nabla a)-\nabla\cdot\mathring E_{\Omega} (a\nabla \F_1)]\\
=(\gamma^{-1})^*[\nabla\cdot\mathring E_{\Omega} (\F_1\nabla a)-\gamma^*\Psi_{**}-\mathring E_{\Omega}\nabla\cdot(\F_1\nabla a)]
=-\Psi_{**}-(\gamma^+\F_1)\partial_na,
\end{multline}
because for any $w\in H^{\tha-s}(\pO)$,
\begin{multline*}
\left\langle(\gamma^{-1})^*[\nabla\cdot\mathring E_{\Omega} (\F_1\nabla a)-\gamma^*\Psi_{**}
-\mathring E_{\Omega}\nabla\cdot(\F_1\nabla a)], w\right\rangle_{ \partial\Omega}\\
=\left\langle\nabla\cdot\mathring E_{\Omega} (\F_1\nabla a)-\gamma^*\Psi_{**}
-\mathring E_{\Omega}\nabla\cdot(\F_1\nabla a),\gamma^{-1}w\right\rangle_{ {\Omega}}\\
=\left\langle\nabla\cdot\mathring E_{\Omega} (\F_1\nabla a),\gamma^{-1}w\right\rangle_{\R^n}-\Psi_{**}
-\left\langle\mathring E_{\Omega}\nabla\cdot(\F_1\nabla a),\gamma^{-1}w\right\rangle_{ {\Omega}}\\
=-\left\langle\mathring E_{\Omega} (\F_1\nabla a),\nabla\gamma^{-1}w\right\rangle_{\R^n}-\Psi_{**}
+\left\langle\F_1\nabla a,\nabla\gamma^{-1}w\right\rangle_{ {\Omega}}
-\left\langle n\cdot\gamma^+(\F_1\nabla a),\gamma^+\gamma^{-1}w\right\rangle_{ {\Omega}}\\
=-\left\langle(\gamma^+\F_1)\partial_na,w\right\rangle_{ \partial\Omega}
-\left\langle\Psi_{**},w\right\rangle_{ \partial\Omega}.
\end{multline*}
Hence \eqref{4.8Psi2T**} reduces to
$
\Psi_{**}=-\F_2-(\gamma^+\F_1)\partial_n a,
$
and \eqref{f**} to \eqref{f**1}.

Now \eqref{4.8Psi1a**} can be written in the form
\begin{eqnarray}
W_\Delta (a\Phi_{**})  = \F_\Delta \quad
\text{in } {\Omega},
\label{WF**}
\end{eqnarray}  
where 
\begin{eqnarray}
\label{4.8Psi1c**} \F_\Delta:=-a\F_1+{\cal P}_\Delta\tilde f_{**}
= -a\F_1+{\cal P}_\Delta[\check\Delta_{\Omega}(a\F_1)+\gamma^*(\F_2+(\gamma^+\F_1)\partial_n a)], \quad
\end{eqnarray}  
is a harmonic function in $ {\Omega}$ due to \eqref{2.6'genPhi=**}.
The trace of equation \eqref{WF**} gives 
\begin{align}\label{Phi*eq**}
 -\ha a\Phi_{**} + {\cal W}_\Delta(a\Phi_{**})=\gamma^+\F_\Delta\  \mbox{ on }  \partial\Omega.
\end{align}Since the operator 
$ -\ha I + {\cal W}_\Delta : H^{s-\frac{1}{2}}( \partial\Omega) \to H^{s-\frac{1}{2}}( \partial\Omega)$  is an 
isomorphism (see
Theorem~\ref{T3.6}), this implies 
\begin{align*} 
 \Phi_{**} =&\frac{1}{a}\left(-\ha I + {\cal W}_\Delta\right)^{-1}\gamma^+\F_\Delta 
=\frac{1}{a}\left(-\ha I + {\cal W}_\Delta\right)^{-1}\gamma^+
\{ -a\F_1+{\cal P}_\Delta[\check\Delta_{\Omega}(a\F_1)+\gamma^*(\F_2+(\gamma^+\F_1)\partial_n a)]\},
\end{align*}
which coincides with \eqref{Phi*F**0}.

Relations \eqref{f**1},  \eqref{Phi*F**0} can be written as 
$(\tilde f_*,\Phi_{*})={\cal C}_{**}(\F_1,\F_2)$,
where 
${\cal C}_{**}: H^{1}( {\Omega})\times {H}^{-\ha}(\pO)\to \s{H}^{-1} ( {\Omega})\times {H}^{\ha}(\pO)$
is a linear continuous operator, as claimed.
We still have to check that the functions $\tilde f_{**}$ and $\Phi_{*}$, given by \eqref{f**1} and  \eqref{Phi*F**0}, 
satisfy equations \eqref{4.8Psi1**} and \eqref{4.8Psi2**}.
Indeed, $\Phi_{*}$ given by \eqref{Phi*F**0} satisfies equation \eqref{Phi*eq**} and thus 
$\gamma^+W_\Delta (a\Phi_{*})  = \gamma^+\F_\Delta$.
Since both $W_\Delta (a\Phi_{*}) $ and  $\F_\Delta$ are harmonic functions belonging to the space $H^s(\Omega)$, this implies \eqref{WF**}-\eqref{4.8Psi1c**} 
and by \eqref{f**1} also \eqref{4.8Psi1**}.
Finally, \eqref{f**1} implies by \eqref{4.8Psi2TT**} that \eqref{4.8Psi2T**} is satisfied, and adding \eqref{6.32**} to it 
leads to \eqref{4.8Psi2**}.

Let us now prove that the operator ${\cal C}_{**}$ is unique.
Indeed, let a couple
$(\tilde f_{**},\Phi_{**})\in \s{H}^{s-2} ( {\Omega})\times {H}^{s-\ha}(\pO)$ 
be a solution of linear system \eqref{4.8Psi1**}-\eqref{4.8Psi2**} with $\F_1=0$ and $\F_2=0$. 
Then \eqref{2.6'genPhi**} implies that $r_ {\Omega} \tilde f_{**}=0$ in $ {\Omega}$, i.e., 
$\tilde f_{**}\in{H}^{s-2}_{\partial  {\Omega}}\subset \s{H}^{s-2} ( {\Omega})$.
Hence, \eqref{4.8Psi2T**} reduces to
$
0=T^+(\tilde f_{**},0)
$
on $\partial\Omega.$
By the first Green identity \eqref{Tgen}, this gives,
\begin{equation*}
0=\left\langle T^+(\tilde f_{**},0), \gamma^{+}v \right\rangle _{\pO}
=\langle \tilde f_{**},v \rangle_ {\Omega} \quad  \forall\ v\in H^{2-s} ( {\Omega}),
\end{equation*}
which implies $\tilde f_{**}=0$ in $\R^n$. 
Finally, \eqref{Phi*F**0} gives $\Phi_{**}=0$. 
Hence, any solution of non-homogeneous linear system \eqref{4.8Psi1**}-\eqref{4.8Psi2**} has only one solution, which implies uniqueness of the operator ${\cal C}_{**}$. 
\end{proof}
\begin{theorem}\label{T3.521HF}
Let $\Omega$ be a bounded simply-connected Lipschitz domain, $\ha<s<\tha$, 
and $a\in C_+^{\tha}(\overline\Omega)$.
The cokernel of operator \eqref{mfN1} is spanned over the functional
\begin{align}\label{SSH10C}
g^{*1}:=((\gamma^{+})^*\partial_n a,1)^\top 
\end{align}
in  $[H^{s}(\Omega )\times H^{s-\tha}(\pO)]^*=\widetilde H^{-s}( {\Omega} )\times H^{\tha-s}(\pO)$, i.e.,
$
g^{*1}(\mathcal F_1,\mathcal F_2)
=\langle (\gamma^+\F_1)\partial_n a +\F_2,\gamma^+ u^0\rangle_{ \partial\Omega},
$
where $u^0=1$.
\end{theorem} 
\begin{proof} 
Let us consider the equation $\mathfrak N^{1}\mathcal U=(\mathcal F_1,\mathcal F_2)^\top$, i.e., the BDIE system (N1) for $(u,\varphi)\in H^{s}( {\Omega} )\times H^{s-\ha}(\pO)$,
\begin{align}
\label{2.44HF}
 u +{\mathcal R}u  +W \varphi &= \mathcal F_1  \quad \mbox{in}\quad {\Omega},
\\[1ex]
 \label{2.45HF}
 T^+{\mathcal R} u +T^+ W^+\varphi
  &=   \mathcal F_2  \quad   \mbox{on}\quad   \partial\Omega,
\end{align}
with arbitrary $(\mathcal F_1,\mathcal F_2)\in H^{s}( {\Omega} )\times H^{s-\tha}(\pO)$. 
By Lemma~\ref{L4},  the right hand side of the system has form \eqref{4.8Psi1**}-\eqref{4.8Psi2**}, i.e., 
system \eqref{2.44HF}-\eqref{2.45HF} reduces to
\begin{eqnarray}
\label{2.44HFr}
 u +{\mathcal R}u  +W (\varphi+\Phi_{**}) &=& {\cal P}\tilde f_{**}  \quad \mbox{in}\quad    {\Omega},
\\[1ex]
 \label{2.45HFr}
 T^+{\mathcal R} u +T^+ W(\varphi+\Phi_{**}) 
  &=&  T^+(\tilde f_{**}+\mathring E_{\Omega}\,{\mathcal R}_*\tilde f_{**},{\cal P}\tilde f_{**})  \quad   \mbox{on}\quad   \partial\Omega,
\end{eqnarray}
where the couple $(\tilde f_{**}, \Phi_{**})\in \s{H}^{s-2} ( {\Omega})\times {H}^{s-\ha}(\pO)$ is given by \eqref{f**1}, \eqref{Phi*F**0}.
Up to the notations, system \eqref{2.44HFr}-\eqref{2.45HFr} is the same as \eqref{Basic1}-\eqref{Basic2} with the right hand side given by \eqref{FN1}, where $\psi_0=0$. 

Let $s=1$ first. 
Then Theorems \ref{Rem1N} and \ref{equivalenceN} 
imply that BDIE system \eqref{2.44HFr}-\eqref{2.45HFr} and hence \eqref{2.44HF}-\eqref{2.45HF} is solvable if and only if
\begin{multline} \label{3.suf*2}
\langle \tilde f_{**},u^0\rangle_ {\Omega} 
=
\langle \check\Delta_{\Omega}(a\F_1)+\gamma^*(\F_2+(\gamma^+\F_1)\partial_n a),u^0\rangle_ {\Omega}
=
\langle \nabla\cdot\mathring E_{\Omega} \nabla (a\F_1) +\gamma^*(\F_2+(\gamma^+\F_1)\partial_n a),u^0\rangle_{\R^n}\\
=
-\langle \mathring E_{\Omega} \nabla (a\F_1) ,\nabla u^0\rangle_{\R^n}
+\langle \F_2+(\gamma^+\F_1)\partial_n a,\gamma^+ u^0\rangle_{ \partial\Omega} 
=\langle (\gamma^+\F_1)\partial_n a +\F_2,\gamma^+ u^0\rangle_{ \partial\Omega}
=0,
\end{multline}
where we took into account that $u^0=1$ in $\R^n$.
Thus the functional $g^{*1}$ defined by \eqref{SSH10C} generates the necessary and sufficient solvability condition of equation $\mathfrak N^{1}\mathcal U=(\mathcal F_1,\mathcal F_2)^\top$. 
Hence $g^{*1}$ is a basis of the cokernel of  $\mathfrak N^{1}$ (and thus the kernel of the operator adjoint to $\mathfrak N^{1}$), for $s=1$.

Let now $s\in(\ha,\tha)$. By Theorem~\ref{FredholmN}, operator \eqref{mfN1} and thus its adjoint are Fredholm with zero index. We already proved that at $s=1$ the kernel of the adjoint operator is spanned over $g^{*1}$.
Then Lemma~\ref{Lem2-Fredholm} implies that the  kernel will be the same for any $s\in(\ha,\tha)$.
\end{proof}

To find the cokernel of operator \eqref{mfN2},  we will 
need some auxiliary assertions. 
Lemma~\ref{lemma3P1} and Theorem~\ref{teoremP1inv} were proved in \cite[Lemma 6.4 and Theorem 6.5]{MikArxiv2015} for the infinitely smooth coefficient $a$ and boundary $\partial\Omega$. Below we only slightly modified them for the non-smooth coefficients and Lipschitz boundary.

\begin{lemma}\label{lemma3P1}
Let $\Omega$ be a bounded simply-connected Lipschitz domain, $s>\ha$,  $a\in C_+^{s}(\overline\Omega)$ and $\tilde f\in \widetilde H^{s-2}( {\Omega})$.  If
\begin{equation}
 \label{B.190}
r_ {\Omega}\mathbf P \,\tilde f=0\quad \text{in } {\Omega},
\end{equation}
then $\tilde f=0$ in $\R^n$.
\end{lemma}
\begin{proof} 
Multiplying \eqref{B.190} by $a$, taking into account \eqref{3.d1}  and applying the Laplace operator, we obtain $r_ {\Omega}\tilde f=0$, which means $\tilde f\in H^{s-2}_{ \partial\Omega}$. If $s\ge\tha$, then $\tilde f=0$ by Theorem 2.10 from \cite{MikJMAA2011}. If $\ha<s<\tha$, then by the same theorem  there exists $v\in H^{s-\tha}( \partial\Omega)$ such that $\tilde f=\gamma^*v$. This gives
$\mathbf P \,\tilde f=\mathbf P \,\gamma^*v=-Vv$ in $\R^n$, cf. \eqref{3.6d}.
Then \eqref{B.190} reduces to
$Vv=0$ in $ {\Omega}$, which implies  $v=0$ on $ \partial\Omega$
by Lemma \ref{VW0s}(i), and thus $\tilde f=0$ in $\R^n$.
\end{proof} 

\begin{theorem}\label{teoremP1inv}
Let $\Omega$ be a bounded simply-connected Lipschitz domain, $\ha<s<\tha$, and $a\in C_+^{s}(\overline\Omega)$. 
The operator
\begin{equation}
 \label{B.19wP1}
r_ {\Omega}\mathbf P:\widetilde H^{s-2}( {\Omega} )\to H^{s}( {\Omega} )
\end{equation}
and its inverse
\begin{equation}
 \label{B.19wP1in}
(r_ {\Omega}\mathbf P)^{-1}: H^{s}( {\Omega} )\to \widetilde H^{s-2}( {\Omega} )
\end{equation}
are continuous and
\be\label{B.194a}
(r_ {\Omega}\mathbf P)^{-1} g
=[\Delta \mathring E_{\Omega} (I- r_ {\Omega} V_\Delta \mathcal{V}_{\Delta}^{-1}\gamma^+)
-\gamma^*\mathcal{V}_{\Delta}^{-1}\gamma^+](ag)\ \mbox{ in } \R^n,\quad \forall g\in  H^{s}( {\Omega} ).
\ee
\end{theorem}
\begin{proof} 
The continuity of \eqref{B.19wP1} is given by Theorem~\ref{T3.1P}. By Lemma~\ref{lemma3P1}, operator \eqref{B.19wP1} is injective. Let us prove its surjectivity. To this end, for arbitrary $g\in H^{s}( {\Omega} )$ let us consider the following equation with respect to $\tilde f\in \widetilde H^{s-2}( {\Omega} )$,
\begin{equation}
 \label{B.191}
r_ {\Omega}\mathbf P_\Delta  \,\tilde f=g \text{ in } {\Omega} .
\end{equation}

Let $g_1\in H^{s}( {\Omega} )$ be the (unique) solution of the following Dirichlet problem: $\Delta g_1=0$ in $ {\Omega}$, $\gamma^+g_1=\gamma^+g$, which can be particularly presented as $g_1=V_\Delta \mathcal{V}_{\Delta}^{-1}\gamma^+g,$ see e.g \cite{Costabel1988} or proof of Lemma 2.6 in \cite{MikJMAA2011}.
Let $g_0:=g-r_ {\Omega} g_1$. 
Then $g_0\in H^s( {\Omega})$ and $\gamma^+g_0=0$ and thus $g_0$ can be uniquely extended  to $\mathring E_{\Omega} g_0\in \widetilde H^s( {\Omega})$. 
Thus by  \eqref{3.6d}, equation \eqref{B.191} takes form
\begin{equation}
 \label{B.192}
r_ {\Omega}\mathbf P_\Delta  [\tilde f+\gamma^*\mathcal{V}_{\Delta}^{-1}\gamma^+g]=g_0 \text{ in } {\Omega}.
\end{equation}
Any solution $\tilde f\in \widetilde H^{s-2}( {\Omega} )$ of the corresponding equation in $\R^n$,
\begin{equation}
 \label{B.193}
\mathbf P_\Delta  [\tilde f+\gamma^*\mathcal{V}_{\Delta}^{-1}\gamma^+g]= \mathring E_{\Omega} g_0\ \mbox{ in } \R^n,
\end{equation}
will evidently solve \eqref{B.192}. If $\tilde f$ solves \eqref{B.193} then applying the Laplace operator to \eqref{B.193}, we obtain
\be\label{B.194}
\tilde f=\tilde Q g:= \Delta \mathring E_{\Omega} g_0 -\gamma^*\mathcal{V}_{\Delta}^{-1}\gamma^+g
=\Delta \mathring E_{\Omega} (g- r_ {\Omega} V_\Delta \mathcal{V}_{\Delta}^{-1}\gamma^+g)
-\gamma^*\mathcal{V}_{\Delta}^{-1}\gamma^+g\ \mbox{ in } \R^n.
\ee
On the other hand, substituting $\tilde f$ given by \eqref{B.194} to \eqref{B.193} and taking into account that $\mathbf P_\Delta  \Delta\tilde h=\tilde h$ for any $\tilde h\in \widetilde H^{s}( {\Omega} )$, $s\in\R$, we obtain that $\tilde Q g$ is indeed a solution of equation \eqref{B.193} and thus \eqref{B.192}. 
By Lemma~\ref{lemma3P1} the solution of \eqref{B.192} is unique, which means that the operator $\tilde Q$ is inverse to operator \eqref{B.19wP1}, i.e., $\tilde Q=(r_ {\Omega}\mathbf P)^{-1}$.
Since $\Delta$ is a continuous operator from $\widetilde H^{s}( {\Omega} )$ to $\widetilde H^{s-2}( {\Omega} )$, equation \eqref{B.194} implies that the operator $(r_ {\Omega}\mathbf P)^{-1}=\tilde Q:H^{s}( {\Omega} )\to \widetilde H^{s-2}( {\Omega} )$ is continuous.
The relations $\mathbf P =\frac{1}{a}\mathbf P_\Delta $ and $a(x)>c>0$ then imply invertibility of operator \eqref{B.19wP1} and ansatz \eqref{B.194a}.
\end{proof}

\begin{theorem}\label{T3.521HF2}
Let $\Omega$ be a bounded simply-connected Lipschitz domain, $\ha< s< \tha$, and $a\in C_+^{\sigma}(\overline\Omega)$, $\sigma=\max\{1, s\}$. 
The cokernel of operator \eqref{mfN2} is spanned over the functional
\begin{align}\label{SSH10C2}
g^{*2}:=\left(\begin{array}{c}
-a\gamma^{+*}\left(\frac{1}{2}+\W'_\Delta\right)\mathcal{V}_{\Delta}^{-1}\gamma^+u^0
\\[1ex] 
-a\left(\frac{1}{2}-\W'_\Delta\right)\mathcal{V}_{\Delta}^{-1}\gamma^+u^0
\end{array}\right)
\end{align}
in  $[H^{s}(\Omega )\times H^{s-\ha}(\pO)]^*=\widetilde H^{-s}( {\Omega} )\times H^{\ha-s}(\pO)$, i.e.,
$$
g^{*2}(\mathcal F_1,\mathcal F_2)
=\left\langle -a\gamma^{+*}\left(\frac{1}{2}+\W'_\Delta\right)\mathcal{V}_{\Delta}^{-1}\gamma^+u^0,\F_1\right\rangle_{ {\Omega}}
+\left\langle -a\left(\frac{1}{2}-\W'_\Delta\right)\mathcal{V}_{\Delta}^{-1}\gamma^+u^0,\F_2 \right\rangle_{ \partial\Omega},
$$
where $u^0(x)=1$.
\end{theorem}
\begin{proof} 
Let us consider the equation $\mathfrak N^{2}\mathcal U=(\mathcal F_1,\mathcal F_2)^\top$, i.e., the BDIE system (N2),
\begin{eqnarray}
\label{2.44HF2}
 u +{\mathcal R}u  +W \varphi &=& \mathcal F_1  \quad \mbox{in}\quad    {\Omega},
\\[1ex]
 \label{2.45HF2}
 \frac{1}{2}\varphi+\gamma^+{\mathcal R} u   +{\mathcal W} \varphi
  &=&   \mathcal F_2  \quad   \mbox{on}\quad   \partial\Omega,
\end{eqnarray}
with arbitrary $(\mathcal F_1,\mathcal F_2)\in H^{s}( {\Omega} )\times H^{s-\ha}(\pO)$ for $(u,\varphi)\in H^{s}( {\Omega} )\times H^{s-\ha}(\pO)$. 

Introducing the new variable, $\varphi'=\varphi-(\F_2-\gamma^+\F_1)$, BDIE system \eqref{2.44HF2}-\eqref{2.45HF2} takes form
\begin{eqnarray}
\label{2.44HF2'}
 u +{\mathcal R}u  +W \varphi' &=& \mathcal F'_1  \quad \mbox{in}\quad    {\Omega},
\\[1ex]
 \label{2.45HF2'}
 \frac{1}{2}\varphi'+\gamma^+{\mathcal R} u   +{\mathcal W} \varphi'
  &=&   \gamma^+\mathcal F_1'  \quad   \mbox{on}\quad   \partial\Omega,
\end{eqnarray}
where 
$$
\mathcal F_1'=\mathcal F_1-W(\F_2-\gamma^+\F_1)\in H^{s}( {\Omega} ).
$$
On the other hand, by Theorem~\ref{teoremP1inv}, we can always represent $\mathcal F_1'=\mathcal P \tilde f_*$, with 
$$ \tilde f_*=[\Delta \mathring E_{\Omega} (I- r_ {\Omega} V_\Delta \mathcal{V}_{\Delta}^{-1}\gamma^+)
-\gamma^{+*}\mathcal{V}_{\Delta}^{-1}\gamma^+](a\mathcal F_1')\in \widetilde H^{s-2}( {\Omega} ).
$$
For $\mathcal F_1'=\mathcal P \tilde f_*$, the right hand side of BDIE system \eqref{2.44HF2}-\eqref{2.45HF2} is the same as in \eqref{FN2} with $\tilde f=\tilde f_*$ and $\psi_0=0$. 

Let $s=1$ first. Then Theorems \ref{Rem1N} and \ref{equivalenceN}  imply that BDIE system \eqref{2.44HF2'}-\eqref{2.45HF2'} is solvable if and only if
\begin{multline}\label{3.suf*'}
\langle \tilde f_*,u^0\rangle_ {\Omega} 
=
\langle [\Delta \mathring E_{\Omega} (I- r_ {\Omega} V_\Delta \mathcal{V}_{\Delta}^{-1}\gamma^+)
-\gamma^{+*}\mathcal{V}_{\Delta}^{-1}\gamma^+](a\F'_1),u^0\rangle_{\R^n}\\
=
\langle \mathring E_{\Omega} (I- r_ {\Omega} V_\Delta \mathcal{V}_{\Delta}^{-1}\gamma^+)(a\F'_1),\Delta u^0\rangle_{\R^n}
-\langle \mathcal{V}_{\Delta}^{-1}\gamma^+(a\F'_1),\gamma^{+} u^0\rangle_{ \partial\Omega}
=-\left\langle\gamma^+(a\F'_1),\mathcal{V}_{\Delta}^{-1}\gamma^+u^0\right\rangle_{ \partial\Omega}\\
=-\left\langle\frac{1}{2}[\gamma^+(a\F_1)+(a\F_2)]-\W_\Delta[a(\F_2-\gamma^+\F_1)] ,\mathcal{V}_{\Delta}^{-1}\gamma^+u^0\right\rangle_{ \partial\Omega}
\\
=-\left\langle \F_1, 
a\gamma^{+*}\left(\frac{1}{2}+\W'_\Delta\right)\mathcal{V}_{\Delta}^{-1}\gamma^+u^0\right\rangle_{ {\Omega}}-\left\langle \F_2, a\left(\frac{1}{2}-\W'_\Delta\right)\mathcal{V}_{\Delta}^{-1}\gamma^+u^0 \right\rangle_{ \partial\Omega}
=0.
\end{multline}
Thus the functional $g^{*2}$ defined by \eqref{SSH10C2} generates the necessary and sufficient solvability condition of equation $\mathfrak N^{2}\mathcal U=(\mathcal F_1,\mathcal F_2)^\top$. Hence $g^{*2}$ is a basis of the cokernel of  operator \eqref{mfN2} for s=1.

Let us now choose any $s\in(\ha,\tha)$. 
By Theorem~\ref{FredholmN}, operator \eqref{mfN2} and thus its adjoint are Fredholm with zero index. We already proved that at $s=1$ the kernel of the adjoint operator is spanned over $g^{*2}$.
For any fixed coefficient $a\in C_+^{\sigma}(\overline \Omega)$, the operator 
\begin{align}\label{mfN2a'}
\mathfrak N^{2}: H^{s'}( {\Omega})\times H^{s'-\frac{1}{2}}(\partial\Omega) 
\to H^{s'}({\Omega} )\times H^{s'-\frac{1}{2}}( \partial\Omega)
\end{align}
 is continuous for any $s'\in(\ha, \sigma]$ and particularly  for $s'=s$ and $s'=1$. 
Then Lemma~\ref{Lem2-Fredholm} implies that the  co-kernel of operator \eqref{mfN2a'} will be the same for $s'=s$ and $s'=1$  and is spanned over $g^{*2}$.
\end{proof}

Theorems \ref{equivalenceN}, \ref{KernelN} and \ref{T3.521HFD} (or \ref{T3.521HF}) imply the following extension of Theorem \ref{Rem1N} to the range $\ha< s< \tha$.
\begin{corollary} Let $\Omega$ be a bounded simply-connected Lipschitz domain,  $\ha< s< \tha$, $\tilde f\in\widetilde{H}^{s-2}({\Omega})$, $ \psi_0 \in H^{s-\tha}(\pO)$, and 
$a\in C_+^{\sigma}(\overline \Omega)$, $\sigma=\max\{1, s\}$. 

The homogeneous Neumann problem, \eqref{2.6N}-\eqref{2.8}, admits only one linearly independent solution $u^0=1$ in ${H}^{s}( {\Omega})$.
The non-homogeneous Neumann problem \eqref{2.6N}-\eqref{2.8} is solvable in ${H}^{s}( {\Omega})$ if and only if condition \eqref{3.suf} is satisfied.
\end{corollary}
\begin{proof}
Assuming that a function $u$ is a solution of the homogeneous Neumann problem, by Theorem \ref{equivalenceN} the couple $(u,\varphi)=(u,\gamma_+\varphi)$ solves the homogeneous BDIE system {\rm(N1$_\Delta$)}, and then Theorem \ref{T3.521HFD} implies that $u$ is spanned over $u^0=1$.

Assume that solvability condition \eqref{3.suf} is satisfied. Then the right-hand side \eqref{FN1} of the BDIE system {\rm(N1$_\Delta$)} satisfies its solvability condition 
$g^{*1\Delta}(\mathcal F_1,\mathcal F_2)=\langle \F_2,\gamma^+ u^0\rangle_{ \partial\Omega}=0$ 
given by Theorem \ref{T3.521HFD}. 
Indeed, due to the first Green identities \eqref{Tgen} and \eqref{Tcan} applied for the operator $\Delta$ and taking into account that ${\W}_\Delta\psi_0$ is a harmonic function in $\Omega$ and $u^0=1$, we obtain,
\begin{multline}\label{6.82}
\langle \F_2,\gamma^+ u^0\rangle_{ \partial\Omega}
=\langle T^+_\Delta(\tilde f;\P_\Delta\tilde f)-\dfrac{1}{2}\psi_0
+{\W}'_\Delta\psi_0,\gamma^+ u^0\rangle_{ \partial\Omega}\\
=\langle T^+_\Delta(\tilde f;\P_\Delta\tilde f)-\psi_0
+T^+_\Delta{\W}_\Delta\psi_0,\gamma^+ u^0\rangle_{ \partial\Omega}
=\langle \tilde f,u^0\rangle_{\Omega}
-\langle\psi_0,\gamma^+ u^0\rangle_{ \partial\Omega}.
\end{multline}
Hence the BDIE system {\rm(N1$_\Delta$)} is solvable, implying solvability of the Neumann BVP due to Theorem \ref{equivalenceN}(ii). 
This proves that condition \eqref{3.suf} is sufficient.

Let us now assume that there exists a solution of the Neumann BVP. 
Hence Theorem \ref{equivalenceN}(i) implies that the BDIE system {\rm(N1$_\Delta$)} with the right-hand side \eqref{FN1} is solvable, implying that its solvability condition $\langle \F_2,\gamma^+ u^0\rangle_{ \partial\Omega}=0$ is satisfied.
Then \eqref{6.82} implies condition \eqref{3.suf}, proving that it is necessary.
\end{proof}

\subsection{Perturbed (stabilised) segregated BDIE systems for the Neumann problem}\label{PSBDIE}

Theorem \ref{KernelN} implies, that even when the solvability condition \eqref{3.suf} is satisfied, the solutions of BDIE systems (N1$_\Delta$), (N1) and (N2) are not unique, and moreover, the BDIE left hand side operators, $\mathfrak N^{1\Delta}$, $\mathfrak N^1$ and $\mathfrak N^2$, have non-zero kernels and thus are not invertible. 
To find a solution $(u,\varphi)$ from uniquely solvable BDIE systems with continuously invertible left hand side operators,
let us consider, following \cite{MikPert1999}, some stabilised BDIE systems obtained from  (N1$_\Delta$), (N1) and (N2) by finite-dimensional operator perturbations. Note that other choices of the perturbing operators are also possible.

Below we use the notations $\mathcal U^N=(u,\varphi)^\top$, $\mathcal U^0=(1,1)^\top$, and $|\pO|:=\int_\pO dS$.


Let us introduce the perturbed counterparts of the BDIE systems (N1$_\Delta$), (N1) and (N2)
\begin{align}\label{N1-p}
\hat{\mathfrak N}^{1\Delta}\mathcal U^N=\mathcal F^{N1\Delta},\quad 
\hat{\mathfrak N}^{1}\mathcal U^N=\mathcal F^{N1},\quad 
\hat{\mathfrak N}^{2}\mathcal U^N=\mathcal F^{N2},
\end{align}
where $\hat{\mathfrak N}^{1\Delta}:={\mathfrak N}^{1\Delta}+\mathring{\mathfrak N}^{1\Delta},$
$\hat{\mathfrak N}^{1}:={\mathfrak N}^{1}+\mathring{\mathfrak N}^{1},$
$\hat{\mathfrak N}^{2}:={\mathfrak N}^{2}+\mathring{\mathfrak N}^{2}$
and
\begin{align}\label{6.83a}
\mathring{\mathfrak N}^{1\Delta}=\mathring{\mathfrak N}^1\mathcal U^N(y):=
g^0(\mathcal U^N)\,{\mathcal G}^1(y)=\frac{1}{|\pO|}\int_\pO \varphi(x)dS\ 
\left(\begin{array}{c}0\\ 1\end{array}\right) ,
\end{align}
that is, 
\be \label{Gg*A}
g^0(\mathcal U^N):=\frac{1}{|\pO|}\int_\pO \varphi(x)dS,
\quad
{\mathcal G}^1(y):=\left(\begin{array}{c}0\\ 1\end{array}\right),
\ee
while
$$
\mathring{\mathfrak N}^2\mathcal U^N:=
g^0(\mathcal U^N)\,{\mathcal G}^2=\frac{1}{|\pO|}\int_\pO \varphi(x)dS 
\left(\begin{array}{c}a^{-1}(y)\\ \gamma^+a^{-1}(y)\end{array}\right)  ,
$$
that is, $g^0(\mathcal U^N)$ as in \eqref{Gg*A} and
$
{\mathcal G}^2:=\left(\begin{array}{c}a^{-1}(y)\\ \gamma^+a^{-1}(y)\end{array}\right).
$

 \begin{theorem}
 Let $\Omega$ be a bounded simply-connected Lipschitz domain and $\ha< s< \tha$.
 
(i) The following operators 
are continuous and continuously invertible.
\begin{eqnarray}
\hat{\mathfrak N}^{1\Delta}&:& H^s( {\Omega})\times H^{s-\frac{1}{2}}(\partial\Omega) 
 \to H^s( {\Omega} )\times H^{s-\frac{3}{2}}( \partial\Omega)\quad\mbox{if }
 a\in C_+^{\sigma}(\overline \Omega),\label {mfN1Da^}\\
\hat{\mathfrak N}^1&:& H^s( {\Omega})\times H^{s-\frac{1}{2}}(\partial\Omega) 
 \to H^s({\Omega} )\times H^{s-\frac{3}{2}}( \partial\Omega)\quad\mbox{if }
 a\in C_+^{\tha}(\overline \Omega),\label {mfN1a^}\\
\hat{\mathfrak N}^2&:&H^s( {\Omega})\times H^{s-\frac{1}{2}}({\partial\Omega})
\to H^s( {\Omega} )\times H^{s-\frac{1}{2}}({ \partial\Omega})\quad\mbox{if }
a\in C_+^{\sigma}(\overline \Omega).
\label{mfN2a^}
  \end{eqnarray}
(ii)
If condition $g^{*1\Delta}(\mathcal F^{N1\Delta})=0$, $g^{*1}(\mathcal F^{N1})=0$ or $g^{*2}(\mathcal F^{N2})=0$ 
are satisfied, then the unique solutions of the corresponding perturbed BDIE systems in \eqref{N1-p} give the same solution of original BDIE systems (N1$_\Delta$), (N1) and (N2) such that
\bes 
g^0(\mathcal U^N)=\frac{1}{|\pO|}\int_\pO \varphi\, dS=\frac{1}{|\pO|}\int_\pO \gamma^+ u\, dS=0.
\ees
\end{theorem}
\begin{proof}
For the functional $g^{*1\Delta}$ given by \eqref{SSH10CD} in Theorem~\ref{T3.521HFD}, 
$g^{*1\Delta}({\mathcal G}^1)=|\pO|$.
Similarly, for the functional $g^{*1}$ given by \eqref{SSH10C} in Theorem~\ref{T3.521HF}, 
$g^{*1}({\mathcal G}^1)=|\pO|$.

For the functional $g^{*2}$ given by \eqref{SSH10C2} in Theorem~\ref{T3.521HF2}, since the operator $\mathcal{V}_{\Delta}^{-1}:H^\ha(\pO)\to H^{-\ha}(\pO)$ is positive definite and $u^0(x)=1$, there exists a positive constant $C$ such that
\begin{multline}
g^{*2}({\mathcal G}^2)=
\left\langle -a\gamma^{+*}\left(\frac{1}{2}+\W'_\Delta\right)\mathcal{V}_{\Delta}^{-1}\gamma^+u^0,a^{-1}u^0\right\rangle_{ {\Omega}}
+\left\langle -a\left(\frac{1}{2}-\W'_\Delta\right)\mathcal{V}_{\Delta}^{-1}\gamma^+u^0,\gamma^+(a^{-1}u^0) \right\rangle_{ \partial\Omega}\\
=-\left\langle \left(\frac{1}{2}+\W'_\Delta\right)\mathcal{V}_{\Delta}^{-1}\gamma^+u^0
+ \left(\frac{1}{2}-\W'_\Delta\right)\mathcal{V}_{\Delta}^{-1}\gamma^+u^0,\gamma^+u^0 \right\rangle_{ \partial\Omega}
=-\left\langle\mathcal{V}_{\Delta}^{-1}\gamma^+u^0,\gamma^+u^0\right\rangle_{ \partial\Omega}\\
\le -C\|\gamma^+u^0\|_{H^\ha(\pO)}^2
\le -C\|\gamma^+u^0\|_{L_2(\pO)}^2=-C|\pO|^2<0.\label{g0U0A2}
\end{multline}

On the other hand $g^0(\mathcal U^0)=1.$
Hence Theorem~\ref{L1} 
extracted from \cite{MikPert1999}, implies the theorem claims. 
\end{proof}

\section{Auxiliary assertions}\label{Appendix}
We provide below some auxiliary results used in the main text.

\begin{theorem}\label{Tseq}
Let $\ha<s<\tha$, $u\in {H}^{s}( {\Omega})$,  $a\in C_+^{\sigma}(\overline \Omega)$, $\sigma=\max\{1, s\}$, $Au=r_ {\Omega}\tilde f$ in an interior or exterior Lipschitz domain $ {\Omega}$ for some $\tilde f\in
\widetilde{H}^{s-2}( {\Omega})$. Let  $\{f_k\}\in \widetilde H^{-\ha}_\bullet( {\Omega})$ be a sequence such that 
$\|\tilde f-\mathring E_{\Omega}f_k\|_{\widetilde H^{s-2}( {\Omega})}\to 0$ as $k\to\infty$. 

Then there exists a sequence  $\{u_k\}\in H^{s,0}( {\Omega};A)$ such that 
$Au_k=f_k$ in $ {\Omega}$ and $\|u-u_k\|_{H^{s}( {\Omega})}\to 0$  as $k\to \infty$. Moreover, $\|T^+(u_k)-T^+(\tilde{f};u)\|_{H^{s-\tha}( \partial\Omega)}\to 0$   as $k\to \infty$.
\end{theorem}
\begin{proof}
Let us consider the Dirichlet problem
\begin{eqnarray}
\label{2.6t} && A\,u_k= f_k  \;\;\; \mbox{\rm in}\;\;\;\;  {\Omega},
\\ 
 \label{2.7t} &&  \gamma^+u_k=\gamma^+u  \;\;\; \mbox{\rm on}\;\;\;\; \pO,
\end{eqnarray} 
By Corollary~\ref{Cor5.5},  the unique solution of problem \eqref{2.6t}-\eqref{2.7t} in $H^{s}( {\Omega})$ is  
$u_k=(\mathcal A^D)^{-1}(f_k,\varphi_k)^\top$, where 
$(\mathcal A^D)^{-1}:H^{s-2}( {\Omega})\times H^{s-\ha}(\pO)\to H^{s}( {\Omega})$ is a continuous operator. 
Hence the functions $u_k$ converge to $u$ in $H^{s}( {\Omega})$ as $k\to\infty$. 
Since $\tilde A u_k=\mathring E_{\Omega} f_k\in\widetilde H^{-\ha}( {\Omega})$, we obtain that $u_k\in H^{s,-\ha}( {\Omega};A)$ and the canonical conormal derivative, $T^+u_k$, is  well defined. 
Then subtracting \eqref{Tcandef} for $u_k$ from \eqref{Tgend}, we obtain,
$$
T^+(\tilde f;u)-T^+u_k=
 (\gamma^{-1})^*[\tilde f-\mathring E_{\Omega}f_k + \check A_{\Omega}(u-u_k)].
$$
Hence 
\be \label{T-Tk}
\|T^+(\tilde f;u)-T^+u_k\|_{H^{s-\tha}( \partial\Omega)}\le 
 C\left(\|\tilde f-\mathring E_{\Omega}f_k\|_{\widetilde H^{s-2}( {\Omega})} 
 + C_1\|u-u_k\|_{H^{s}( {\Omega})}\right)
\ee
for some positive $C$ and $C_1$.
Since the right hand side  of \eqref{T-Tk} tends to zero as $k\to\infty$, so does also the left hand side.
\end{proof}

Note that since $\mathcal D( {\Omega})\subset \widetilde H^{-\ha}_\bullet( {\Omega})$ is dense in $\widetilde{H}^{s-2}( {\Omega})$, the sequence $\{f_k\}\in \widetilde H^{-\ha}_\bullet( {\Omega})$ from Theorem~\ref{Tseq} hypotheses  does always exist.

{
The following multiplication theorem is well know, see e.g., \cite[Theorems
1.4.1.1, 1.4.1.2]{Grisvard1985}, \cite[Theorem 2(b)]{Zolesio1977}, \cite[Theorems 1.9.1, 1.9.2, 1.9.5]{Agranovich2015}, \cite[Theorem 3.2]{MikJMAA2013}. 
\begin{theorem}\label{GrL}
Let $\Omega_0$ be an open set.

(i) If $g\in L_\infty({\Omega_0})$, then $gv\in L_2(\Omega_0)$
for every $v\in L_2(\Omega_0)$  and
  $
 \|gv\|_{L_2( {\Omega})}\le c \|g\|_{L_\infty({\Omega_0})} \|v\|_{L_2(\Omega_0)}.
$ 

(ii) If $\sigma$ is a non-zero integer and  $g\in {C}^{|\sigma|-1,1}(\overline{\Omega_0})$, then $gv\in H^\sigma(\Omega_0)$
for every $v\in H^\sigma(\Omega_0)$ and
  $
 \|gv\|_{H^\sigma( {\Omega})}\le c \|g\|_{{C}^{|\sigma|-1,1}(\overline{\Omega_0})} \|v\|_{H^\sigma(\Omega_0)}.
  $

(iii)  If $\sigma$ is a non-integer, $|\sigma|=m+\theta$, where $m$ is a non-negative integer and $0<\theta<1$, while  $g\in {C}^{m,\eta}(\overline{\Omega_0})$ with $\theta<\eta<1$,
 then $gv\in H^\sigma(\Omega_0)$
for every $v\in H^\sigma(\Omega_0)$ and
  $
 \|gv\|_{H^\sigma( {\Omega})}\le c \|g\|_{{C}^{m,\eta}(\overline{\Omega_0})} \|v\|_{H^\sigma(\Omega_0)}.
  $

In all cases $c$ is a positive constant independent of $g$, $v$ or $\Omega_0$.
\end{theorem}


\
\begin{theorem}\label{T3.6}
Let $\Omega$ be a bounded simply-connected Lipschitz domain and  $0\le \sigma\le 1$.
The operators
\begin{align}
\label{3.15}
{\cal V}_\Delta  &: H^{\sigma-1}(\pO)\to H^{\sigma}(\pO), \\
\label{3.15W}
-\ha I + {\cal W}_\Delta &: H^{\sigma}(\pO)\to H^{\sigma}(\pO),\\
\label{3.15W'}
 -\ha I + {\cal W}'_\Delta &: H^{-\sigma}(\pO)\to H^{-\sigma}(\pO)
\end{align}
 are isomorphisms, while the operators
\begin{align}
\label{3.15W+}
\ha I + {\cal W}_\Delta &: H^{\sigma}(\pO)\to H^{\sigma}(\pO),\\
\label{3.15W'+}
 \ha I + {\cal W}'_\Delta &: H^{-\sigma}(\pO)\to H^{-\sigma}(\pO),\\
 \label{LFredh}
\mathcal L_\Delta&: H^{\sigma}(\pO)\to H^{\sigma-1}(\pO)
\end{align}
are Fredholm with zero index. 
\begin{proof}
Invertibility of the boundary integral operators \eqref{3.15}-\eqref{3.15W'} related with the harmonic layer potential is well known, cf. e.g., \cite{Verchota1984}, \cite[Theorem 4.1]{Mitrea_D1997}, \cite[Theorem 8.1]{FMM1998}.
The Fredholm property of operator \eqref{LFredh} for $\sigma=\ha$ is also well known, see, e.g. \cite[Theorem 7.8]{McLean2000}. Then the corresponding result for $0\le \sigma\le 1$ can be proved as in \cite[Theorem 7.17]{McLean2000} but using a sharper regularity result from \cite[Theorem 3]{Costabel1988}.
\end{proof}

\end{theorem}
%


Theorem~\ref{L1} below is implied by \cite[Lemma 2]{MikPert1999}  (see also  \cite[\S 21]{Vain-Tren1974}, \cite[Section 21.4]{Trenogin1980}, where the particular
case, 
$ h^* _i(\stackrel{\circ }{x}_j)=\mathring x^*_{i}( h _j)= \delta _{ij}$, 
has been considered). Another approach, although with hypotheses similar to the ones in Theorem~\ref{L1}, is presented in \cite[Lemma 4.8.24]{Hackbush1995}.

\begin{theorem}\label{L1} 
Let $B_1$ and $B_2$ be two Banach spaces.
Let
$\underline{A}:B_1\rightarrow B_2$ be a linear Fredholm operator with zero index,
$\underline{A}^*:B_2^*\rightarrow B_1^*$ be the operator adjoined to it,
and
$\dim \ker \underline{A}=\dim \ker\underline{A}^*=n<\infty$,  
where
$\ker\underline{A}=\mathrm{span}\{\mathring{x}_i\}_{i=1}^n$ $\subset B_1$,
$\ker\underline{A}^*=\mathrm{span}\{\mathring x^*_{i}\}_{i=1}^n\subset B_2^*$. 
Let 
$$
\underline{A}_1x:=\sum_{i=1}^k h _i h^* _i(x),
$$
where
$ h^* _i$, $ h _i \, (i=1,..., n)$
are  elements from $B^*_1$ and $B_2$, respectively, such that
\begin{equation}\label{2.3}
\det [ h^* _i
(\mathring{x}_j)]\not= 0,
\qquad
\det [\mathring x^*_{i}(h _j)]\not= 0 \qquad i,j=1,..., n.
\end{equation}
Then:

(i) the operator $\underline{A}-\underline{A}_1:B_1\rightarrow B_2$ is continuous and continuously invertible;

(ii) if $y\in  B_2$  satisfies the solvability  conditions, 
\begin{equation}
\label{1.4}
\mathring x^*_{i}(y)=0, \qquad i=1,..., n,
\end{equation}
of equation 
\begin{equation}
\label{1.1C}
\underline{A}x=y,
\end{equation} 
then the unique solution $x$ of equation 
\begin{equation}\label{2.1C}
(\underline{A}-\underline{A}_1)x=y, 
\end{equation}  is a
solution of equation (\ref{1.1C}) such that\begin{equation}\label{2.4} h^*
_i(x) = 0\qquad (i=1,..., k).\end{equation}

(iii) Vice versa, if $x$ is a solution of equation (\ref{2.1C}) satisfying
conditions (\ref{2.4}), then conditions (\ref{1.4})
are satisfied for the right-hand side $y$ of equation (\ref{2.1C})
and $x$ is a solution of equation (\ref{1.1C}) with the same right-hand
side $y$.
\end{theorem}

Note that more results about finite-dimensional operator perturbations are available in \cite{MikPert1999}.

The following known result (cf., e.g., \cite[Lemma 11.9.21]{Mitrea-Wright2012}) is useful for us.
\begin{lemma}
\label{Lem2-Fredholm}
Let $ X_{1}, X_{2} $ and $ Y_{1}, Y_{2} $, be Banach spaces such that the embeddings $ X_{1} \hookrightarrow X_{2} $ and $ Y_{1} \hookrightarrow Y_{2} $ are continuous, and the embedding $ Y_{1} \hookrightarrow Y_{2} $ has dense range. Assume that $T:X_{1}\rightarrow Y_{1}$  and  $T : X_{2} \rightarrow Y_{2} $ are Fredholm operators with the same index, ${\rm{ind}}\left(T : X_{1}\rightarrow Y_{1}\right)= {\rm{ind}}\left(T : X_{2} \rightarrow Y_{2}\right)$. Then
${\rm{Ker}}\{T : X_{1}\rightarrow Y_{1}\}={\rm{Ker}}\{T : X_{2} \rightarrow Y_{2}\}.$
\end{lemma}

\section*{Concluding remarks}

The Dirichlet and Neumann problems on a bounded Lipschitz domain for a variable--coefficient second order PDE with  general
right-hand side functions from ${H}^{s-2}( {\Omega})$ and $\widetilde{H}^{s-2}( {\Omega})$, $\ha<s<\tha$, respectively, were equivalently reduced to three direct segregated
boundary-domain integral equation systems, for each of the BVPs. 
This involved systematic use of the generalised co-normal derivatives.
The operators associated with
the left-hand sides of all the BDIE systems were analysed  in corresponding
Sobolev spaces. It was shown that the operators of the BDIE systems for the Dirichlet problem are continuous and continuously invertible. For the Neumann problem the BDIE system operators are continuous but only Fredholm with zero index,  their kernels and co-kernels were analysed, and appropriate finite-dimensional perturbations were constructed to make the perturbed (stable) operators invertible and provide a solution of the original BDIE systems and the Neumann problem. 

The same approach can be used to extend, to the general PDE right hand sides, the BDIE systems for the mixed problems, 
unbounded domains, BDIEs of more general scalar PDEs and the systems of PDEs, as well as to the united and localised BDIEs, for which the analysis is now available for the right hand sides only from $L^2( {\Omega})$, with smooth coefficient and domain boundary, see \cite{CMN-1}%
\nocite{CMN-LocJIEA2009, CMN-NMPDE-crack, CMN-MDEMP2011, CMN-Ext-AA2013, CMN-IEOT2013}%
--\cite{CMN-SysDir2017}, \cite{MikMMAS2006}, \cite{AyeleMik-EMJ2011}, \cite{DufMikIMSE2015}, \cite{MikPorIMSE2015}, \cite{MikPorUKBIM2015}.
\\

\noindent
{\bf Acknowledgement}\\
This research was supported by the  grants
EP/H020497/1: "Mathematical Analysis of Localized Boundary-Domain Integral Equations
 for Variable-Coefficient Boundary Value Problems" 
and
EP/M013545/1: "Mathematical Analysis of Boundary-Domain Integral Equations for Nonlinear PDEs"  
 from the EPSRC, UK.



\end{document}